\documentclass[11pt,oneside,english,jou]{amsart}
\usepackage[T1]{fontenc}
\usepackage[latin9]{inputenc}
\usepackage{babel}
\usepackage{verbatim}
\usepackage{mathrsfs}
\usepackage{amstext}
\usepackage{amsthm}
\usepackage{amssymb}
\usepackage{amsfonts}
\usepackage{amsmath}
\usepackage{esint}
\usepackage{enumerate}
\usepackage[unicode=true,pdfusetitle,
bookmarks=true,bookmarksnumbered=false,bookmarksopen=false,
breaklinks=false,pdfborder={0 0 1},colorlinks=false]
{hyperref}
\usepackage[margin=2.5cm]{geometry}
\usepackage[foot]{amsaddr}
\usepackage{mathtools}
\usepackage{ dsfont }
\usepackage[shortlabels]{enumitem}
\usepackage{bm}

\makeatletter
\numberwithin{equation}{section}
\numberwithin{figure}{section}
\theoremstyle{plain}
\newtheorem{thm}{\protect\theoremname}[section]
\theoremstyle{definition}
\newtheorem{rem}[thm]{\protect\remarkname}
\theoremstyle{definition}

\theoremstyle{plain}
\newtheorem{prop}[thm]{\protect\propositionname}
\theoremstyle{plain}
\newtheorem{lem}[thm]{\protect\lemmaname}
\theoremstyle{plain}

\theoremstyle{plain}
\newtheorem{cor}[thm]{\protect\corollaryname}

\theoremstyle{definition}

\theoremstyle{definition}

\theoremstyle{definition}
\newtheorem{examplex}[thm]{\protect\examplename}
\theoremstyle{definition}

\newtheorem{definition}{Definition}
\newtheorem*{claim}{Claim}

\newenvironment{example}
{\pushQED{\qed}\examplex}
{\popQED\endexamplex}

\usepackage{tikz}
\usetikzlibrary{shapes,arrows}
\usepackage{verbatim}
\usepackage{amsthm}
\usepackage{amstext}

\DeclareMathOperator{\diam}{diam}

\DeclareMathOperator{\Leb}{Leb}

\DeclareMathOperator{\supp}{supp}

\newcommand{\ind}{\mathds{1}}

\newcommand{\R}{\mathbb R}
\newcommand{\Z}{\mathbb Z}
\newcommand{\N}{\mathbb N}

\newcommand{\ii}{{\underline{i}}}
\newcommand{\jj}{{\bf j}}

\renewcommand{\r}{\color{red}}

\newcommand{\eps}{\varepsilon}

\newcommand{\tc}{\tilde{c}}

\newcommand{\hdim}{\dim_H}

\makeatother

\providecommand{\conjecturename}{Conjecture}
\providecommand{\corollaryname}{Corollary}
\providecommand{\definitionname}{Definition}
\providecommand{\examplename}{Example}
\providecommand{\lemmaname}{Lemma}
\providecommand{\problemname}{Problem}
\providecommand{\propositionname}{Proposition}
\providecommand{\remarkname}{Remark}
\providecommand{\theoremname}{Theorem}
\providecommand{\taskname}{Task}

\def\diam{{\rm diam}}

\def\supp{{\rm supp}}
\def\vphi{\varphi}

\def\half{\frac{1}{2}}

\newcommand{\lam}{\lambda}

\def\Lam{\Lambda}
\newcommand{\gam}{\gamma}
\newcommand{\om}{\omega}

\def\Om{\Omega}

\newcommand{\sig}{\sigma}

\def\N{{\mathbb N}}

\def\pb{{\mathbf p}}

\def\Ak{{\mathcal A}}
\def\Bk{{\mathcal B}}

\def\Jk{{\mathcal J}}
\def\Lk{{\mathcal L}}

\def\be{\begin{equation}}
	\def\ee{\end{equation}}

\newcommand{\Ek}{{\mathcal E}}
\newcommand{\Fk}{{\mathcal F}}

\newcommand{\es}{\emptyset}

\def\ov{\overline}

\newcommand{\const}{{\rm const}}

\def\what{\widehat}
\def\wt{\widetilde}

\renewcommand{\r}{\color{red}}

\begin{document}
\title{Typical absolute continuity for classes of dynamically defined measures}

\author{Bal\'azs B\'ar\'any$^1$}
\address{$^1$Budapest University of Technology and Economics, Department of Stochastics, MTA-BME Stochastics Research Group, P.O.Box 91, 1521 Budapest, Hungary}
\email{balubsheep@gmail.com}

\author{K\'aroly Simon$^{1,2}$}
\address{$^2$Alfr\'ed R\'enyi Institute of Mathematics -- E\"otv\"os Lor\'and Research Network, Re\'altanoda u. 13-15., 1053 Budapest, Hungary}
\email{simonk@math.bme.hu}

\author{Boris Solomyak$^3$}
\address{$^3$Bar-Ilan University, Department of Mathematics, Ramat Gan, 5290002 Israel}
\email{bsolom3@gmail.com}

\author{Adam \'Spiewak$^3$}
\email{ad.spiewak@gmail.com}

\date{\today}

\begin{abstract} 
We consider one-parameter families of smooth uniformly contractive iterated function systems $\{f^\lambda_j\}$ on the real line.
Given a family of parameter dependent measures $\{\mu_{\lambda}\}$ on the symbolic space, we study geometric and dimensional properties of their images under the natural projection maps $\Pi^\lambda$. 
The main novelty of our work is that the measures $\mu_\lambda$ depend on the parameter, whereas up till now it has been usually assumed that the measure on the symbolic space is fixed and the parameter dependence comes only from the natural projection. This is especially the case in the question of absolute continuity of the projected measure $(\Pi^\lambda)_*\mu_\lambda$, where we had to develop a new approach in place of earlier attempt which contains an error.
Our main result states that if $\mu_\lambda$ are Gibbs measures for a family of H\"older continuous potentials $\phi^\lambda$, with H\"older continuous dependence on $\lambda$ and $\{\Pi^\lambda\}$ satisfy the transversality condition, then the projected measure $(\Pi^\lambda)_*\mu_\lambda$ is absolutely continuous for Lebesgue a.e.\ $\lambda$, such that the ratio of entropy over the Lyapunov exponent is strictly greater than $1$. We deduce it from a more general almost sure lower bound on the Sobolev dimension for families of measures with regular enough dependence on the parameter. Under less restrictive assumptions, we also obtain an almost sure formula for the Hausdorff dimension. As applications of our results, we study stationary measures for iterated function systems with place-dependent probabilities (place-dependent Bernoulli convolutions and the Blackwell measure for binary channel) and equilibrium measures for hyperbolic IFS with overlaps (in particular: natural measures for non-homogeneous self-similar IFS and certain systems corresponding to random continued fractions).

\end{abstract}



\maketitle
\tableofcontents


\section{Introduction}

Let $\Ak = \{1,\ldots,m\}$ and let $\Psi=\{f_j\}_{j \in \Ak}$ be a set of contracting smooth functions on a compact interval $I\subset \R$ mapping $I$ into itself. We call the set $\Psi$ an \textit{iterated function system} (IFS) on $I$. It is well known that there exists a unique non-empty compact set $\Lambda\subseteq I$ such that it is invariant with respect to 
the IFS, that is $\Lambda=\bigcup_{j\in\Ak}f_j\left(\Lambda\right)$. We call the set $\Lambda$ the \textit{attractor} of the IFS, see Hutchinson \cite{H} or Falconer \cite{Falconer}. 

Moreover, let $\Omega=\Ak^\N$ be the symbolic space and  $\sigma$ the left shift transformation on $\Omega$. There is a natural projection $\Pi\colon\Omega\mapsto\Lambda$ defined as
\begin{equation*}\label{enatproj}
\Pi(\omega):=\lim_{n\rightarrow\infty}f_{\omega_1}\circ\cdots\circ f_{\omega_n}(x),\text{ for $\omega=(\omega_1, \omega_2, \ldots)\in\Omega$,}
\end{equation*}
where $x \in I$ is any point (the limit does not depend on the choice of $x$). If $\mu$ is a probability measure on $\Omega$ then we call the measure $\Pi_*\mu=\mu\circ\Pi^{-1}$ on $\Lambda$ the {\em push-forward measure} of $\mu$. Usually, we assume that $\mu$ is $\sigma$-invariant and ergodic. Let us denote the entropy of $\mu$ by $h_\mu$ and the Lyapunov exponent by $\chi_\mu$. The ratio $h_\mu / \chi_\mu$ is called the {\em Lyapunov dimension} of $\mu$.

Considerable attention has been paid to the dimension theory and measure theoretic properties of attractors and  push-forward measures of iterated function systems. A natural upper bound for the Hausdorff and box counting dimension of the attractor is the unique root $s$ of the pressure function $s\mapsto P(-s\log|f'_{\om_1}(\Pi(\sig \om))|)=0$, see the next section for definitions.
Ruelle \cite{Ru} showed that in case of separation, e.g., the Open Set Condition (OSC), the Hausdorff dimension of the attractor equals to the root of the pressure function, 
see also Falconer \cite{falcbook3}. Similarly, the Hausdorff dimension of the push-forward measures is bounded above by the Lyapunov dimension of $\mu$; moreover, if the OSC holds, then the dimension equals to the Lyapunov dimension of $\mu$, see Feng and Hu \cite{FH}.

The situation becomes more complicated if there are overlaps between the maps. To handle this case, Pollicott and Simon \cite{PS} introduced the transversality method for parametrized families of iterated function systems. Later, this method was widely applied and generalised, see for example, Solomyak~\cite{Sol,So}, Peres and Solomyak~\cite{peso,peso1}, Simon and Solomyak~\cite{SiSo}, Neunh\"auserer \cite{Neun}, Ngai and Wang \cite{NW}, and Peres and Schlag \cite{PS00}. 

We have a deeper understanding in the special case, when the maps of the IFS are similarities and the measure $\mu$ is Bernoulli, thanks  to recent results. In his seminal paper, Hochman \cite{Hoch}, using methods of additive combinatorics, determined the value of the Hausdorff dimension of the attractor (self-similar set) and the push-forward measure (self-similar measure) under the {\em exponential separation condition}. Relying on this result and the Fourier decay of the push-forward measure, Shmerkin \cite{Shm} proved that the exceptional set of parameters for absolute continuity of Bernoulli convolution measures has zero Hausdorff dimension. These results were extended by Shmerkin and Solomyak \cite{SS} and Saglietti, Shmerkin and Solomyak \cite{SSS} to more
general IFS of similarities and Bernoulli measures. 
Further progress on absolute continuity of Bernoulli convolutions was obtained by Varj\'u \cite{Varju}.
Jordan and Rapaport \cite{JR} showed that the dimension of the push-forward measure of any ergodic shift-invariant measure equals to the entropy over Lyapunov exponent ratio under the exponential separation condition. However, such strong results are unknown in the case when the IFS consists of general conformal maps.

Simon, Solomyak and Urba\'nski \cite{SSU1, SSU Parabolic} showed that if a smoothly parametrized (hyperbolic or parabolic) family of conformal IFS's $\{f_i^\lambda\}_{i\in\Ak}$ satisfies the transversality condition over a bounded open domain $U$ of parameters, then for Lebesgue almost every parameter $\lambda\in U$ the dimension of the attractor equals to $\min\{1,s_\lambda\}$, where $s_\lambda$ is the root of the pressure function, which depends on the parameter.
Moreover, it has positive Lebesgue measure for almost every parameter, such that $s_\lam>1$. Similarly, the dimension of the push-forward measure of any fixed ergodic shift-invariant measure $\mu$ is equal to the Lyapunov dimension of $\mu$, and the measure is absolutely continuous for almost every parameter  where $h_\mu/\chi_\mu>1$.
Peres and Schlag \cite{PS00} obtained upper bounds on the Hausdorff dimension of the set of exceptional parameters using a version of transversality, in the framework of a ``generalized projection''. All these results required a fixed ergodic shift-invariant measure on $\Omega$. However, there are important cases when the measure on $\Omega$ depends also on the parameter $\lambda$. There are two natural occurrences of such situation.

One is the so-called place-dependent measures, which were studied  by Fan and Lau \cite{FL}, Hu, Lau and Wang \cite{HLW}, 
Jaroszewska \cite{J}, Jaroszewska and Rams \cite{JR}, Kwieci\'nska and W. S\l omczy\'nski \cite{KS}, Czudek \cite{Cz} and others. Let $\{p_i\}_{i\in\Ak}$ be a family of H\"older continuous maps $p_i\colon I\mapsto[0,1]$ such that $\sum_{i\in\Ak}p_i\equiv1$. Fan and Lau \cite{FL} showed that there exists a unique measure $\nu$ on $I$ such that
$$
\int \varphi(x)d\nu(x)=\int\sum_{i\in\Ak}p_i(x)\varphi(f_i(x))d\nu(x)\ \ \text{ for any continuous test function }\varphi.
$$
In view of a result by Bowen \cite{Bow}, it is clear that 
$\nu$ is the push-forward of the equilibrium measure $\mu$ (on the symbolic space $\Ak^\N$) of the pressure corresponding to the potential $\omega\mapsto\log p_{\omega_1}(\Pi(\sigma\omega))$.
It is shown in \cite{FL} that if the open set condition holds, then the dimension of $\nu$ equals $\frac{h_\mu}{\chi_\mu}$.  In  the case of parametrized family $\{f_i^\lambda\}_{i\in\Ak}$ the equilibrium measure depends on the parameter. 

B\'ar\'any \cite{BB15} studied such parametrized place-dependent families and claimed to generalise the result of \cite{SSU Parabolic} for this case. However, the proof contains a crucial
error, which cannot be fixed easily.  In the present paper we have managed to overcome the obstacles and correct the error, using a delicate modification of the Peres-Schlag \cite{PS00}
	method. In fact, our results are much more general. Here we state the main result in the most important situation, in non-technical terms; complete statements may be found in Section 3.
	
	\begin{thm}
		Let $\{f_j^\lam\}_{j\in \Ak}$ be a  $C^{2+\delta}$ smooth family of hyperbolic IFS on a compact interval, smoothly depending on a real parameter $\lam\in U$, and let $\Pi^\lam:\Om\to \R$ be the corresponding natural projection map. We assume that the (classical) transversality condition holds on $U$. 
		Let $\{\mu_\lam\}_{\lam\in U}$ be a family of Gibbs measures, corresponding to a family of H\"older-continuous potentials, with a H\"older-continuous
		dependence on parameter. Then the push-forward measure $(\Pi^\lam)_*\mu_\lam$ is absolutely continuous for Lebesgue-a.e.\ $\lam\in U$ such that $h_{\mu_\lam}/\chi_{\mu_\lam}>1$.
	\end{thm}

	We also showed, under slightly less restrictive assumptions, that the push-forward measure $(\Pi^\lam)_*\mu_\lam$ has Hausdorff dimension equal to $\min\{1,  h_{\mu_\lam}/\chi_{\mu_\lam}\}$
	almost everywhere in $U$. The proof of this result is not as difficult,  similar to B\'ar\'any-Rams \cite{BR18}, and is included for completeness.

Place-dependent measures play an important role, for example, in the theory of hidden Markov chains. Blackwell \cite{Bla} expressed the entropy of hidden Markov chains over finite state space as an integral with respect to a place-dependent measure, which is nowadays called the {\em Blackwell measure}. The singularity of the Blackwell measure was studied by B\'ar\'any, Pollicott and Simon \cite{BPS}. Later, B\'ar\'any and Kolossv\'ary \cite{BK} showed that the transversality condition holds on a certain region of parameters and applied the main theorem of B\'ar\'any \cite{BB15} to claim
absolute continuity almost everywhere in this region. Since the Blackwell measure satisfies the assumptions of the main result of the present paper, we recover this 
result of B\'ar\'any and Kolossv\'ary \cite{BK}.

Another important case, when the parameter dependence of the measure occurs, is the natural measure of the parametrized IFS $\{f_i^\lambda\}_{i\in\Ak}$, which is the equilibrium measure $\nu_\lambda$ with respect to the potential $\omega\mapsto s_\lambda\log|(f_{\omega_1}^{\lambda})'(\Pi_\lambda(\sigma\omega))|$. See \cite{PU} for more on the subject. In case of overlaps, neither the dimension nor the absolute continuity was known. Our result applies in this situation as well. In particular, it follows that a natural measure for non-homogeneous self-similar IFS is absolutely continuous for almost every parameter with similarity dimension strictly larger than $1$, in the transversality region (such regions were found e.g. for non-homogeneous Bernoulli convolutions, see \cite{Neun,NW}). A similar conclusion is obtained for a (non-linear) system corresponding to certain random continued fractions.

\subsection{About the proof}
In order to prove ``almost-sure'' results for push-forwards of measures $\mu_\lam$ depending on parameter, we need to impose ``correct'' continuity assumptions on the measure, which are, on one hand, sufficiently strong to apply the techniques, but on the other hand,  can be verified in practice. These continuity assumptions are imposed on measures of cylinder sets
and involve estimates of the ratios $\mu_{\lam}([w])/\mu_{\lam_0}([w])$ for $\lam$ close to $\lam_0$. For the result on Hausdorff dimension of the push-forward measure, the condition is less restrictive, see \ref{as:measure_cont} below, and we could apply more or less ``classical'' transversality techniques, since roughly speaking, we can ``afford'' to lose  $\epsilon$ in dimension estimates.

The results on absolute continuity are much more delicate. The idea is to adapt the method of Peres-Schlag \cite{PS00} and to show that almost everywhere in the super-critical parameter interval, the
Sobolev dimension of the push-forward measure is greater than one. This implies not just absolute continuity, but also $L^2$-density and even existence of $L^2$-fractional derivatives of some positive order. This adaptation is not straightforward. First, \cite{PS00} uses the notion of transversality  of degree $\beta$, which has to be verified in our situation. Second, we cannot apply the result of
\cite{PS00} as a ``black box'', but rather have to work at a certain ``discretized'' level, in order to utilize the continuity assumptions on the measure dependence, see \ref{as:measure} below.
It should be mentioned that Peres-Schlag \cite{PS00} used their theorem on Sobolev dimension to estimate the Hausdorff dimension of the set of exceptional parameters for absolute continuity.
We do not deal with this issue and only concern ourselves with almost sure absolute continuity. We should also point out that \cite{PS00} contains two kinds of results: the infinite regularity case and the limited regularity case. It is the latter one (in fact, with the lowest possible regularity) that we adapt.

Another issue that comes up is that absolute continuity by the Peres-Schlag method is originally
shown under the assumption that the {\em correlation dimension} of the measure $\mu_\lam$ is greater than
one (in the metric corresponding to $\lam$), which is a stronger condition, in general, than $h_{\mu_\lam}/\chi_{\mu_\lam}>1$. The usual approach to overcome this is to restrict the measure to a 
``Egorov set'', where the convergence in the definitions of the entropy  and the Lyapunov exponent is uniform. This works fine when we consider the push-forward of a fixed measure, but in our case this is more delicate, since we have to guarantee that \ref{as:measure} is preserved under the restriction. Here we manage to overcome the obstacle with the help of large deviations estimates for Gibbs measures (see \cite{Young90, DK, OP}). 

\subsection{Organization of the paper} In the next section we collect all the main assumptions, definitions and notation. In Section 3 we state our main results. In fact, we state two results on almost sure absolute continuity: in the first one we don't make the assumption that $\mu_\lam$ is a family of Gibbs measure and only assume what is needed to prove almost sure absolute continuity  in the parameter
interval where the correlation dimension is greater than one. The second one is the sharp result for Gibbs measures. Section 4 is devoted to preliminaries, mainly the regularity properties of the IFS and the parameter dependence. Shorter proofs are included in this section, but longer and more technical calculations are postponed to the Appendices. In Section 5 we prove the theorem on the 
Hausdorff dimension of the push-forward measures. In Section 6 we verify that the transversality of degree $\beta$ condition of Peres-Schlag holds under our ``standard'' transversality assumptions, given sufficient regularity. The ``heart'' of the proof, namely, the adaptation of a discretized Peres-Schlag method, where transversality condition is used, is contained in Section 7. Section 8 is devoted to the
case of Gibbs measures: first we show that under the continuity assumptions on the potential, the Gibbs measures satisfy \ref{as:measure}, and then use large deviation estimates to extract 
``large submeasures'' still satisfying \ref{as:measure}, but with correlation dimension arbitrary close to $h_{\mu_\lam} / \chi_{\mu_\lam}$. After that, it only remains to collect the pieces to complete the proof of the main results; this is done in Section 9. Section 10 is devoted to applications. There we also present a sufficient condition for transversality to hold for "vertical" translation families of the form $f^\lam_j(x) = f_j(x) + a_j(\lam)$. Last, but not least, Section \ref{sec:open} contains some open questions and possible directions for further research.

\subsection{Acknowledgements} 
Bal\'azs B\'ar\'any and K\'aroly Simon acknowledge support from grants OTKA~K123782 and OTKA~FK134251. Boris Solomyak and Adam \'Spiewak acknowledge support from the Israel Science Foundation, grant 911/19.

\section{Assumptions, notation and definitions}

Let $\Ak = \{1,\ldots,m\}$ and suppose we have an IFS $\{f_j^\lam\}_{j \in \Ak}$ on a compact interval $X\subset \R$, depending on a parameter $\lam\in \ov{U} \subset \R$ with $U$ being an open and bounded interval. Let $\diam(X)=1$ for simplicity. We assume that there exists  $\delta \in (0,1]$ such that
\begin{enumerate}[start=1,label={(A\arabic*)}]
\item\label{as:C2} the maps $f_j^\lam$ are $C^{2+\delta}$-smooth on $X$ with $M_1 = \sup \limits_{\lam \in U} \sup \limits_{j \in \Ak} \left\{ \left\| \frac{d^2}{dx^2}f^\lam_j \right\|_{\infty} \right\}  < \infty $ and there exist constants $C_1, C_2 > 0$ such that
\[ \left| \frac{d^2}{dx^2}f^\lam_j(x) - \frac{d^2}{dx^2}f^\lam_j(y) \right| \leq C_1|x-y|^\delta \text{ and } \left| \frac{d^2}{dx^2}f^{\lam_1}_j(x) - \frac{d^2}{dx^2}f^{\lam_2}_j(x) \right| \leq C_2|\lam_1-\lam_2|^\delta  \]
hold for all $x,y \in X,\ j \in \Ak,\ \lam, \lam_1, \lam_2 \in U$.

\bigskip

\item\label{as:lam hoelder} the maps $\lam \mapsto f^\lam_j(x)$ are $C^{1+\delta}$-smooth on $U$ and there exists a constant $C_3 > 0$ such that 
\[\left|\frac{d}{d\lam}f_j^{\lam_1}(x) - \frac{d}{d\lam}f_j^{\lam_2}(x)\right| \leq C_3 | \lam_1 - \lam_2|^\delta\] holds for all $x \in X,\ j \in \Ak,\ \lam_1, \lam_2 \in U$.

\bigskip

\item\label{as:dxdlam} the second partial derivatives $\frac{d^2}{dxd\lam}f^\lam_j(x), \frac{d^2}{d\lam dx}f^\lam_j(x)$ exist and are continuous on $U \times X$ (hence equal) with $M_2 = \sup \limits_{j \in \Ak} \sup \limits_{\lambda \in U} \left\| \frac{d^2}{d\lam dx} f^\lam_j(x)\right\|_{\infty} < \infty$ and there exist constants $C_4, C_5>0$ such that
\[ \left| \frac{d^2}{dxd\lam}f^{\lam}_j(x) - \frac{d^2}{dxd\lam}f^{\lam}_j(y) \right| \leq C_4 |x-y|^\delta \text{ and } \left| \frac{d^2}{dxd\lam}f^{\lam_1}_j(x) - \frac{d^2}{dxd\lam}f^{\lam_2}_j(x) \right| \leq C_5 |\lam_1 - \lam_2|^\delta \]
hold for all $x,y \in X,\ j \in \Ak,\ \lam, \lam_1, \lam_2 \in U$.

\bigskip

\item\label{as:hyper} the system $\{f_j^\lam\}_{j \in \Ak}$ is {\em uniformly hyperbolic and contractive}: there exists $\gamma_1,\ \gamma_2 > 0$ such that
\[
0 < \gam_1 \le |(\textstyle{\frac{d}{dx}}f_j^\lam)(x)| \le \gam_2 < 1\ \ \mbox{holds for all}\ j \in \Ak,\ x\in X,\ \lam\in \ov{U}.
\]
\end{enumerate}

\bigskip

\noindent Let $\Om = \Ak^\N$ and let $\sigma$ denote the left shift on $\Om$. Let $\Om^* = \bigcup \limits_{n \geq 0} \Ak^n$ be the set of finite words over $\Ak$ and let $|u|$ be the length of $u$. For $u = (u_1, \ldots u_n) \in \Om^*$ denote
\[ f^\lam_{u} = f^\lam_{u_1 \ldots u_n} := f^\lam_{u_1} \circ \ldots \circ f^\lam_{u_n} \]
(with $f_{u} = \mathrm{id}$ if $u$ is an empty word) and let $\Pi^\lam : \Om \to X,\ \lam \in \ov{U}$
$$
\Pi^\lam(u) = \lim_{n\to \infty} f^\lam_{u_1\ldots u_n}(x_0) \text{ for } u \in \Om
$$
be the natural projection (it does not depend on the choice of $x_0 \in X$). For $u \in \Om^* \cup \Om$ let $u|_n = (u_1, \ldots, u_n)$ denote the restriction of $u$ to the first $n$ coordinates. For $u = (u_1, \ldots, u_n) \in \Om^*$ and $0 \leq k \leq |u|$ let $\sigma^k u = (u_{k+1},\ldots,u_n)$. For $u, v \in \Om$ let $u \wedge v = (u_1, \ldots, u_n)$, where $n = \sup \{ k \geq 1: u_k = v_k  \}$, i.e. $u \wedge v$ is the common prefix of $u$ and $v$. For $u \in \Om^*$ let $[u] = \{ \om \in \Om : \om|_{|u|} = u \}$ be the cylinder corresponding to $u$.

We will assume that the following transversality condition is satisfied for $\lam \in U$:
\medskip
\begin{enumerate}[label={(T)}]
\item\label{as:trans} $\exists\,\eta>0:\ \forall\, u,v\in \Om,\ \ u_1 \ne v_1, \ \left|\Pi^\lam(u) - \Pi^\lam(v)\right| < \eta \implies \left|\textstyle{\frac{d}{d\lam}}(\Pi^\lam(u) - \Pi^\lam(v))\right| \ge \eta.$
\end{enumerate}

\medskip
In our setting, transversality condition \ref{as:trans} is equivalent to other transversality conditions appearing in the literature - see Section \ref{sec:tran} and Lemma \ref{lem-tran} for details.

Let $\left\{ \mu_\lam \right\}_{\lam \in \ov{U}}$ be a collection of finite Borel measures on $\Om$. We will consider two continuity assumptions on $\mu_{\lam}$:

\medskip

\begin{enumerate}[label={(M0)}]
	\item\label{as:measure_cont} for every $\lambda_0$ and every $\eps>0$ there exist $C,\xi>0$ such that
	\[
	C^{-1} e^{-\eps|\omega|} \mu_{\lambda_0}([\omega]) \leq\mu_{\lambda}([\omega])\leq C  e^{\eps|\omega|} \mu_{\lambda_0}([\omega])
	\]
	holds for every $\om \in \Om^*,\ |\om| \geq 1$ and $\lam \in \ov{U}$ with $|\lam - \lam_0|<\xi$;
\end{enumerate}

\medskip

\begin{enumerate}[label={(M)}]
\item\label{as:measure} there exists $c>0$ and  $\theta \in (0,1]$ such that for all $\omega\in\Om^*,\ |\om|\geq 1$, and all $\lambda,\lambda' \in \ov{U}$,
\[
	e^{-c|\lambda-\lambda'|^\theta|\omega|} \mu_{\lambda'}([\omega]) \leq\mu_{\lambda}([\omega])\leq e^{c|\lambda-\lambda'|^\theta|\omega|} \mu_{\lambda'}([\omega]).
\]
\end{enumerate}
\medskip
Note that \ref{as:measure} implies \ref{as:measure_cont}.

For a compact metric space $(X,d)$, let $\mathcal{M}(X)$ denote the set of finite Borel measures on $X$ and $\mathcal{P}(X)$ the set of Borel probability measures on $X$. For $\mu \in \mathcal{M}(X)$ and $\alpha>0$, define the $\alpha$-energy as
\begin{equation}\label{eq:energy} \Ek_{\alpha}(\mu, d) = \int \int d(x,y)^{-\alpha}d\mu(x)d\mu(y). \end{equation}
Define the {\em correlation dimension} of $\mu$ with respect to the metric $d$ as
\[ \dim_{cor}(\mu, d) = \sup \{ \alpha > 0 : \Ek_{\alpha}(\mu, d) < \infty \}. \]
For a Borel measure $\nu$ on $\R$, the Fourier transform of $\nu$ is given by
$
\hat{\nu}(\xi)=\int e^{i\xi x}d\nu(x).
$
For a finite Borel measure $\nu$ and $\gamma \in \R$, we define the homogenous Sobolev norm as
\[ \| \nu \|^2_{2,\gamma} = \int \limits_{\R}|\hat{\nu}(\xi)|^2|\xi|^{2\gamma}d\xi \]
and the {\em Sobolev dimension}
\[ \dim_S(\nu) = \sup \left\{ \alpha \in \R : \int \limits_\R |\hat{\nu}(\xi)|^2(1+|\xi|)^{\alpha - 1}d\xi < \infty \right\}. \]
Note that $0 \leq \dim_S(\nu) \leq \infty$ and
\[ \int \limits_\R |\hat{\nu}(\xi)|^2(1+|\xi|)^{\alpha - 1}d\xi < \infty\  { \iff}\  \int \limits_\R |\hat{\nu}(\xi)|^2|\xi|^{\alpha - 1}d\xi = \| \nu \|_{2,\frac{\alpha - 1}{2}}^2 < \infty\] for $\alpha > 0$ (see \cite[Section 5.2]{Mattila Fourier}). If $0 < \dim_S(\nu) < 1$, then $\dim_S(\nu) = \dim_{cor}(\nu)$, where the correlation dimension is taken with respect to the standard metric on $\R$. If $\dim_S(\nu) > 1$, then 
{ $\nu$ is absolutely continuous with a density (Radon-Nikodym derivative) in $L^2(\R)$},
and moreover $\nu$ has fractional derivatives in $L^2$ of some positive order -- see \cite[Theorem 5.4]{Mattila Fourier}

For an IFS $\{f_j^\lam\}_{j \in \Ak}$ and a family of shift-invariant and ergodic probability measure $\mu_{\lam}$ on $\Om$, let $h_{\mu_{\lam}}$ be the entropy of $\mu_{\lam}$ defined as
\[ h_{\mu_\lam} = - \lim \limits_{n \to \infty} \frac{1}{n} \sum \limits_{\om \in \Ak^n} \mu_\lam([\om])\log\mu_\lam([\om])  \]
and let $\chi_{\mu_{\lam}}$ be the Lyapunov exponent of $\mu_{\lam}$ given by
\[\chi_{\mu_{\lam}} = - \int \limits_{\Om} \log \left| \left(f^\lam_{\om_1}\right)' (\Pi^{\lam}(\sigma \om)) \right|d\mu_{\lam}(\om).\]

For $\lam \in \ov{U}$ we define a metric $d_\lam$ on $\Om$ by
\be\label{eq:metric_def_lam}
d_\lam(u,v) = \left|f_{u\wedge v}^{\lam} (X)\right| \text{ for } u,v \in \Om.
\ee

Let $\phi : \Omega \to \R$ be a continuous function on the symbolic space $\Om$. A shift-invariant ergodic probability measure $\mu$ on $\Om$ is called a {\em Gibbs measure of the potential} $\phi$ if there exists $P \in \R$ and $C_G\ge 1$ such that for every $\om \in \Om$ and $n \in \N$, holds the inequality
\[ C_G^{-1} \leq \frac{\mu([\om|_n])}{\exp(-Pn + \sum \limits_{k=0}^{n-1} \phi(\sigma^k \om))} \leq C_G. \]
It is known that if $\phi$ is H\"older continuous, then there exists a unique Gibbs measure of $\phi$ (see \cite{Bow}).


\section{Main results}

\begin{thm}\label{thm:main_hausdorff}
Let $\{f_j^\lam\}_{j \in \Ak}$ be a parametrized IFS satisfying smoothness assumptions \ref{as:C2} - \ref{as:hyper} and the transversality condition \ref{as:trans} on $U$. Let $\left\{ \mu_\lam \right\}_{\lam \in \ov{U}}$ be a collection of finite ergodic shift-invariant Borel measures on $\Om$ satisfying \ref{as:measure_cont}, such that $h_{\mu_{\lam}}$ and $\chi_{\mu_{\lam}}$ are continuous in $\lam$. Then equality
\[ \hdim((\Pi^\lam)_*\mu_\lam) = \min\left\{1, \frac{h_{\mu_\lam}}{\chi_{\mu_\lam}}\right\} \]
holds for Lebesgue almost every $\lam \in U$.
\end{thm}

The most general version of our main result is the following:

\begin{thm}\label{thm:main_cor_dim}
Let $\{f_j^\lam\}_{j \in \Ak}$ be a parametrized IFS satisfying smoothness assumptions \ref{as:C2} - \ref{as:hyper} and the transversality condition \ref{as:trans} on $U$. Let $\left\{ \mu_\lam \right\}_{\lam \in \ov{U}}$ be a collection of finite Borel measures on $\Om$ satisfying \ref{as:measure}. Then
\[ \dim_S((\Pi^\lam)_* \mu_\lam) \geq \min \left\{\dim_{cor}(\mu_\lam, d_\lam), 1 + \min\{\delta, \theta \} \right\}  \]
holds for Lebesgue almost every $\lam \in U$, where $d_\lam$ is the metric on $\Om$ defined in \eqref{eq:metric_def_lam} {and $\delta, \theta$ are from assumptions \ref{as:C2}-\ref{as:hyper} and \ref{as:measure} respectively.} Consequently, $(\Pi^\lam)_* \mu_\lam$ is absolutely continuous with a density in $L^2$ for Lebesgue almost every $\lam$ in the set $\{ \lam \in U : \dim_{cor}(\mu_\lam, d_\lam) > 1 \}$.
\end{thm}

In the special case of Gibbs measures for potentials with H\"older continuous dependence on the parameter, we get the following:

\begin{thm}\label{thm:main_gibbs}
Let $\{f_j^\lam\}_{j \in \Ak}$ be a parametrized IFS satisfying smoothness assumptions \ref{as:C2} - \ref{as:hyper} and the transversality condition \ref{as:trans} on $U$. Let $\left\{ \mu_\lam \right\}_{\lam \in \ov{U}}$ be a family of Gibbs measures on $\Om$ corresponding to a family of continuous potentials $\phi^\lambda\colon\Om\mapsto\R$ such that there exists $0<\alpha<1$ and $b>0$ with
\begin{equation}\label{eq:unifvar}
\sup_{\lambda\in \ov{U}}\mathrm{var}_k(\phi^\lambda)\leq b\alpha^k,
\end{equation}
where $\mathrm{var}_k(\phi)=\sup\{|\phi(\omega_1)-\phi(\omega_2)|:|\omega_1\wedge\omega_2|=k\}$. Moreover, suppose that there exist constants $c_0>0$ and $\theta > 0$ such that
\begin{equation}\label{eq:assong}
|\phi^{\lambda}(\omega)-\phi^{\lam'}(\omega)|\leq c_0|\lambda-\lam'|^\theta\text{ for every $\omega\in\Omega$ and $\lambda,\lam'\in \ov{U}$.}
\end{equation}
Then $\{ \mu_{\lambda} \}_{\lam \in \ov{U}}$ satisfies \ref{as:measure}, hence conclusions of Theorem \ref{thm:main_cor_dim} hold (with $\theta$ as in \eqref{eq:assong}). Furthermore, $(\Pi^\lam)_* \mu_\lam$ is absolutely continuous for Lebesgue almost every $\lam$ in the set $\{ \lam \in U : \frac{h_{\mu_\lam}}{\chi_{\mu_\lam}} > 1 \}$.
\end{thm}


\section{Preliminaries}
Throughout this section we assume that we are given an IFS $\{f_j^\lam\}_{j \in \Ak}$  satisfying \ref{as:C2} - \ref{as:hyper} for some $\delta \in (0,1]$. We state several auxiliary results concerning regularity properties of the IFS $\{f_j^\lam\}_{j \in \Ak}$  and the natural projection $\Pi^\lam$, which will be used in subsequent sections. As some of the proofs are lengthy, yet standard in techniques, we postpone them partially to the Appendix.

\begin{lem}\label{lem:d2 bounds}
	There exist constants $C_{51} >0$ and $C_{52} > 0$ such that
	\be\label{eq:d2dx bound}
	\left|\frac{d^2}{dx^2} f^\lam_{u}(x) \right| \leq C_{51} \left|\frac{d}{dx} f^\lam_{u}(x) \right|
	\ee
	and
	\be\label{eq:dxdlam bound}
	\left|\frac{d^2}{d\lam dx} f^\lam_{u}(x) \right| \leq C_{52} |u| \left|\frac{d}{dx} f^\lam_{u} (x) \right|
	\ee
	hold for all $\lam \in U,\ x \in X,\ u \in \Om^*$.
\end{lem}

\begin{proof} See Appendix \ref{app:proof of d2_bounds}.
\end{proof}

\begin{lem}[Parametric bounded distortion property]\label{lem:pbdp}
	There exist constants $c_{62}>0,\ C_{62}>1$ such that inequality
	\be \label{eq:bdp}
	\frac{1}{C_{62}}e^{-c_{62}|\lam_1 - \lam_2||u|} \leq \frac{\left| \frac{d}{dx} f^{\lam_1}_{u}(x) \right|}{\left| \frac{d}{dx} f^{\lam_2}_{u}(y) \right|} \leq C_{62}e^{c_{62}|\lam_1 - \lam_2||u|}
	\ee
	holds for all $\lam_1,\lam_2 \in \ov{U},\ x,y \in X,\ u \in \Om^*$.
\end{lem}

\begin{proof} First, let us prove the inequality with $\lambda_1 = \lambda_2$. For $u = (u_1, \ldots, u_n) \in \Om^*$, applying  \ref{as:C2} and  \ref{as:hyper}, together with inequality $\log \frac{x}{y} \leq \frac{\left| x - y \right|}{\min \{ x, y \}}$ for $x,y>0$ yields
	\begin{eqnarray}
	\log \frac{ \left|\frac{d}{dx} f^\lam_{u}(x) \right|}{\left| \frac{d}{dx} f^\lam_{u}(y) \right|} & = & \sum \limits_{k=1}^n \log\left| \frac{ \left(\frac{d}{dx} f^\lam_{u_k}\right)(f^\lam_{\sigma^k u}x) }{\left(\frac{d}{dx} f^\lam_{u_k}\right)(f^\lam_{\sigma^k u}y)}\right| \leq \sum \limits_{k=1}^n  \frac{\Big| \left|\left(\frac{d}{dx} f^\lam_{u_k}\right)(f^\lam_{\sigma^k u}x)\right| - \left|\left(\frac{d}{dx} f^\lam_{u_k}\right)(f^\lam_{\sigma^k u}y)\right|\Big| }{\min\left\{\left|\left(\frac{d}{dx} f^\lam_{u_k}\right)(f^\lam_{\sigma^k u}x)\right|,\left|\left(\frac{d}{dx} f^\lam_{u_k}\right)(f^\lam_{\sigma^k u}y)\right|\right\}} \nonumber\\
	& \leq & \frac{1}{\gamma_1} \sum \limits_{k=1}^n \left| \left(\frac{d}{dx} f^\lam_{u_k}\right)(f^\lam_{\sigma^k u}x) - \left(\frac{d}{dx} f^\lam_{u_k}\right)(f^\lam_{\sigma^k u}y) \right| \leq \frac{M_1}{\gamma_1} \sum \limits_{k=1}^n \left| f^\lam_{\sigma^k u}x - f^\lam_{\sigma^k u}y \right| \nonumber \\ & \leq & \frac{M_1}{\gamma_1} \sum \limits_{k=1}^n \gamma_2^{n-k} |x-y| \leq \frac{M_1\diam(X)}{\gamma_1(1-\gamma_2)} < \infty. \label{eq:nonpar bdp}
	\end{eqnarray}
	Therefore, \eqref{eq:bdp} holds for $\lambda_1 = \lambda_2$ with some constant $C_{62} > 1$. Fix now $\lambda_1, \lambda_2 \in U$. By the mean value theorem we have
	\[ \log \frac{ \left|\frac{d}{dx} f^{\lam_1}_{u}(x) \right|}{\left| \frac{d}{dx} f^{\lam_2}_{u}(x) \right|} \leq \left| \log \left|\frac{d}{dx} f^{\lam_1}_{u}(x) \right| - \log \left| \frac{d}{dx} f^{\lam_2}_{u}(x) \right| \right| = \frac{ \left| \frac{d^2}{d\lam dx} f^\xi_{u}(x) \right|}{\left| \frac{d}{dx}f^\xi_{u}(x) \right|} |\lam_1 - \lam_2|\]
	for some $\xi$ between $\lam_1$ and $\lam_2$. Applying \eqref{eq:dxdlam bound} we obtain
	\be\label{eq:par bdp} \log \frac{ \left|\frac{d}{dx} f^{\lam_1}_{u}(x) \right|}{\left| \frac{d}{dx} f^{\lam_2}_{u}(x) \right|} \leq C_{62} |u||\lambda_1 - \lambda_2|.
	\ee
	Combining \eqref{eq:nonpar bdp} with \eqref{eq:par bdp} finishes the proof.
\end{proof}

The following proposition implies that, in the language of \cite[Section 4.2]{PS00}, the natural projection $\Pi^\lam$ belongs to the class $C^{1,\delta}(U)$.

\begin{prop}\label{prop:pi hoelder}
	There exists a constant $C_{\delta} > 0$ such that
	\[ \left| \frac{d}{d\lam} \Pi^{\lam_1}(u) - \frac{d}{d\lam} \Pi^{\lam_2}(u)\right| \leq C_{\delta} |\lam_1 - \lam_2|^\delta \]
	holds for all $\lam_1, \lam_2 \in U$ and $u \in \Omega$.
\end{prop}

\begin{proof}
	Fix $u = (u_1, u_2, \ldots) \in \Om,y \in X$ and let $F_n(\lambda) = f_{u_1}^\lam \circ\cdots\circ f_{u_n}^\lam(y)$ for $\lam \in U$. It is clear from \ref{as:hyper} that $F_n(\lam)$ converge to $\Pi^\lam$ uniformly on $U$. Therefore, by Lemma \ref{lem:dlam hoelder}, it is enough to show that $\frac{d}{d\lam}F_n$ is uniformly convergent. It is sufficient to show
	\be\label{eq:Fn sum}
	\sum \limits_{n=1}^{\infty} \left\|\frac{d}{d\lam}F_{n+1} - \frac{d}{d\lam}F_n\right\|_{\infty} < \infty.
	\ee
	We have
	\[ \frac{d}{d\lam}F_{n+1}(\lam) = \left( \frac{d}{dx}f^\lam_{u_1 \ldots u_n}(f^\lam_{u_{n+1}}(y)) \right) \cdot \left(\frac{d}{d\lam} f^\lam_{u_{n+1}}(y) \right) + \left(\frac{d}{d\lam}f^{\lam}_{u_1 \ldots u_n} \right) \left( f^\lam_{u_{n+1}}(y) \right). \]
	Consequently, by  \ref{as:hyper} and \eqref{eq:dxdlam bound}
	\begin{eqnarray}
	\left| \frac{d}{d\lam}F_{n+1} - \frac{d}{d\lam}F_n \right| &\leq& \left| \left( \frac{d}{dx}f^\lam_{u_1 \ldots u_n}(f^\lam_{u_{n+1}}(y)) \right) \cdot \left(\frac{d}{d\lam} f^\lam_{u_{n+1}}(y)  \right) \right| + \nonumber \\
	&  & \left| \left(\frac{d}{d\lam}f^{\lam}_{u_1 \ldots u_n} \right) \left( f^\lam_{u_{n+1}}(y) \right) - \left(\frac{d}{d\lam}f^{\lam}_{u_1 \ldots u_n} \right) \left( y \right) \right| \nonumber \\
	& \leq & \gamma_2^n \sup \limits_{\lam \in U} \left|\frac{d}{d\lam} f^\lam_{u_{n+1}}(y)\right| + \sup \limits_{\lam \in U} \left\| \frac{d^2}{dx d\lam} f^{\lam}_{u_1 \ldots u_n}\right\|_{\infty}|f^{\lam}_{u_{n+1}}(y) - y| \nonumber \\
	& \leq & \gamma_2^n \sup \limits_{\lam \in U} \left|\frac{d}{d\lam} f^\lam_{u_{n+1}}(y)\right| + 2C_{52}n \sup \limits_{\lam \in U} \left\| \frac{d}{dx} f^{\lam}_{u_1 \ldots u_n}\right\|_{\infty} \nonumber \\
	& \leq & \left(\sup \limits_{\lam \in U} \left|\frac{d}{d\lam} f^\lam_{u_{n+1}}(y)\right| + 2C_{52}\right)n\gamma_2^n. \nonumber
	\end{eqnarray}
	As $\sup \limits_{\lam \in U} \left|\frac{d}{d\lam} f^\lam_{u_{n+1}}(y)\right| < \infty$ by  \ref{as:lam hoelder}, we have proved \eqref{eq:Fn sum}.
\end{proof}

\begin{lem}\label{lem:dx dist comparision}
	For every $\beta>0$ and $\lam_0$ there exist constants $\xi>0$ and ${0 < c_1<1}$ such that
	\[ c_1 d_{\lam_0}(u,v)^{1 + \beta / 4}\leq \left| \frac{d}{dx} f^\lambda_{u \wedge v}(x) \right| \leq \frac{1}{c_1} d_{\lam_0}(u,v)^{1 - \beta/4}  \]
	holds for all $x \in X, u,v \in \Omega$ and $\lam \in U$ with $|\lambda - \lambda_0| < \xi$.
\end{lem}

\begin{proof}
	Let $n = |u \wedge v|$. Note that by the mean value theorem $d_{\lam_0}(u,v) = |\frac{d}{dx} f^{\lam_0}_{u \wedge v}(y)|$ for some $y \in X$ (recall that we assume $\diam(X)=1$). Therefore, Lemma \ref{lem:pbdp} implies
	\be\label{eq:dist bdp bound}
	\frac{1}{C_{62}}e^{-c_{62}|\lam - \lam_0|n} \leq \left| \frac{ \frac{d}{dx} f^\lambda_{u \wedge v}(x) }{d_{\lam_0}(u,v)} \right| \leq C_{62}e^{c_{62}|\lam - \lam_0|n}.
	\ee
	On the other hand, by  \ref{as:hyper},
	\[ d_{\lam_0}(u,v) \leq \gamma_2^n,  \]
	hence
	\[
	c_1 d_{\lam_0}(u,v)^{\beta/4} \leq c_1 \gamma_2^{n\beta/4} \leq  \frac{1}{C_{62}}e^{-c_{62}|\lam - \lam_0|n},
	\]
	where the second inequality holds for all $n \in \N$ provided that $c_1$ and $|\lambda - \lambda_0|$ are small enough. Combining this with \eqref{eq:dist bdp bound} finishes the proof.
\end{proof}

The following proposition implies that the natural projection $\Pi^\lam$ is $1,\delta$-regular, as defined in \cite[Section 4.2]{PS00}

\begin{prop}\label{prop:C_1_delta_pi}
	For every $\beta>0$ and $\lam_0$ there exist constants $C_{\beta, 1}, C_{\beta, 1, \delta} >0$ such that inequalities
	\be\label{eq:4.21.1}
	\left| \frac{d}{d\lam} \left( \Pi^{\lam}(u) - \Pi^{\lam}(v) \right) \right| \leq C_{\beta, 1}d_{\lam_0}(u,v)^{1 - \beta}
	\ee
	and
	\be\label{eq:4.21.2}
	\left| \frac{d}{d\lam} \left( \Pi^{\lam_1}(u) - \Pi^{\lam_1}(v) \right) - \frac{d}{d\lam} \left( \Pi^{\lam_2}(u) - \Pi^{\lam_2}(v) \right) \right| \leq C_{\beta, 1, \delta}|\lam_1 - \lam_2|^\delta d_{\lam_0}(u,v)^{1 - \beta}
	\ee
	hold for all $u,v \in \Om$ and $\lam, \lam_1, \lam_2 \in U$ close enough to $\lam_0$. 
\end{prop}

\begin{proof} See Appendix \ref{app:proof of C_1_delta_pi}.
\end{proof}

\section{Proof of Theorem \ref{thm:main_hausdorff}}\label{sec:proof of main_hausdorff}

The { argument} follows closely the proof of \cite[Theorem 4.2]{BR18} (note that we do not assume measures $\mu_\lam$ to be quasi-Bernoulli), extending the method of \cite{SSU Parabolic} to the case of parameter dependent measures.

%

The key step in the proof of Theorem \ref{thm:main_hausdorff} is the following proposition.

\begin{prop}\label{prop:hdim_lam0_main}
Let $\{f_j^\lam\}_{j \in \Ak}$ be a parametrized IFS satisfying smoothness assumptions \ref{as:C2} - \ref{as:hyper} and the transversality condition \ref{as:trans} on $U$. Let $\left\{ \mu_\lam \right\}_{\lam \in \ov{U}}$ be a collection of finite ergodic shift-invariant Borel measures on $\Om$ satisfying \ref{as:measure_cont}, such that $h_{\mu_{\lam}}$ and $\chi_{\mu_{\lam}}$ are continuous in $\lam$. Then for every $\lam_0 \in \ov{U}$ and every $\eps>0$ there exists an open neighbourhood $U'$ of $\lam_0$ such that
\[ \hdim((\Pi^\lam)_*\mu_\lam) \geq \min \left\{ 1 , \frac{h_{\mu_{\lam_0}}}{\chi_{\mu_{\lam_0}}} \right\} - \eps \]
holds for Lebesgue almost every $\lam \in U'$.
\end{prop}

\begin{proof} Fix $\lam_0 \in \ov{U},\ \eps>0$ and $\eps'>0$. By the Shannon-McMillan-Breiman theorem and Birkhoff's ergodic theorem applied to the function $\Om \ni \om \mapsto -\log \left| f'_{\om_1} (\Pi^{\lam}(\sigma \om)) \right|$, we have that
\[ \frac{1}{n} \log \mu_{\lam}([\om|_n]) \to - h_{\mu_\lam} \text{ for } \mu_{\lam} \text{-a.e. } \om \in \Om  \]
and
\[ \frac{1}{n} \log \left| \left(f^\lam_{\om|_n}\right)' (\Pi^\lam (\sigma^n\om)) \right| \to - \chi_{\mu_\lam} \text{ for } \mu_{\lam} \text{-a.e. } \om \in \Om,  \]
hold for every $\lam \in \ov{U}$.
By Egorov's theorem, for every $\lam \in \ov{U}$ there exists $C_\lam > 0$ and a Borel set $A_{\lam} \subset \Om$ with $\mu_{\lam}(A_{\lam}) > 1 - \eps'$, such that
\be\label{eq:entropy ineq lam} C^{-1}_\lam e^{-n(h_{\mu_{\lam}} + \eps)} \leq \mu_\lam([\om|_n]) \leq C_\lam e^{-n(h_{\mu_{\lam}} - \eps)} 
\ee
and
\be\label{eq:lyap ineq lam} C^{-1}_\lam e^{-n(\chi_{\mu_{\lam}} + \eps)} \leq \left| \left(f^\lam_{\om|_n}\right)' (\Pi^\lam (\sigma^n\om)) \right| \leq C_\lam e^{-n(\chi_{\mu_{\lam}} - \eps)}
\ee
hold for every $\om \in A_\lam$ and $n \geq 1$. Let $\xi>0$ be such that \ref{as:measure_cont} holds and $|h_{\mu_\lam} - h_{\mu_{\lam_0}}| < \eps$, $|\chi_{\mu_{\lam}} - \chi_{\mu_{\lam_0}}| < \eps,\ c_{62}|\lam - \lam_0|<\eps$ for $|\lam - \lam_0| < \xi$ ($c_{62}$ is the constant from Lemma \ref{lem:pbdp}), and set $U' = B(\lam_0, \xi) \cap \ov{U}$. By Lusin's theorem, there exists $\tilde{C}>0$ and a Borel set $U_{\eps'} \subset U'$ containing $\lam _0$ such that \[\Leb(U' \setminus U_{\eps'}) < \eps' \text{ and } C_{\lam} \leq \tilde{C} \text{ for } \lam \in U_{\eps'}.\]
Now let
\[
\begin{split} A = \Bigg\{  \om \in \Om : \underset{n \geq 1}{\forall}\  C^{-1}\tilde{C}^{-1} e^{-n(h_{\mu_{\lam_0}} + 2\eps)} & \leq  \mu_{\lam_0}([\om|_n]) \leq C\tilde{C} e^{-n(h_{\mu_{\lam_0}} - 2\eps)} \text{ and } \\
C_{62}^{-1}\tilde{C}^{-1} e^{-n(\chi_{\mu_{\lam_0}} + 2\eps)} &\leq  \left| \left(f^{\lam_0}_{\om|_n}\right)' (\Pi^{\lam_0} (\sigma^n\om)) \right| \leq C_{62}\tilde{C} e^{-n(\chi_{\mu_{\lam_0}} - 2\eps)}  \Bigg\}.
\end{split}
\]
It follows from \eqref{eq:entropy ineq lam}, \eqref{eq:lyap ineq lam}, the choice of $\xi$ and Lemma \ref{lem:pbdp} that for each $\lam \in U_{\eps'}$ we have $A_\lam \subset A$, hence $\mu_\lam(A) > 1 - \eps'$. Let $\tilde{\mu}_\lam = \mu_\lam|_{A}$. Note that the set $A$ does not depend on $\lam$. Define
\[ A_n = \{ u \in \Ak^n : \text{ there exists } \om \in A \text{ with } u = \om|_n \}. \]
Note that if $u \notin A_n$, then $[u] \cap A = \emptyset$, hence $\tilde{\mu}_\lam([u])=0$. If $u \in A_n$, then
\begin{equation}\label{eq:entropy ineq}
C^{-1}\tilde{C}^{-1} e^{-n(h_{\mu_{\lam_0}} + 2\eps)} \leq  \mu_{\lam_0}([u]) \leq C\tilde{C} e^{-n(h_{\mu_{\lam_0}} - 2\eps)}
\end{equation}
and
\begin{equation}\label{eq:lyap ineq}
\tilde{C}^{-1}C_{62}^{-2} e^{-n(\chi_{\mu_{\lam_0}} + 3\eps)} \leq  \left| \left(f^\lam_{u}\right)' (x) \right| \leq \tilde{C}C^2_{62} e^{-n(\chi_{\mu_{\lam_0}} - 3\eps)}
\end{equation}
hold for any $x \in X$ by Lemma \ref{lem:pbdp}. Fix $0 < s <1$ and consider the integral
\[ \mathcal{I} = \int \limits_{U_{\eps'}} \int \limits_{\Om} \int \limits_{\Om} \left| \Pi^\lam(\om_1) - \Pi^\lam(\om_2) \right|^{-s}\,d\tilde{\mu}_\lam(\om_1)\,d\tilde{\mu}_\lam(\om_2)\,d\lam. \]
If $\mathcal{I} < \infty$, then by Frostman's lemma \cite[Theorem 4.13]{Falconer} we have $\hdim((\Pi^\lam)_*\mu_\lam) \geq \hdim((\Pi^\lam)_*\tilde{\mu}_{\lam}) \geq s$ for Lebesgue almost every $\lam \in U_{\eps'}$. By \eqref{eq:lyap ineq},
\[\begin{split} \mathcal{I} &= \int \limits_{U_{\eps'}} \sum \limits_{n=0}^\infty \sum \limits_{u \in A_n} \sum_{\substack{a,b\in \Ak\\ a \neq b}}\ \iint\limits_{[ua] \times [ub]} \left| f^\lam_u \left(\Pi^\lam(\sigma^n\om_1)\right) - f^\lam_u \left(\Pi^\lam(\sigma^n\om_2)\right) \right|^{-s}\,d\tilde{\mu}_\lam(\om_1)\,d\tilde{\mu}_\lam(\om_2)\,d\lam \\
& \leq \tilde{C}^{s}C_{62}^{2s} \int \limits_{U_{\eps'}} \sum \limits_{n=0}^\infty e^{ns(\chi_{\mu_{\lam_0}} + 3\eps)} \sum \limits_{u \in A_n} \sum_{\substack{a,b\in \Ak\\ a \neq b}}\ \iint\limits_{[ua] \times [ub]} \left| \Pi^\lam(\sigma^n\om_1) - \Pi^\lam(\sigma^n\om_2) \right|^{-s}\,d\tilde{\mu}_\lam(\om_1)\,d\tilde{\mu}_\lam(\om_2)\,d\lam.
\end{split}
\]
For $m \geq 0$ set
\[B^{\lam}_m = \{ (\om_1, \om_2) \in \Om \times \Om : \left| \Pi^\lam(\om_1) - \Pi^\lam(\om_2) \right| \leq 2^{-m} \}\] and note that
\be\label{eq:energy sum}
 \left| \Pi^\lam(\om_1) - \Pi^\lam(\om_2) \right|^{-s} \leq \sum \limits_{m=0}^{\infty}2^{s(m+1)} \mathds{1}_{B_m^\lam}(\om_1, \om_2).
\ee
Indeed, if $\Pi^\lam(\om_1) = \Pi^\lam(\om_2)$, then the right-hand side is divergent. Otherwise, there exists $m \geq 0$ such that $2^{-(m+1)}< \left| \Pi^\lam(\om_1) - \Pi^\lam(\om_2) \right| \leq 2^{-m}$, hence $\left| \Pi^\lam(\om_1) - \Pi^\lam(\om_2) \right|^{-s} \leq 2^{s(m+1)} \mathds{1}_{B_m^\lam}(\om_1, \om_2)$. For $m \geq 0$ let $k = k(m)$ be {minimal}
such that $\gamma_2^k \leq 2^{-(m+1)}$, so $k \leq Q(m+1)$ for a constant $Q = \lceil \frac{\log 2}{- \log \gamma_2} \rceil$. Let
\[D^{\lam}_m = \{ (\om_1, \om_2) \in \Om \times \Om : \left| \Pi^\lam(\om_1|_k 1^{\infty}) - \Pi^\lam(\om_2|_k 1^{\infty}) \right| \leq 2^{-(m-1)} \},\]
where $1^{\infty}$ denotes the infinite sequence in $\Omega$ formed by the symbol $1\in\mathcal{A}$. Note that by \ref{as:hyper} and the choice of $k$, we have $B^{\lam}_m \subset D^{\lam}_m$. Moreover, $D^{\lam}_m$ is a union of cylinders of length $k$. Applying this together with \eqref{eq:energy sum} and \ref{as:measure_cont} for $\lam \in U_{\eps'}$ yields
\[
\begin{split}
\iint\limits_{[ua] \times [ub]} & \left| \Pi^\lam(\sigma^n\om_1) - \Pi^\lam(\sigma^n\om_2) \right|^{-s}d\tilde{\mu}_\lam(\om_1)\,d\tilde{\mu}_\lam(\om_2) \\
& \leq \sum \limits_{m=0}^{\infty} 2^{(m+1)s} \iint\limits_{[ua] \times [ub]}  \mathds{1}_{B^{\lam}_m}(\sigma^n \om_1, \sigma^n \om_2)\,d\tilde{\mu}_\lam(\om_1)\,d\tilde{\mu}_\lam(\om_2) \\
& \leq 2^s \sum \limits_{m=0}^{\infty} 2^{ms} \iint\limits_{[ua] \times [ub]}  \mathds{1}_{D^{\lam}_m}(\sigma^n \om_1, \sigma^n \om_2)\,d\tilde{\mu}_\lam(\om_1)\,d\tilde{\mu}_\lam(\om_2) \\
& = 2^s \sum \limits_{m=0}^{\infty} 2^{ms} \sum \limits_{l, p \in \Ak^{k-1}} \tilde{\mu}_\lam \left( [ual] \right)  \tilde{\mu}_\lam \left( [ubp] \right) \mathds{1}_{D^{\lam}_m}(al 1^\infty, bp1^{\infty}) \\
& \leq {C^2 2^s} \sum \limits_{m=0}^{\infty} 2^{ms}e^{{2}\eps(n+Q(m+1))} \sum \limits_{l, p \in \Ak^{k-1}} \mu_{\lam_0} \left( [ual] \right)  \mu_{\lam_0} \left( [ubp] \right) \mathds{1}_{D^{\lam}_m}(al 1^\infty, bp1^{\infty}) \\
& = {C^2 2^s} e^{2\eps Q} \sum \limits_{m=0}^{\infty} 2^{ms}e^{{2}\eps(n+Qm)} \iint\limits_{[ua] \times [ub]}  \mathds{1}_{D^{\lam}_m}(\sigma^n \om_1, \sigma^n \om_2)\,d\mu_{\lam_0}(\om_1)\,d\mu_{\lam_0}(\om_2).
\end{split}
\]
Moreover, transversality condition \ref{as:trans} implies that for $(\om_1, \om_2) \in [ua] \times [ub]$ with $a \neq b$ we have (we use here an equivalent condition \eqref{tran3}, see Lemma \ref{lem-tran})
\[ \int \limits_{U_{\eps'}} \mathds{1}_{D^{\lam}_m}(\sigma^n \om_1, \sigma^n \om_2)\, d\lam \leq \Lk^1\left\{\lam\in \ov U:\ |\Pi^\lam(\sigma^n \om_1) - \Pi^\lam(\sigma^n \om_2)| \le 2^{-(m-1)} \right\} \leq C_T 2^{-(m-1)}  \]
for some constant $C_T$ (depending only on the IFS). Applying both of the above calculations to $\mathcal{I}$, changing the order of integration, and applying \eqref{eq:entropy ineq}, we obtain, setting $C_{70} = \tilde{C}^{s}C_{62}^{2s}C^{2}C_T2^{s+1}$ and $C_{71} = \tilde{C}CC_{70}$,
\[
\begin{split}
\mathcal{I} & \leq C_{70} e^{{2}\eps Q} \sum \limits_{n=0}^\infty e^{n\left(s(\chi_{\mu_{\lam_0}} + 3\eps) + {2}\eps\right)} \sum \limits_{u \in A_n} \sum_{\substack{a,b\in \Ak\\ a \neq b}} \sum \limits_{m=0}^{\infty} 2^{m(s-1)}e^{{2}\eps Qm} \mu_{\lam_0}\left([ua]\right)\mu_{\lam_0}\left([ub]\right) \\
& \leq C_{70} e^{{2}\eps Q} \sum \limits_{n=0}^\infty e^{n\left(s(\chi_{\mu_{\lam_0}} + 3\eps) + {2}\eps\right)} \sum \limits_{u \in A_n} \mu_{\lam_0}\left([u]\right)^2  \sum \limits_{m=0}^{\infty} 2^{m(s-1)}e^{{2}\eps Qm}  \\
& \leq C_{71} e^{{2}\eps Q} \sum \limits_{n=0}^\infty e^{n\left(s(\chi_{\mu_{\lam_0}} + 3\eps) - h_{\mu_{\lam_0}} + {4}\eps\right)} \sum \limits_{u \in A_n} \mu_{\lam_0}\left([u]\right) \sum \limits_{m=0}^{\infty} 2^{m(s + Q' \eps-1)}\\
& \leq C_{71} e^{{2}\eps Q} \sum \limits_{n=0}^\infty e^{n\left(s(\chi_{\mu_{\lam_0}} + 3\eps) - h_{\mu_{\lam_0}} + {4}\eps\right)} \sum \limits_{m=0}^{\infty} 2^{m(s + Q' \eps-1)},
\end{split}
\]
where ${Q'=2Q\log_2 e}$. Therefore, $\mathcal{I}<\infty$ provided $s+Q'\eps < 1$ and $s < \frac{h_{\mu_{\lam_0}} - {4}\eps}{\chi_{\mu_{\lam_0}} + 3\eps}$. Consequently,
\[ \hdim((\Pi^\lam)_* \mu_\lam) \geq \hdim((\Pi^\lam)_* \tilde{\mu}_\lam) \geq \min \left\{ 1 - Q' \eps, \frac{h_{\mu_{\lam_0}} - {4}\eps}{\chi_{\mu_{\lam_0}} + 3\eps}  \right\} \text{ for } \Leb\text{-a.e.}\ \lam \in U_{\eps'}. \]
As $\eps'$ can be taken arbitrary small, the proof is finished.
\end{proof}

We can now finish the proof of Theorem \ref{thm:main_hausdorff}. {As $\hdim((\Pi^\lam)_*\mu_\lam) \leq \min \left\{ 1, \frac{h_{\mu_\lam}}{\chi_{\mu_\lam}} \right\}$ (see \cite[Theorem 3.1 and Remark 3.2]{U})}, it is enough to prove that $\hdim((\Pi^\lam)_*\mu_\lam) \geq \min \left\{ 1, \frac{h_{\mu_\lam}}{\chi_{\mu_\lam}} \right\}$ holds almost surely. Assume that this is not the case. Then, there exists $\eps>0$ such that the set
\[ A = \left\{ \lam \in \ov{U} : \hdim((\Pi^\lam)_*\mu_\lam) < \min \left\{ 1, \frac{h_{\mu_\lam}}{\chi_{\mu_\lam}} \right\} - \eps \right\}\] has positive Lebesuge measure. Let $\lam_0$ be a density point of $A$. By the continuity of $\lam \mapsto h_{\mu_{\lam}},\ \lam \mapsto \chi_{\mu_{\lam}}$ and $\chi_{\mu_\lam} > 0$ (following from \ref{as:hyper}), we obtain that $\lam \mapsto \min \left\{ 1, \frac{h_{\mu_\lam}}{\chi_{\mu_\lam}} \right\}$ is continuous as well. Therefore, there exists an open neighbourhood $U'$ of $\lam_0$ such that
\[ \min \left\{ 1, \frac{h_{\mu_\lam}}{\chi_{\mu_\lam}} \right\} \leq \min \left\{ 1, \frac{h_{\mu_{\lam_0}}}{\chi_{\mu_{\lam_0}}} \right\} + \frac{\eps}{2}\ \text{ for }\ \lam \in U'.\] By Proposition \ref{prop:hdim_lam0_main} we can also assume that
\[ \hdim((\Pi^\lam)_*\mu_\lam) \geq \min \left\{ 1 , \frac{h_{\mu_{\lam_0}}}{\chi_{\mu_{\lam_0}}} \right\} - \frac{\eps}{2} \text{ for } \Leb \text{-a.e.}\ \lam \in U',\]
hence
\[ \hdim((\Pi^\lam)_*\mu_\lam) \geq \min \left\{ 1, \frac{h_{\mu_\lam}}{\chi_{\mu_\lam}} \right\} - \eps  \text{ for } \Leb \text{-a.e.}\ \lam \in U'. \]
This however means that $\lam_0$ cannot be a density point of $A$, a contradiction. Theorem \ref{thm:main_hausdorff} is proved. 

\section{Transversality of degree $\beta$}

In this section we prove that an IFS satisfying the transversality condition \ref{as:trans}, satisfies also the \textit{transversality of degree $\beta$}, as defined in \cite{PS00}, with arbitrary small $\beta>0$. This will be useful later, as the proof of \ref{thm:main_cor_dim} follows the approach of Peres and Schlag \cite{PS00}, where the transversality of degree $\beta$ is a key concept. In fact, \cite{PS00} uses the term ``transversality of order $\beta$'', but the term ``transversality of degree $\beta$,'' as in Mattila, seems more appropriate.

\begin{prop}\label{prop_beta trans}
Let $\{f_j^\lam\}_{j \in \Ak}$ be a parametrized IFS satisfying smoothness assumptions \ref{as:C2} - \ref{as:hyper} and the transversality condition \ref{as:trans} on $U$. For every $\lam_0 \in U$ and $\beta>0$ there exists $c_\beta>0$ and an open neighbourhood $J$ of $\lam_0$ such that
\be \label{tran beta}
\left|\Pi^\lam(u) - \Pi^\lam(v)\right| < c_\beta\cdot d_{\lam_0}(u,v)^{1+\beta} \implies \left|\textstyle{\frac{d}{d\lam}}(\Pi^\lam(u) - \Pi^\lam(v))\right| \ge c_\beta\cdot d_{\lam_0}(u,v)^{1+\beta}.
\ee
holds for all $u, v\in \Om$ and $\lam \in J$.
\end{prop}

\begin{proof}
For short, let us denote the metric $d_{\lam_0}$ by $d$. Let $n = |u\wedge v|$, so  that $u\wedge v = u_1\ldots u_n$. We have 
\begin{eqnarray}
\textstyle{\frac{d}{d\lam}} (\Pi^\lam(u) - \Pi^\lam(v)) & = & \textstyle{\frac{d}{d\lam}} \left[f_{u_1\ldots u_n}^\lam (\Pi^\lam(\sig^n u)) - f_{u_1\ldots u_n}^\lam (\Pi^\lam(\sig^n v)) \right] \nonumber \\[1.2ex]
& = & \left(\textstyle{\frac{d}{d\lam}} f_{u_1\ldots u_n}^\lam\right)(\Pi^\lam (\sig^n u)) - \left(\textstyle{\frac{d}{d\lam}} f_{u_1\ldots u_n}^\lam\right)(\Pi^\lam (\sig^n v)) + \nonumber \\[1.2ex]
& & \left(\textstyle{\frac{d}{dx}} f_{u_1\ldots u_n}^\lam\right)(\Pi^\lam (\sig^n u))\cdot \textstyle{\frac{d}{d\lam}} \Pi^\lam(\sig^n u) - \left(\textstyle{\frac{d}{dx}} f_{u_1\ldots u_n}^\lam\right)(\Pi^\lam (\sig^n v)) \cdot \textstyle{\frac{d}{d\lam}} \Pi^\lam(\sig^n v) \nonumber \\[1.2ex]
& = & \left(\textstyle{\frac{d}{d\lam}} f_{u_1\ldots u_n}^\lam\right)(\Pi^\lam (\sig^n u)) - \left(\textstyle{\frac{d}{d\lam}} f_{u_1\ldots u_n}^\lam\right)(\Pi^\lam (\sig^n v)) + \nonumber \\[1.2ex]
& & \left(\textstyle{\frac{d}{dx}} f_{u_1\ldots u_n}^\lam\right)(\Pi^\lam (\sig^n u))\cdot \left[ \textstyle{\frac{d}{d\lam}} \left(\Pi^\lam(\sig^n u) - \Pi^\lam(\sig^n v)\right)\right] + \nonumber \\[1.2ex]
& & \left[\left(\textstyle{\frac{d}{dx}} f_{u_1\ldots u_n}^\lam\right)(\Pi^\lam (\sig^n u)) - \left(\textstyle{\frac{d}{dx}}f_{u_1\ldots u_n}^\lam\right)(\Pi^\lam (\sig^n v))\right] \cdot  \textstyle{\frac{d}{d\lam}} \Pi^\lam(\sig^n v) \nonumber  \\[1.2ex]
& =: & A_1 + A_2 + A_3.
\label{est0}
\end{eqnarray}
On the other hand,
\begin{eqnarray} 
\left|\Pi^\lam(u) - \Pi^\lam(v)\right| & = & \left|\textstyle{\frac{d}{dx}}\ f^\lam_{u_1,\ldots u_n}(y)\right|\cdot \left|\Pi^\lam(\sig^n u) - \Pi^\lam(\sig^n v)\right| \nonumber \\[1.2ex]
& \ge & c_1\cdot d(u,v)^{1 + \beta/4} \cdot \left|\Pi^\lam(\sig^n u) - \Pi^\lam(\sig^n v)\right| \label{est2},
\end{eqnarray}
for some $y \in X$, $c_1>0$, and $\lam$ sufficiently close to $\lam_0$, by Lemma \ref{lem:dx dist comparision}. Similarly,
\be \label{est3}
|A_2| \ge c_1\cdot d(u,v)^{1 + \beta/4} \cdot \left|\textstyle{\frac{d}{d\lam}}(\Pi^\lam(\sig^n u) - \Pi^\lam(\sig^n v))\right|.
\ee
Further, by Lemmas \ref{lem:d2 bounds}, \ref{lem:dx dist comparision} and Proposition \ref{prop:pi hoelder} (which implies that $\frac{d}{d\lam}\Pi^\lam$ is bounded) we have
\be \label{est4}
|A_1| \le \left|\Pi^\lam(\sig^n u) - \Pi^\lam(\sig^n v)\right| C'_2 n \cdot d(u,v)^{1-\beta/4}
\ee
and
\be \label{est5}
|A_3| \le \left|\Pi^\lam(\sig^n u) - \Pi^\lam(\sig^n v)\right| C'_2 \cdot d(u,v)^{1-\beta/4}
\ee
for some constant $C_2'$ large enough.
Assuming
$$
\left|\Pi^\lam(u) - \Pi^\lam(v)\right| < c_\beta\cdot d(u,v)^{1+\beta},
$$
we obtain from \eqref{est2}:
\be \label{est7}
\left|\Pi^\lam(\sig^n u) - \Pi^\lam(\sig^n v)\right| \le \frac{c_\beta}{c_1}\cdot d(u,v)^{3\beta/4},
\ee
and then, from \eqref{est0}, \eqref{est3}, \eqref{est4}, \eqref{est5}:
\begin{eqnarray*}
	\left|\textstyle{\frac{d}{d\lam}}(\Pi^\lam(u) - \Pi^\lam(v))\right| & \ge & |A_2| - |A_1| - |A_3| \\[1.2ex]
	& \ge & c_1\cdot d(u,v)^{1 + \beta/4} \cdot \left|\textstyle{\frac{d}{d\lam}}(\Pi^\lam(\sig^n u) - \Pi^\lam(\sig^n v))\right| \\[1.2ex]
	&       & - C'_2 (n+1) \cdot d(u,v)^{1-\beta/4} \cdot \left|\Pi^\lam(\sig^n u) - \Pi^\lam(\sig^n v)\right| \\[1.2ex]
	& \ge & c_1\cdot d(u,v)^{1 + \beta/4} \cdot \left|\textstyle{\frac{d}{d\lam}}(\Pi^\lam(\sig^n u) - \Pi^\lam(\sig^n v))\right| \\[1.2ex]
	&       & - \frac{C'_2 c_\beta}{c_1}\cdot (n+1) \cdot  d(u,v)^{1+\beta/2}.
\end{eqnarray*}
Assuming $c_\beta<c_1\eta$, so that we can use transversality assumption \ref{as:trans} for the pair $\sig^n u, \sig^n v$ by \eqref{est7}, keeping in mind that $d(u,v)\le 1$, we obtain
$$
\left|\textstyle{\frac{d}{d\lam}}(\Pi^\lam(\sig^n u) - \Pi^\lam(\sig^n v))\right| \ge \eta,
$$
hence
$$
\left|\textstyle{\frac{d}{d\lam}}(\Pi^\lam(u) - \Pi^\lam(v))\right| \ge c_1\cdot d(u,v)^{1 + \beta/4} \cdot \left[\eta - \frac{C'_2 c_\beta}{c^2_1}\cdot (n+1)\cdot d(u,v)^{\beta/4}\right].
$$
Note that $d(u,v) \le \gam_2^n$, where $\gam_2 < 1$ is from \ref{as:hyper}, and let
$$
C'_3 := \max \{(n+1) \gam_2^{n\beta/4},\ n\ge 0\}.
$$
Choose $$c_\beta < \frac{\eta c_1^2}{2 C_2' C'_3},$$ then
$$
\left|\textstyle{\frac{d}{d\lam}}(\Pi^\lam(u) - \Pi^\lam(v))\right| \ge \frac{c_1\eta}{2} \cdot d(u,v)^{1 + \beta/4} \ge c_\beta\cdot d(u,v)^{1+\beta},
$$
if we also make sure that $c_\beta < c_1\eta/2$, completing the proof of \eqref{tran beta}.
\end{proof}


\section{Energy estimates}
The following theorem is the main result of this section and the main ingredient of the proof of Theorem \ref{thm:main_cor_dim}. {It is modelled after \cite[Theorem 4.9]{PS00}.}

\begin{thm}\label{thm:sobolev integral bound}
Let $\{f_j^\lam\}_{j \in \Ak}$ be a parametrized IFS satisfying smoothness assumptions \ref{as:C2} - \ref{as:hyper} and the transversality condition \ref{as:trans} on $U$. Let $\left\{ \mu_\lam \right\}_{\lam \in \ov{U}}$ be a collection of finite Borel measures on $\Om$ satisfying \ref{as:measure}. Fix $\lambda_0 \in U$, $\beta > 0$, $\gamma>0$, $\eps>0$ and $q>1$ such that $1+2\gamma + \eps<q<1+\min\{\delta, \theta\}$. Then, there exists a (small enough) open interval $J \subset U$ containing $\lam_0$ such that for every smooth function $\rho$ on $\R$ with
{$0\le \rho\le 1$ and} $\supp(\rho) \subset J$ there exist constants $\widetilde{C}_1>0,\ \widetilde{C}_2>0$ such that
	$$
	\int_{J}\|\nu_\lambda\|_{2,\gamma}^2\rho(\lambda)\,d\lambda \leq \widetilde{C}_1\Ek_{q(1+a_0\beta)}(\mu_{\lambda_0}, d_{\lam_0}) + \widetilde{C}_2,
	$$
where $a_0=\frac{8+4\delta}{1+\min\{\delta, \theta\}}$.
\end{thm}

The rest of this section is devoted to the proof of the above theorem and we assume throughout  that all the assumptions of Theorem \ref{thm:sobolev integral bound} hold. We follow the approach of \cite{PS00}, with suitable modifications coming from the fact that measures $\mu_{\lam}$ depend on the parameter.

Throughout the section $x \lesssim y$ will mean $x \leq A y + B$, while $x \asymp y$ will mean $\frac{x}{A} \leq y \leq A \cdot x$, with positive constants $A,B$ being possibly dependent on all the parameters on which constants $\widetilde{C}_1, \widetilde{C}_2$ depend in Theorem \ref{thm:sobolev integral bound}.

Let $\psi$ be a Littlewood-Paley function on $\R$ from \cite[Lemma 4.1]{PS00}, that is, $\psi$ is of Schwarz class, $\what\psi\ge 0$, 
$$
\supp(\what\psi) \subset \{\xi:\ 1\le |\xi|\le 4\},\ \ \ \ \ \sum_{j\in \Z}\what\psi(2^{-j}\xi) = 1\ \ \mbox{for all}\ \xi\ne 0.
$$
It is known that such a function exists. We will need that $\psi$ decays faster than any power, that is, for any $q>0$ there is $C_q$ such that
\be\label{eq:psi decay}
|\psi(\xi)| \le C_q(1 + |\xi|)^{-q}.
\ee
We will also use that
\be \label{meanzero}
\int_\R \psi(\xi)\,d\xi = \what\psi(0) = 0.
\ee
In fact, all higher moments of $\psi$ also vanish, but this will not be needed for our purposes.  As $\psi$ has bounded derivative on $\R$, there exists $L>0$ such that
\begin{equation}\label{eq:psi_lipschitz}
|\psi(x) - \psi(y)|\leq L |x - y| \text{ for all } x,y \in \R.
\end{equation}
We have (see \cite[Lemma 4.1]{PS00}):
\be\label{eq:int to sum} \int_{\R}\|\nu_{\lambda}\|^2_{2,\gamma}\rho(\lambda)d\lambda \asymp \int_{\R} \sum \limits_{j=-\infty}^\infty 2^{2j\gamma} \int_{\R} (\psi_{2^{-j}} * \nu_\lambda)(x)d\nu_\lambda(x)\rho(\lambda)d\lambda,
\ee
where $\psi_{2^{-j}}(x) = 2^j \psi(2^j x)$. Let $\kappa = - \log_2 \gamma_1,\ Q = \log_2 e$ and choose $\xi > 0 $ small enough to have $2(4+Qc)\xi < \eps$ and
\begin{equation}\label{eq:eta} 0<\frac{4+2\gamma}{\kappa - Q\xi} < \frac{\eps}{2(4+Qc)\xi}.
\end{equation}
Choose an open interval $J$ containing $\lam_0$ so small that $2c|J|^\theta\leq\xi$ (with $c,\theta$ as in \ref{as:measure}) and \eqref{tran beta} hold. In order to prove Theorem \ref{thm:sobolev integral bound}, it is enough to consider in \eqref{eq:int to sum} the sum over $j\geq 0$, as $(\psi_{2^{-j}} * \nu_\lambda)(x)$ is bounded by $2^j \|\psi\|_\infty$, hence the sum over $j < 0$ converges to a bounded function. We now calculate for $\lambda \in B(\lambda_0, \xi),\ j \geq 0$ and $n \in \N$ (we will set later $n = n(j) = \tc j$ for suitable $\tc$):
\[
\begin{split} 
&\int_{\R} (\psi_{2^{-j}} * \nu_\lambda)(x)\,d\nu_\lambda(x) \\
&\quad= 2^j \int_{\R} \int_{\R} \psi(2^j(x-y))\, d\nu_\lambda(y)\,d\nu_\lambda(x) \\
&\quad = 2^j \int\limits_{\Omega} \int\limits_{\Omega} \psi\bigl(2^j(\Pi^\lam(\omega_1) - \Pi^\lam(\omega_2))\bigr) \,d\mu_\lambda(\omega_1)\,d\mu_\lambda(\omega_2)\\
&\quad \leq 2^j \int\limits_{\Omega} \int\limits_{\Omega} \psi\bigl(2^j(\Pi^\lam(\omega_1|_n 1^\infty) - \Pi^\lam(\omega_2|_n 1^\infty))\bigr) \,d\mu_\lambda(\omega_1)\,d\mu_\lambda(\omega_2) \,+  \\
&\quad + 2^j \int\limits_{\Omega} \int\limits_{\Omega} \bigl|\psi\bigl(2^j(\Pi^\lam(\omega_1) - \Pi^\lam(\omega_2))\bigr) - \psi\bigl(2^j(\Pi^\lam(\omega_1|_n 1^\infty) - \Pi^\lam(\omega_2|_n 1^\infty))\bigr)\bigr|\, d\mu_\lambda(\omega_1)\,d\mu_\lambda(\omega_2) \leq
\end{split}
\]
Using \eqref{eq:psi_lipschitz} we get that the last expression is
\[
\begin{split}
\leq\, & 2^j \sum \limits_{i \in \Ak^n} \sum \limits_{k \in \Ak^n} \psi\bigl(2^j(\Pi^\lam(i 1^\infty) - \Pi^\lambda(k 1^\infty))\bigr)\,\mu_{\lambda}([i])\,\mu_{\lambda}([k]) \,+  \\
& + 2^j \int\limits_{\Omega} \int\limits_{\Omega} L2^j\big( |\Pi^\lam(\omega_1) -\Pi^\lam(\omega_1|_n 1^\infty)| + |\Pi^\lam(\omega_2) -\Pi^\lam(\omega_2|_n 1^\infty)| \big)\, d\mu_\lambda(\omega_1)\,d\mu_\lambda(\omega_2) \leq
\end{split}
\]
Applying \ref{as:hyper} to the integral, we obtain (recall that we assume $\diam(X) =1$):
\[ \leq  2^j \sum \limits_{i \in \Ak^n} \sum \limits_{k \in \Ak^n} \psi\bigl(2^j(\Pi^\lam( i 1^\infty) - \Pi^\lam(k 1^\infty))\bigr)\,\mu_{\lambda}([i])\,\mu_{\lambda}([k]) + 2L2^{2j-\kappa n}=(*) \]
%
	Choose $\tc\geq 1$ such that
	\begin{equation}\label{eq:forn} \frac{4+2\gamma}{\kappa - Q\xi} \leq \tc \leq \frac{\eps}{2Q(2+c)\xi}\end{equation}
	(it exists due to \eqref{eq:eta}) and set $n = \tc j$.	Let us define a map $e_j\colon\Om\times\Om\times J\mapsto\R$ by
	\begin{equation}\label{eq:ej def}
	e_j(\omega_1,\omega_2,\lambda):= \begin{cases}\frac{\mu_\lambda([\omega_1|_n])\mu_\lambda([\omega_2|_n])}{\mu_{\lambda_0}([\omega_1|_n])\mu_{\lambda_0}([\omega_2|_n])}, &\text{ if } \mu_{\lambda_0}([\omega_1|_n])\mu_{\lambda_0}([\omega_2|_n]) \neq 0,\\
	1, & \text{ otherwise}.
	\end{cases}
	\end{equation}
	By \eqref{eq:forn}, \ref{as:measure} and the choice of $J$,
	\begin{equation}\label{eq:errorest}
	e_j(\omega_1,\omega_2,\lambda)\leq e^{2c|\lambda-\lambda_0|^\theta n}\leq e^{\xi \tc j} = 2^{Q \xi \tc j}\text{ for all $\omega_1,\omega_2$ and $\lambda\in B(\lambda_0,\xi)$.}
	\end{equation}
	Note also that by \ref{as:measure}, if $i \in \Om^*$ is a fixed finite word, then  $\mu_{\lam_0}([i]) = 0$ if and only if $\mu_{\lam}([i]) = 0$ for all $\lam \in \ov{U}$ (in other words: $\supp(\mu_{\lam_0}) = \supp(\mu_\lam)$). {Denote $\wt\Ak^n:= \{i\in\Ak^n:\ \mu_{\lam_0}([i])\ne 0\}$.}  We have, therefore, (note that now the integral is with respect to $\mu_{\lambda_0}$),
\[
\begin{split}
(*)& = 2^j \sum \limits_{i \in \wt\Ak^n} \sum \limits_{k \in \wt\Ak^n} 
\psi\bigl(2^j(\Pi^\lam(i 1^\infty) - \Pi^\lam(k 1^\infty))\bigr)\,\frac{\mu_\lambda([i])\mu_\lambda([k])}{\mu_{\lambda_0}([i])\mu_{\lambda_0}([k])}\,\mu_{\lambda_0}([i])\,\mu_{\lambda_0}([k]) + 2L2^{2j-\kappa \tc j} \\
&
= 2^{j} \int\limits_{\Omega} \int\limits_{\Omega} \psi\bigl(2^j(\Pi^\lam(\omega_1|_n 1^\infty) - \Pi^\lam(\omega_2|_n 1^\infty))\bigr)\,e_j(\omega_1,\omega_2,\lambda)\, d\mu_{\lambda_0}(\omega_1)\,d\mu_{\lambda_0}(\omega_2) + 2L2^{2j-\kappa \tc j} \\
&\leq 2^{j} \int\limits_{\Omega} \int\limits_{\Omega} \psi\bigl(2^j(\Pi^\lam(\omega_1) - \Pi^\lam(\omega_2))\bigr)\,e_j(\omega_1,\omega_2,\lambda) \,d\mu_{\lambda_0}(\omega_1)\,d\mu_{\lambda_0}(\omega_2) \\
&\quad + 2^{j}\int\limits_{\Omega} \int\limits_{\Omega} \bigl|\psi\bigl(2^j(\Pi^\lam( \omega_1) - \Pi^\lam(\omega_2))\bigr) - \psi\bigl(2^j(\Pi^\lam(\omega_1|_n 1^\infty) - \Pi^\lam(\omega_2|_n 1^\infty))\bigr)\bigr| \ \times \\
&\qquad\ \ \  \times e_j(\omega_1,\omega_2,\lambda)\,d\mu_{\lambda_0}(\omega_1)\,d\mu_{\lambda_0}(\omega_2) + 2L 2^{2j-\kappa \tc j} = (**)
\end{split}
\]
Estimating the second integral, similarly as before, by $2L 2^{j-\kappa \tc j}2^{Q\xi\tc j}$ we get
\[
(**) \leq 2^{j} \int\limits_{\Omega} \int\limits_{\Omega} \psi\bigl(2^j(\Pi^\lam(\omega_1) - \Pi^\lam(\omega_2))\bigr)\,e_j(\omega_1,\omega_2,\lambda)\, d\mu_{\lambda_0}(\omega_1)\,d\mu_{\lambda_0}(\omega_2) { + 2L 2^{2j-\kappa \tc j}(1+2^{Q \xi \tc j})}.
\]
Finally,
\begin{equation}\label{eq:discretization_bound}
\begin{split}
\int_{\R} (\psi_{2^{-j}} &* \nu_\lambda)(x) \, d\nu_\lambda(x) \leq \\
&\quad \leq 2^{j} \int\limits_{\Omega} \int\limits_{\Omega} \psi\bigl(2^j(\Pi^\lam(\omega_1) - \Pi^\lam(\omega_2))\bigr)\,e_j(\omega_1,\omega_2,\lambda)\, d\mu_{\lambda_0}(\omega_1)\,d\mu_{\lambda_0}(\omega_2) + 4L 2^{(2 + Q\tc\xi-\tc\kappa)j}.
\end{split}
\end{equation}
For $j$ large enough, we have, { in view of (\ref{eq:forn})},
\begin{equation}\label{eq:error_bound}
\begin{aligned}
2^{2j\gamma}4L 2^{2j + (Q\xi - \kappa) \tc j} &= 4L  2^{j(2+2\gamma + \tc(Q\xi - \kappa))} \leq 2^{j(3+2\gamma + \tc(Q\xi - \kappa))} \\ &= 2^{-j}2^{j(4+2\gamma + \tc(Q\xi - \kappa))} \leq 2^{-j}.
\end{aligned}
\end{equation}
Combining \eqref{eq:discretization_bound}, and \eqref{eq:error_bound} we obtain, recalling that the sum over $j<0$ in \eqref{eq:int to sum} converges:
\[\begin{split} 
\int_{J}\|&\nu_{\lambda}\|^2_{2,\gamma}\rho(\lambda)\,d\lambda\lesssim \int \limits_{\R} \sum \limits_{j=0}^\infty 2^{2j\gamma} \int\limits_{\R} (\psi_{2^{-j}} * \nu_\lambda)(x)\,d\nu_\lambda(x)\,\rho(\lambda)\,d\lambda\\
&\leq \int \limits_{\R} \sum \limits_{j=0}^\infty 2^{2j\gamma} \Bigl( 2^{j} \int\limits_{\Omega} \int\limits_{\Omega} \psi\bigl(2^j(\Pi^\lam(\omega_1) - \Pi^\lam (\omega_2))\bigr)\,e_j(\omega_1,\omega_2,\lambda)\, d\mu_{\lambda_0}(\omega_1)\,d\mu_{\lambda_0}(\omega_2)\\
&\qquad \qquad \qquad + 4L 2^{(2 + Q\tc\xi-\tc\kappa)j} \Bigr)\rho(\lambda)\,d\lambda\\
&\leq \int \limits_{\R} \sum \limits_{j=0}^\infty 2^{2j\gamma} 2^{j} \int\limits_{\Omega} \int\limits_{\Omega} \psi\bigl(2^j(\Pi^\lam(\omega_1) - \Pi^\lam(\omega_2))\bigr)\, e_j(\omega_1,\omega_2,\lambda) \,d\mu_{\lambda_0}(\omega_1)\, d\mu_{\lambda_0}(\omega_2)\,\rho(\lambda)\,d\lambda\\
&\qquad \qquad \qquad \qquad +\int_\R\sum\limits_{j=0}^\infty4L 2^{-j} \rho(\lambda)\,d\lambda\\
&\lesssim  \sum \limits_{j=0}^\infty 2^{j(2\gamma+1)} \int\limits_{\Omega} \int\limits_{\Omega} \int \limits_{\R}\psi\bigl(2^j(\Pi^\lam(\omega_1) - \Pi^\lambda (\omega_2))\bigr)\,e_j(\omega_1,\omega_2,\lambda)\,\rho(\lambda)\,d\lambda\, d\mu_{\lambda_0}(\omega_1)\, d\mu_{\lambda_0}(\omega_2).
\end{split}
\]
To finish the proof of Theorem \ref{thm:sobolev integral bound}, it is enough to show the following proposition (with notation as in Theorem \ref{thm:sobolev integral bound}). Recall that $\xi$ is chosen by requiring \eqref{eq:eta} and $J$ is an open interval containing $\lam_0$ so small that $2c|J|^\theta\leq\xi$ (with $c,\theta$ as in \ref{as:measure}) and \eqref{tran beta} hold.

\begin{prop}\label{prop:trans int}
	There exists {$ C_7 > 0$} such that for any distinct $\om_1,\om_2\in \Om$, any $j\in \N$  we have 
	\begin{equation}\label{claim1}
	\int\limits_{\R}\psi\bigl(2^j(\Pi^\lam(\omega_1) - \Pi^\lam(\omega_2))\bigr)\,e_j(\omega_1,\omega_2,\lambda)\,\rho(\lambda)\,d\lambda \leq  
	C_7\cdot \tc j2^{Q(2+c)\xi\tc j} \left( 1 + 2^j d(\om_1,\om_2)^{1+a_0\beta}\right)^{-q},\\
	\end{equation}
	where 
	$C_7$ depends only on $q, \rho$, and $\beta$, and $a_0 = \frac{8 + 4\delta}{1 + \min\{\delta, \theta\}}$, and $d(\om_1,\om_2) = d_{\lam_0}(\om_1,\om_2)$ is the metric defined in \eqref{eq:metric_def_lam}.
\end{prop}

Indeed, if \eqref{claim1} holds, then, {recalling the definition of energy \eqref{eq:energy}}, in view of \eqref{eq:forn},
\[
\begin{split}
\int_{J}\|&\nu_{\lambda}\|^2_{2,\gamma}\,\rho(\lambda)\,d\lambda \\
& \lesssim \sum \limits_{j=0}^\infty 2^{j(2\gamma+1)} \int\limits_{\Omega} \int\limits_{\Omega} \int \limits_{\R}\psi\bigl(2^j(\Pi^\lam(\omega_1) - \Pi^\lambda (\omega_2))\bigr)\,e_j(\omega_1,\omega_2,\lambda)\,\rho(\lambda)\,d\lambda\, d\mu_{\lambda_0}(\omega_1)\,d\mu_{\lambda_0}(\omega_2)\\
&\leq  
C_7\cdot\tc\sum_{j=0}^\infty2^{j(2\gamma+1)}j2^{Q(2+c)\xi \tc j}\int_\Omega\int_\Omega\left( 1 + 2^j d(\om_1,\om_2)^{1+a_0\beta}\right)^{-q}\,d\mu_{\lambda_0}(\omega_1)\,d\mu_{\lambda_0}(\omega_2)\\
&\leq  
C_7\cdot\tc\sum_{j=0}^\infty j2^{j(2\gamma+Q(2+c)\xi \tc+1-q)}\Ek_{q(1+a_0\beta )}(\mu_{\lambda_0}, d_{\lam_0})\\
&\leq  
C_7\cdot \tc \sum_{j=0}^\infty j2^{j(1 + 2\gamma+\frac{\eps}{2}-q)}\Ek_{q(1+a_0\beta)}(\mu_{\lambda_0}, d_{\lam_0}) \leq  
C_7\cdot \tc\sum_{j=0}^\infty j2^{-\frac{\eps}{2}j}\Ek_{q(1+a_0\beta)}(\mu_{\lambda_0}, d_{\lam_0}),
\end{split}\]
and Theorem \ref{thm:sobolev integral bound} is proved.
\begin{proof}[Proof of Proposition \ref{prop:trans int}] {The proof is similar to that of \cite[Lemma 4.6]{PS00} in the case of limited regularity; however, some technical issues are treated here 
differently and in more detail, especially, since \cite{PS00} leaves much to the reader.}

Fix distinct $\om_1,\om_2\in \Om$ and denote $r = d(\om_1,\om_2)$. For short, let $e_j(\lambda):=e_j(\omega_1,\omega_2,\lambda)$. 
Let $\ov I = \supp(\rho)\subset J$. Since $J$ is open, there exists $K=K(\rho)\ge 1$ such that the $(2K^{-1})$-neighborhood of $\ov I$ is contained in $J$.

We can assume without loss of generality that $2^j r>1$, and later that $2^j r^{1+a_0\beta} > 1$ for a fixed $a_0$, which is stronger, since $r\le 1$.
Indeed, the integral in \eqref{claim1} is bounded above by $|J|\cdot \|\psi\|_\infty\cdot 2^{Q \xi \tc j}$, { in view of \eqref{eq:errorest}}, hence if $2^j r^{1+a_0\beta} \le 1$, then the inequality \eqref{claim1} holds with 
$
C_7= |J|\cdot \|\psi\|_\infty \cdot 2^q$.

Let 
$$
\phi\in C^\infty(\R),\ \ \ {0\le \phi\le 1},\ \ \phi\equiv 1\ \mbox{on}\ \ [-1/2,1/2],\ \ \supp(\phi)\subset (-1,1),
$$
and denote
$$\Phi_\lam = \Phi_\lam(\om_1,\om_2) := \frac{\Pi^\lam(\om_1) - \Pi^\lam(\om_2)}{d(\om_1,\om_2)}{= \frac{\Pi^\lam(\om_1) - \Pi^\lam(\om_2)}{r}.}$$
{The idea, roughly speaking, is to separate the contribution of the zeros of $\Phi_\lam$, which are simple by transversality. Outside of a neighborhood of these zeros, we get an estimate using the rapid decay of $\psi$ at infinity, and near the zeros we linearize and use the fact that $\psi$ has zero mean. The details are quite technical, however.}
We have
\begin{eqnarray*}
	\int_\R \rho(\lam)\, \psi\!\left(2^j [\Pi^\lam(\om_1) - \Pi^\lam(\om_2)]\right)e_j(\lambda) d\lam  & = & \int \rho(\lam) \,\psi\!\left(2^j r \Phi_\lam \right)e_j(\lambda)\, \phi(K c_\beta^{-1} r^{-\beta} \Phi_\lam)\,d\lam \nonumber
	\\
	& + & \int \rho(\lam) \,\psi\!\left(2^j r \Phi_\lam\right)e_j(\lambda) \left[1-\phi(K c_\beta^{-1} r^{-\beta} \Phi_\lam)\right] d\lam  \label{integ1} \\[1.1ex]
	& =: & A_1 + A_2, \nonumber
\end{eqnarray*}
where $c_\beta$ is the constant from \eqref{tran beta}.
The integrand in $A_2$ is constant zero when $|K c_\beta^{-1} r^{-\beta} \Phi_\lam| \le \half$, hence by the rapid decay of $\psi$ {(see \eqref{eq:psi decay}) and  \eqref{eq:errorest}},
$$
|A_2| \le 
C_q \int |\rho(\lam)| |e_j(\lambda)| \bigl(1 + 2^j r \cdot {\textstyle{\half}} K^{-1} c_\beta r^\beta)^{-q}\, d\lam \le 
{\r\const\cdot} 2^{Q\xi \tc j} \bigl(1 + 2^j r^{1+\beta}\bigr)^{-q},
$$
{for some constant depending on $q, \rho$ and $\beta$}, as desired. Thus it remains to estimate $A_1$. 

Next comes the classical ``transversality lemma''. It is a variant of \cite[Lemma 4.3]{PS00} and similar to \cite[Lemma 18.12]{Mattila Fourier}. Let $c_\beta$ be the constant from Proposition \ref{prop_beta trans}.

\begin{lem}\label{lem-trans} 
	Under the assumptions and notation above, let
	$$
	\Jk:= \Bigl\{\lam\in J: \ |\Phi_\lam| < K^{-1} c_\beta r^{\beta} \Bigr\},
	$$
	which is a union of open disjoint intervals. Let $I_1,\ldots, I_{N_\beta}$ be the intervals of $\Jk$ intersecting $\ov I = \supp(\rho)$, enumerated in the order of $\R$. Then each $I_k$ contains a unique zero $\ov\lam_k$ of $\Phi_\lam$ and
	\be \label{inter0}
	[\ov\lam_k - d_\beta r^{2\beta}, \ov\lam_k + d_\beta r^{2\beta}]\subset I_k, \ \ \mbox{where}\ d_\beta =K^{-1} C_{\beta,1}^{-1} \cdot c_\beta,
	\ee
	{ with $C_{\beta,1}$ from \eqref{eq:4.21.1}.}
	For all intervals,
	\be \label{inter1}
	2d_\beta\cdot r^{2\beta}\le |I_k| \le 2K^{-1},
	\ee
	hence
	\be \label{up1}
	N_\beta \le 2 + \textstyle{\half}d_\beta^{-1}|J| \cdot r^{-2\beta}.
	\ee
	Moreover,
	\be \label{low2}
	|\Phi_\lam| \le \textstyle{\half} K^{-1} c_\beta r^\beta\ \ \ \mbox{for all}\ \ \ \lam\in [\ov\lam_k - \textstyle{\half} d_\beta r^{2\beta}, \ov\lam_k + \textstyle{\half} d_\beta r^{2\beta}].
	\ee
	\end{lem}

\begin{proof}[Proof Lemma \ref{lem-trans}]
	Since $\Phi_\lam$ is continuous, the intervals $I_k$ are well-defined. Since $K\ge 1$, on each of the intervals we have $|\frac{d}{d\lam}\Phi_\lam| \ge c_\beta r^{\beta}$ by the transversality condition \eqref{tran beta} {of degree $\beta$}.
	Thus $\Phi_\lam$ is strictly monotonic on each of the intervals. Let $\lam\in I_k {\cap} I$, where $\ov  I = \supp(\rho)$. Then $|\Phi_\lam| < K^{-1} c_\beta r^\beta$, and using the lower bound on the derivative we obtain that there exists unique
	$\ov \lam_k\in I_k$, such that $\Phi_{\ov \lam_k} = 0$, and it satisfies $|\lam - \ov \lam_k| \le K^{-1}$. The same argument then shows that $I_k \subseteq [\ov \lam_k  - K^{-1}, \ov \lam_k + K^{-1}]$, since
	the $K^{-1}$ change in $\lam$ results in at least $K^{-1} c_\beta r^\beta$ change in $\Phi_\lam$. Note that even for $k=1$ and $k=N_\beta$ we have this inclusion, because $\lam\in I$ and
	the $2 K^{-1}$-neighborhood of $I$ is contained in $J$ by construction.
	This proves the upper bound in \eqref{inter1}.
	
	On the other hand, for any $\lam\in J$ we have $|\frac{d}{d\lam}\Phi_\lam| \le C_{\beta,1} r^{-\beta}$ by \eqref{eq:4.21.1}. 
	Therefore, at least a distance of $C^{-1}_{\beta,1} r^{\beta} t$ is required for the graph of $\Phi_\lam$ to reach the level of $t$ from zero. This implies \eqref{inter0}, \eqref{low2} and the lower bound in \eqref{inter1}.
	Then \eqref{up1} is immediate.
\end{proof}

Now let $\chi \in C^\infty(\R)$ be such that $\supp(\chi) \subset (-\half d_\beta, \half d_\beta)$, {$0\le \chi\le 1$}, and $\chi\equiv 1$ on $[-\frac{1}{4} d_\beta, \frac{1}{4} d_\beta]$. 
We apply Lemma~\ref{lem-trans} and write
\begin{eqnarray*}
	A_1 & = & \int \rho(\lam) \,\psi\!\left(2^j r \Phi_\lam \right)e_j(\lambda) \phi(Kc_\beta^{-1} r^{-\beta} \Phi_\lam)\,d\lam \\
	& = & \sum_{k=1}^{N_\beta} \int \rho(\lam) \,\chi\bigl(r^{-2\beta} (\lam- \ov\lam_k)\bigr) \,\psi(2^j r\Phi_\lam)e_j(\lambda) \,\phi(Kc_\beta^{-1} r^{-\beta} \Phi_\lam)\,d\lam \\
	& + & \int \rho(\lam) \left[ 1- \sum_{k=1}^{N_\beta} \chi\bigl(r^{-2\beta} (\lam- \ov\lam_k)\bigr)\right]e_j(\lambda) \psi(2^j r\Phi_\lam) \, \phi(Kc_\beta^{-1} r^{-\beta} \Phi_\lam)\,d\lam\\
	& = & \sum_{k=1}^{N_\beta} A^{(k)}_1 + B.
\end{eqnarray*}

Let us first estimate $B$. Notice that $\sum_{k=1}^{N_\beta} \chi\bigl(r^{-2\beta} (\lam- \ov\lam_k)\bigr)\equiv 1$ on the $\frac{1}{4} d_\beta\,r^{2\beta}$-neighborhood of every $\ov \lam_k$, as by \eqref{inter0}, functions $\chi\bigl(r^{-2\beta} (\lam- \ov\lam_k)\bigr)$ have disjoint supports for distinct $k$.
On the other hand, $\phi(Kc_\beta^{-1} r^{-\beta} \Phi_\lam)$ is supported on $\Jk$, so by the transversality condition we have $|\frac{d}{d\lam}\Phi_\lam| \ge c_\beta r^\beta$ on the support of the integrand. Combining these two claims, we obtain that $|\Phi_\lam| \ge \frac{1}{4} d_\beta c_\beta r^{3\beta}$ on the support of the integrand in $B$.
It follows that on this support,
\be \label{add1}
|\psi(2^j r\Phi_\lam)| \le C_q \bigl( 1 + {(d_\beta c_\beta/4)}\cdot 2^j r^{1 + 3\beta}\bigr)^{-q},
\ee
by the rapid decay of $\psi$, and using \eqref{eq:errorest} we obtain {$|B| \le \const\cdot 2^{Q\xi \wt c j}\bigl( 1 + 2^j r^{1 + 3\beta}\bigr)^{-q}$ for some constant depending on $q$ and $\beta$.}

Now we turn to estimating the integrals $A^{(k)}_1$.
For simplicity, we  assume $k=1$ and let $\ov\lam= \ov\lam_1$. In view of the bound \eqref{up1} on the number of intervals, 
the desired inequality will follow from this. 
Observe that
\be \label{cond5}
\chi\bigl(r^{-2\beta} (\lam- \ov\lam)\bigr) = \chi\bigl(r^{-2\beta} (\lam- \ov\lam)\bigr) \phi(K c_\beta^{-1} r^{-\beta} \Phi_\lam).
\ee
We are using here that 
$\phi\equiv 1$ on $[-\half,\half]$, so $$\phi(Kc_\beta^{-1} r^{-\beta} \Phi_\lam)\equiv 1\ \ \mbox{on}\  \
\bigl\{\lam\in J:\ |\Phi_\lam| \le  {\textstyle{\half}} K^{-1} c_\beta r^\beta\bigr\},$$
which holds on the support of $\chi\bigl(r^{-2\beta} (\lam- \ov\lam)\bigr)$ by construction and \eqref{low2}.

By \eqref{cond5} we have
$$
A_1^{(1)} = \int \rho(\lam) \,\chi\bigl(r^{-2\beta} (\lam- \ov\lam)\bigr)e_j(\lambda) \,\psi(2^j r\Phi_\lam) \,d\lam.
$$
It will be convenient to make a change of variable, so we define a function $H$ via
\be \label{def-H}
\Phi_\lam = u \iff \lam = \ov \lam +H(u),\ \ \mbox{provided}\ \ \chi\bigl(r^{-2\beta} (\lam- \ov\lam)\bigr)\ne 0.
\ee
Note that $\chi\bigl(r^{-2\beta} (\lam- \ov\lam)\bigr)\ne 0$ implies $|\lam - \ov\lam| < \frac{1}{2} d_\beta r^{2\beta}$, so $\lam\in I_1$ by \eqref{inter0}, and by transversality, 
\be\label{trans on chi}
\Bigl|\frac{d}{d\lam} \Phi_\lam\Bigr| \ge c_\beta r^\beta\ \text{ if }\ \chi\bigl(r^{-2\beta} (\lam- \ov\lam)\bigr)\ne 0.
\ee
Therefore, $H$ is well defined. We have
\begin{eqnarray*}
	A_1^{(1)} & = & \int \rho\bigl(\ov\lam + H(u)\bigr)\,\chi\bigl(r^{-2\beta} H(u) \bigr)e_j(\ov\lam+H(u)) \,\psi(2^j ru) \,H'(u)\,du \\
	& = & \int F(u)\,\psi(2^j ru) \,du,
\end{eqnarray*}
where
\be \label{def-F}
F(u) = \rho\bigl(\ov\lam + H(u)\bigr)\,\chi\bigl(r^{-2\beta} H(u) \bigr)e_j(\ov\lam+H(u))  \,H'(u).
\ee
Observe that $H'(u) = [\frac{d}{d\lam} \Phi_\lam]^{-1}$, hence \eqref{trans on chi} gives $|H'(u)|\le c_\beta^{-1} r^{-\beta}$ on the domain of $F$. Since $\rho$ and $\chi$ are bounded {by one}, we obtain by \eqref{eq:errorest}
\be \label{Fbound}
\|F\|_\infty \le {c^{-1}_\beta} \cdot r^{-\beta}2^{Q\xi \tc j}.
\ee

Recall that $\Phi_{\ov\lam} = 0$, so that $H(0)=0$.
Since $\int_\R \psi(\xi)\,d\xi = 0$ by \eqref{meanzero}, we can subtract $F(0)$ from $F(u)$ under the integral sign; we then split the integral as follows:
\begin{eqnarray}
A_1^{(1)} & = & \int [F(u) - F(0)]\, \psi(2^j ru) \,du \nonumber \\
& = & \int_{|u| < (2^j r)^{-1+\eps'}} [F(u) - F(0)]\, \psi(2^j ru) \,du + \int_{|u| \ge (2^j r)^{-1+\eps'} }[F(u) - F(0)]\, \psi(2^j ru) \,du \label{wehave} \\[1.2ex]
& =: & B_1 + B_2, \nonumber
\end{eqnarray}
where $\eps' \in (0, \frac{1}{2})$ is a small fixed number. Recall that our goal is to show
$$
|A^{(1)}_1| \le {C_7' \cdot \tc j2^{Q(2+c)\xi\tc j} } \cdot \bigl(1 + 2^j r^{1 + a_0\beta}\bigr)^{-q},
$$
for some constants $a_0\ge 1$ and {$C_7'$ depending only on $q$, $\rho$, and $\beta$. }We can assume that $2^j r^{1 + a_0\beta} \ge 1$, otherwise, the estimate is trivial by increasing the constant. To estimate $B_2$,  note that  for any $M>0$ we have by the rapid decay of $\psi$:
$$
|\psi(2^j ru)| \le C_M\bigl( 1 + 2^j r|u|\bigr)^{-M},
$$
hence, by \eqref{Fbound},
\begin{eqnarray*}
	|B_2|  & \le & C_{\beta,M} \cdot r^{-\beta} \cdot 2^{Q\xi \tc j}(2^j r)^{-1} \int_{|x|\ge (2^j r)^{\eps'}} (1 + |x|)^{-M}\,dx  \\[1.1ex]
	& \le &  C_{\beta,M}' \cdot  r^{-\beta} \cdot 2^{Q\xi \tc j}(2^j r)^{-1} (2^j r)^{-\eps'(M-1)} \\[1.1ex]
	& \le & C_{\beta,M}''\cdot 2^{Q\xi \tc j} \cdot (2^j r^{1+2\beta})^{-q},
\end{eqnarray*}
for $M = M(q,\eps')$ sufficiently large, as desired. Here we used that $2^j r \ge 2^j r^{1 + 2\beta} \geq 1$.

\smallskip

In order to estimate $B_1$, we show that that the function $F$ from \eqref{def-F} is $\delta$-H\"older by our assumptions; we also need to estimate the constant in the H\"older bound. 
We write
$$
F(u) = \rho\bigl(\ov\lam + H(u)\bigr)\,\chi\bigl(r^{-2\beta} H(u) \bigr)e_j(\ov\lam+H(u))  \,H'(u) =: F_1(u) F_2(u)F_3(u) H'(u),
$$
and then 
\begin{multline*}
F(u) - F(0) = \bigl(F_1(u) - F_1(0)\bigr) F_2(u)F_3(u) H'(u) + F_1(0) \bigl(F_2(u) - F_2(0)\bigr)F_3(u) H'(u) \\
+ F_1(0)F_2(0) \bigl(F_3(u) - F_3(0)\bigr) H'(u) + F_1(0) F_2(0)F_3(0) \bigl(H'(u) - H'(0)\bigr).
\end{multline*}
We have
$$
|F_1(u) - F_1(0)| = |\rho\bigl(\ov\lam + H(u)\bigr) - \rho\bigl(\ov\lam + H(0 )\bigr)| \le \|\rho'\|_\infty \cdot |H(u) - H(0)|.
$$
Observe that
\be \label{eta1}
|H(u) - H(0)| = |H(u)| = |\lam - \ov\lam| \le c^{-1}_\beta r^{-\beta} |\Phi_\lam-\Phi_{\ov\lam}| = c^{-1}_\beta r^{-\beta} |\Phi_\lam| = c^{-1}_\beta r^{-\beta} |u|,
\ee
by transversality, which applies since $\supp(F)\subset I_1$.
Then, of course,
\be \label{eta2}
|F_2(u) - F_2(0)| \le \|\chi'\|_\infty \cdot r^{-2\beta} |H(u) - H(0)| \le C^{-1}_\beta \|\chi'\|_\infty \cdot r^{-3\beta} |u|.
\ee

For $F_3$ it is enough to assume that $\mu_{\lam_0}([\omega_1|_{\tc j}])\mu_{\lam_0}([\omega_2|_{\tc j}]) \neq 0$ (hence the same is true for $\mu_{\ov\lam}$ by \ref{as:measure}), as otherwise $e_j \equiv 1$ and then \eqref{eq:F_3 hoelder}, which is the goal of the calculation below, holds trivially. In this case we have
\[
\begin{split}
|F_3(u)-F_3(0)|&=\frac{\mu_{\ov\lam}([\omega_1|_{\tc j}])\mu_{\ov\lam}([\omega_2|_{\tc j}])\left|\frac{\mu_{\ov\lam+H(u)}([\omega_1|_{\tc j}])\mu_{\ov\lam+H(u)}([\omega_2|_{\tc j}])}{\mu_{\ov\lam}([\omega_1|_{\tc j}])\mu_{\ov\lam}([\omega_2|_{\tc j}])}-1\right|}{\mu_{\lambda_0}([\omega_1|_{\tc j}])\mu_{\lambda_0}([\omega_2|_{\tc j}])}\\[1.1ex]
&\leq 2^{Q\xi \tc j}\left|\frac{\mu_{\ov\lam+H(u)}([\omega_1|_{\tc j}])\mu_{\ov\lam+H(u)}([\omega_2|_{\tc j}])}{\mu_{\ov\lam}([\omega_1|_{\tc j}])\mu_{\ov\lam}([\omega_2|_{\tc j}])}-1\right|\\[1.1ex]
&\leq 2^{Q\xi \tc j}\frac{\mu_{\ov\lam+H(u)}([\omega_1|_{\tc j}])}{\mu_{\ov\lam}([\omega_1|_{\tc j}])}\left|\frac{\mu_{\ov\lam+H(u)}([\omega_2|_{\tc j}])}{\mu_{\ov\lam}([\omega_2|_{\tc j}])}-1\right|+2^{Q\xi \tc j}\left|\frac{\mu_{\ov\lam+H(u)}([\omega_1|_{\tc j}])}{\mu_{\ov\lam}([\omega_1|_{\tc j}])}-1\right|\\[1.1ex]
&\leq 2^{2Q\xi \tc j}\left|\frac{\mu_{\ov\lam+H(u)}([\omega_2|_{\tc j}])}{\mu_{\ov\lam}([\omega_2|_{\tc j}])}-1\right|+2^{Q\xi \tc j}\left|\frac{\mu_{\ov\lam+H(u)}([\omega_1|_{\tc j}])}{\mu_{\ov\lam}([\omega_1|_{\tc j}])}-1\right|
\end{split}
\]
But for both $\omega_1|_{\tc j}$ and $\omega_2|_{\tc j}$, setting ${c_3 = Qc \tc}$, we obtain
\[
\begin{split}
\left|\frac{\mu_{\ov\lam+H(u)}([\omega|_{\tc j}])}{\mu_{\ov\lam}([\omega|_{\tc j}])}-1\right|&\leq\max\{2^{{c_3} j|H(u)|^\theta}-1,1-2^{-{c_3} j|H(u)|^\theta}\}\\
&=\max\{{c_3} j2^{{c_3} j y_1},{c_3} j2^{-{c_3} j y_2}\}|H(u)|^\theta,\\
& \qquad \qquad \text{ with }y_1\in(0,|H(u)|),y_2\in(-|H(u)|,0)\\
&\leq {c_3} j2^{{c_3} j \xi}|H(u)-H(0)|^\theta.
\end{split}\]	
Thus, for ${c_4 =Q(2+c)\tc \xi}$
\begin{equation}\label{eq:F_3 hoelder}
|F_3(u)-F_3(0)|\leq 2 {c_3} j2^{{c_4} j}c_\beta^{-\theta}r^{-\theta\beta}|u|^{\theta}.
\end{equation}

Finally, we need to estimate the term $|H'(u) - H'(0)|$. We have $H'(u) = [\frac{d}{d\lam} \Phi_\lam]^{-1}$, hence
\begin{eqnarray*}
	|H'(u) - H'(0)| & = & \left| \frac{1}{\frac{d}{d\lam} \Phi_\lam} - \frac{1}{\frac{d}{d\lam} \Phi_{\ov \lam}}\right| \\[1.2ex]
	& \le  & \frac{|\frac{d}{d\lam} \Phi_\lam -  \frac{d}{d\lam} \Phi_{\ov \lam}|}{(c_\beta r^{\beta})^2}\ \ \ \hspace{10mm} \mbox{by $\beta$-transversality \eqref{tran beta}} \\[1.1ex]
	& \le  & \frac{C_{\beta,1}|\lam-\ov\lam|^\delta r^{-\beta(1+\delta)}}{(c_\beta r^\beta)^2}  \ \ \hspace{20mm} \mbox{by \eqref{eq:4.21.2}} \\[1.1ex]
	& \le & \wt c_\beta r^{-\beta(3 + 2\delta)} |u|^\delta \hspace{31mm} \mbox{by \eqref{eta1}.}
\end{eqnarray*}
{Below, writing ``$\const$'' means constants depending on $q,\rho$, and $\beta$, which may be different from line to line.}
Using all of the above and $\|H'\|_\infty \le c_\beta^{-1}\cdot r^{-\beta}$ yields
$$
|F(u) - F(0)| \le \const\cdot {c_3} j2^{{c_4} j}\cdot \left( |u|^\delta r^{-\beta(3+2\delta)} + |u| r^{-4\beta}+|u|^{\theta}r^{-\beta(1+\theta)}\right),
$$
hence by \eqref{wehave} {and recalling that $(2^j r)\ge 1$ and $r\le 1$},
\begin{eqnarray*}
	|B_1| & \le & \const\cdot {c_3} j2^{{c_4} j} \int_{|u| < (2^j r)^{-1+\eps'}} \left( |u|^\delta r^{-\beta(3+2\delta)} + |u| r^{-4\beta}+|u|^{\theta}r^{-\beta(1+\theta)}\right)\,du \\[1.2ex]
	& \le & \const\cdot {c_3} j2^{{c_4} j}\left( r^{-\beta(3+2\delta)} (2^j r)^{-(1-\eps')(1+\delta)} + (2^j r)^{-2(1-\eps')} r^{-4\beta}+(2^j r)^{-(1-\eps')(1+\theta)}r^{-\beta(1+\theta)}\right)\\
	 & \leq & {\const\cdot {c_3} j2^{{c_4} j}r^{-\beta(4+2\delta)} \left( (2^j r)^{-(1-\eps')(1+\delta)} + (2^j r)^{-2(1-\eps')} +(2^j r)^{-(1-\eps')(1+\theta)}\right)} \\
	& \leq & \const \cdot {c_3}  j2^{{c_4} j}r^{-\beta(4+2\delta)}(2^jr)^{-(1-\eps')(1+\min\{\delta, \theta\})},
\end{eqnarray*}
{as $\min\{\delta, \theta\} \leq 1$. We therefore obtain}
$$
|B_1| \le \const\cdot {c_3} j2^{{c_4} j}\left(2^j r^{1+a_0\beta}\right)^{-(1-\eps')(1+\min\{\delta, \theta\})},
$$
for appropriate $a_0= \frac{8+4\delta}{1+\min\{\delta, \theta\}} \geq \frac{4+2\delta}{(1-\eps')(1+\min\{\delta, \theta\})}$.

Since $\eps'>0$ can be chosen arbitrarily small, we obtain
\[|B_1| \leq \const \cdot {c_3} j2^{{c_4} j}\left(1+2^j r^{1+a_0\beta}\right)^{-q} \text{ for any } q< 1+\min\{\delta,\theta\},\]
since as already mentioned, we can assume $2^j r^{1+a_0\beta} \ge 1$ without loss of generality.
\end{proof}


\section{The case of Gibbs measures}

In this section we deal with case of Gibbs measures and develop tools for the derivation of Theorem \ref{thm:main_gibbs} from Theorem \ref{thm:main_cor_dim}. Throughout this section, we assume that $\left\{ \mu_\lam \right\}_{\lam \in \ov{U}}$ is a family of shift-invariant Gibbs measures on $\Om$ corresponding to a family of continuous potentials $\phi^\lambda\colon\Om\mapsto\R$ satisfying \eqref{eq:unifvar} and \eqref{eq:assong}; $\alpha,\ b,\ c_0$ and $\theta$ denote constants from \eqref{eq:unifvar} and \eqref{eq:assong}.
\subsection{Proving \ref{as:measure} for Gibbs measures}\label{subsec:M for gibbs}

Let $L_\lambda$ be the Perron operator on the Banach space $C(\Omega)$ of continuous functions on $\Omega$, defined as 
$$
(L_{\lambda}h)(\omega)=\sum_{i\in\Ak}e^{\phi^\lambda(i\omega)}h(i\omega).
$$
Let $C_r$ be the set of functions which are constant over cylinder sets of length $r$, that is,
$$
C_r(\Omega)=\{f\in C(\Om):\mathrm{var}_r(f)=0\}.
$$
Let $\omega\in\Omega$ be arbitrary but fixed and denote the pressure by 
$$
P_\lambda=\lim_{n\to\infty}\frac{1}{n}\log\sum_{|\ii|=n}\exp\left(S_n\phi^\lambda(\ii\omega)\right),
$$
where $S_n\phi(\omega)=\phi(\omega)+\phi(\sigma\omega)+\cdots+\phi(\sigma^{n-1}\omega)$. Note that the value of $P_\lambda$ is independent of the choice of $\omega\in\Omega$.

\begin{thm}\label{thm:bowen}
	There exists $c_2>0$ such that for every $\lambda\in \ov{U}$ there is a unique $h_\lambda\in C(\Om)$ with $h_\lambda>c_2>0$ and $\nu_\lambda\in\mathcal{P}(\Om)$ such that
	$$
	L_\lambda h_\lambda=\gamma_\lambda h_\lambda, \ \ (L_\lambda)_*\nu_\lambda=\gamma_\lambda\nu_\lambda,\ \text{ and }\int h_\lambda d\nu_\lambda=1,
	$$
	where $\gamma_\lambda=\exp(P_\lambda)$.	Moreover, for every $\om_1,\om_2\in\Omega$ and $\lambda\in \ov{U}$,
	$$
	h_\lambda(\om_1)\leq B_{\om_1\wedge\om_2} h_\lambda(\om_2),
	$$
	where $B_m=\exp\left(\sum_{k=m+1}^\infty 2b\alpha^k\right)$.
	
	Furthermore, there exist $A>0$ and $0<\beta<1$ such that for every $f\in C_r(\Omega)$,
	$$
	\Bigl\|\gamma_{\lambda}^{-n-r}L_\lambda^{n+r}f-\int fd\nu_\lambda\cdot h_\lambda\Bigr\|\leq A\beta^n\int fd\nu_\lambda\ \text{ for every }\ \lambda\in \ov{U}\text{ and }n\geq1.
	$$
\end{thm}

\begin{proof}
See \cite[Theorem 1.7, Lemmas 1.8 and 1.12]{Bow} and their proofs.
\end{proof}

The measure $d\mu_\lambda=h_\lambda d\nu_\lambda$ is a left-shift invariant ergodic Gibbs measure with respect to the potential $\phi^\lambda$, see \cite[Theorem~1.16, Proposition~1.14]{Bow}.

We will show that $\gamma_\lambda, h_\lambda$ and $\nu_\lambda$ depend uniformly continuously on the parameters in the following sense: 
\begin{lem}\label{lem:eigen}
For every $0 < \theta' < \theta$, there exists $c_{\theta'}>0$ such that for every $\lambda,\tau\in U$, 
	$$
	\frac{\gamma_\lambda}{\gamma_\tau},\ \frac{h_\lambda(\omega)}{h_\tau(\omega)}\leq e^{c_{\theta'}|\lambda-\tau|^{\theta'}}\text{ for every $\omega\in\Omega$}.
	$$
	For every $\ii\in\Om^*$,
	$$
	\frac{\nu_\lambda([\ii])}{\nu_\tau([\ii])}\leq e^{c_{\theta'}|\lambda-\tau|^{\theta'}|\ii|}.
	$$
Moreover, the constant $C_G$ in the definition of the Gibbs measure can be chosen uniformly for $\lam \in \ov{U}$.
\end{lem}

\begin{proof}
	Simple calculations show that $|P_\lambda-P_\tau|\leq c_0|\lambda-\tau|^\theta$ by \eqref{eq:assong}, hence the claim on $\gamma_\lambda$. Now let us turn to the claim on the eigenfunctions $h_\lambda$. Denote by $\ind_\Omega$ the constant $1$ map over $\Omega$. 
	
	If $\lambda=\tau$, then there is nothing to prove. Suppose that $\lambda\neq\tau$. Then by Theorem~\ref{thm:bowen}, 
	\be\label{eq:eigenfunction bound}
	\left\|\frac{\gamma^{-n}_\lam L_\lambda^n\ind_\Omega}{h_\lambda}-1\right\|\leq c_2^{-1}\left\|\gamma^{-n}_\lam L_\lambda^n\ind_\Omega-h_\lambda\right\|\leq c_2^{-1}A\beta^n=:A'\beta^n.
	\ee
	Note that $L_\lam^n \mathds{1}_{\Om} (\om) = \sum \limits_{i_1, \ldots, i_n \in \Ak} e^{\phi^\lam(i_1 \ldots i_n \om)}$, hence \eqref{eq:assong} gives
	\[ \frac{L^n_\lam \mathds{1}_\Om (\om)}{L^n_\tau \mathds{1}_\Om (\om)} \leq e^{c_0|\lam - \tau|^\theta}.\]
	Combining this with \eqref{eq:eigenfunction bound} gives for every $n\geq1$,
	\[
	\begin{split}
	\frac{h_\lambda(\omega)}{h_\tau(\omega)}&=\frac{h_\lambda(\omega)}{\gamma_\lambda^{-n}(L_\lambda^n\ind_\Omega)(\omega)}\cdot\frac{\gamma_\tau^{-n}(L_\tau^n\ind_\Omega)(\omega)}{h_\tau(\omega)}\cdot\frac{\gamma_\lambda^{-n}(L_\lambda^n\ind_\Omega)(\omega)}{\gamma_\tau^{-n}(L_\tau^n\ind_\Omega)(\omega)}\\
	&\leq\frac{1+A'\beta^n}{1-A'\beta^n}\cdot\frac{\gamma_\tau^n}{\gamma_\lambda^n}\cdot\frac{(L_\lambda^n\ind_\Omega)(\omega)}{(L_\tau^n\ind_\Omega)(\omega)}\\
	&\leq\frac{1+A'\beta^n}{1-A'\beta^n}e^{2c_0|\lambda-\tau|^\theta n}.
	\end{split}
	\]
	Let $n$ be minimal such that $\frac{1+A'\beta^n}{1-A'\beta^n}\leq e^{2c_0|\lambda-\tau|^\theta}$, that is, let 
	\begin{equation}\label{eq:choosen}
	n=\left\lceil\frac{\log\left(1-2(e^{2c_0|\lambda-\tau|^\theta}+1)^{-1}\right)-\log A'}{\log\beta}\right\rceil.
	\end{equation}
	It is easy to see that for any $0<\theta'<\theta$,
	\begin{equation}\label{eq:forbound}
	\lim_{x\to0+}x^{\theta-\theta'}\log\Bigl(1-\frac{2}{e^{2c_0x^\theta}+1}\Bigr)=0,
	\end{equation}
	thus, there exists $c_{\theta'}>0$ such that
	$$
	2c_0|\lambda-\tau|^\theta \left(\frac{\log\left(1-2(e^{2c_0|\lambda-\tau|^\theta}+1)^{-1}\right)-\log A'}{\log\beta}+2\right)\leq c_{\theta'}|\lambda-\tau|^{\theta'}
	$$
	for every $\lambda,\tau\in U$. Hence,
	\[
	\begin{split}
	\frac{h_\lambda(\omega)}{h_\tau(\omega)}&\leq \frac{1+A'\beta^n}{1-A'\beta^n}e^{2c_0|\lambda-\tau|^\theta n}\\
	&\leq \exp\left( 2c_0|\lambda-\tau|^\theta (n+1) \right)\\
	&\leq \exp\left( 2c_0|\lambda-\tau|^\theta \left(\frac{\log\left(1-2(e^{2c_0|\lambda-\tau|^\theta}+1)^{-1}\right)-\log A'}{\log\beta}+2\right)\right)\\
	&\leq \exp\left( c_{\theta'}|\lambda-\tau|^{\theta'} \right).
	\end{split}
	\]
	
	The proof for the measure is similar. In fact, suppose that $\lambda\neq\tau$. Using Theorem~\ref{thm:bowen}, we get for every $n\geq1$ and every $\omega\in\Omega$,
	\[
	\begin{split}
	\frac{\nu_\lambda([\ii])}{\nu_\tau([\ii])}&=\frac{\nu_\lambda([\ii])h_\lambda(\omega)}{\gamma_\lambda^{-(n+|\ii|)}(L_\lambda^{n+|\ii|}\ind_{[\ii]})(\omega)}\cdot\frac{\gamma_\lambda^{-(n+|\ii|)}(L_\lambda^{n+|\ii|}\ind_{[\ii]})(\omega)}{\gamma_\tau^{-(n+|\ii|)}(L_\tau^{n+|\ii|}\ind_{[\ii]})(\omega)}\cdot\frac{\gamma_{\tau}^{-(n+|\ii|)}(L_\tau^{n+|\ii|}\ind_{[\ii]})(\omega)}{\nu_\tau([\ii])h_\tau(\omega)}\cdot\frac{h_\tau(\omega)}{h_\lambda(\omega)}\\
	&\leq \frac{1+A'\beta^n}{1-A'\beta^n}\cdot \exp \left( 2c_0|\lambda-\tau|^\theta(n+|\ii|) +  c_{\theta'}|\lambda-\tau|^{\theta'} \right).
	\end{split}
	\]
	Now, choose again $n\geq1$ as in \eqref{eq:choosen}. Then
	\[
	\begin{split}
	\frac{\nu_\lambda([\ii])}{\nu_\tau([\ii])}&\leq \exp \left(2c_0|\lambda-\tau|^\theta|\ii|+c_{\theta'}|\lambda-\tau|^{\theta'} + 2c|\lambda-\tau|\left(\frac{\log\left(1-2(e^{2c_0|\lambda-\tau|^\theta}+1)^{-1}\right)-\log A'}{\log\beta}+2\right)\right) \\
	&\leq \exp \left( 2c_0|\lambda-\tau|^\theta|\ii|+2c_{\theta'}|\lambda-\tau|^{\theta'}\right)\leq \exp \left( \widetilde{m}(2c+2c_{\theta'})|\lambda-\tau|^{\theta'}|\ii|\right)
	\end{split}
	\]
	for some constant $\widetilde{m} = \widetilde{m}(\theta, \theta')$.
	
	The claim on the Gibbs constant $C_G$ follows from the proof of \cite[Theorem 1.16]{Bow}, combined with uniform bounds on $h_\lam$ and $\gamma_{\lam}$.
\end{proof}

The following proposition concludes the proof of the property $\ref{as:measure}$ for Gibbs measures satisfying assumptions of Theorem \ref{thm:main_gibbs}.

\begin{prop}\label{eq:regeq}
For every $0 < \theta' < \theta$ there exists $c>0$ such that for every $\lambda,\tau\in U$ and for every $\ii\in\Om^*$,
	$$
	\frac{\mu_\lambda([\ii])}{\mu_\tau([\ii])}\leq e^{c|\lambda-\tau|^{\theta'}|\ii|}.
	$$
\end{prop}

\begin{proof} Fix $\theta'' \in (\theta' ,\theta)$.
	By the definition of $\mu_{\lam}$, Theorem \ref{thm:bowen} and Lemma~\ref{lem:eigen},
	\[
	\begin{split}
	\frac{\mu_\lambda([\ii])}{\mu_\tau([\ii])}&=\frac{\int_{[\ii]}h_\lambda(\omega)d\nu_\lambda(\omega)}{\int_{[\ii]}h_\tau(\omega)d\nu_\tau(\omega)}\\
	&\leq B_{n+|\ii|}^2\frac{\sum_{\jj:|\jj|=n}h_\lambda(\ii\jj\omega)\nu_\lambda([\ii\jj])}{\sum_{\jj:|\jj|=n}h_\tau(\ii\jj\omega)\nu_\tau([\ii\jj])}\\
	&\leq B_{n+|\ii|}^2\exp\left( c_{\theta''}|\lambda-\tau|^{\theta''}+c_{\theta''}(n+|\ii|)|\lambda-\tau|^{\theta''}\right).
	\end{split}\]
	Choose $n\geq1$ minimal such that
	$$
	B_{n+|\ii|}^2\leq B_n^2\leq e^{c_{\theta''}|\lambda-\tau|^{\theta''}},
	$$
	that is,
	$$
	n=\left\lceil\theta''\frac{\log|\lambda-\tau|}{\log\alpha}+\frac{(1-\alpha)c_{\theta''}(4b)^{-1}}{\log\alpha}\right\rceil.
	$$
	Then
	\[
	\begin{split}
	\frac{\mu_\lambda([\ii])}{\mu_\tau([\ii])}&\leq \exp \left( 2c_{\theta''}|\lambda-\tau|^{\theta''}+c_{\theta''}(\theta''\frac{\log|\lambda-\tau|}{\log\alpha}+\frac{(1-\alpha)c_{\theta''}(4b)^{-1}}{\log\alpha}+|\ii|+1)|\lambda-\tau|^{\theta''} \right).
	\end{split}\]
	Since for every $\theta'<\theta''<1$ the map $(\lambda,\tau)\mapsto |\lambda-\tau|^{\theta''-\theta'}\log|\lambda-\tau|$ is bounded, the claim follows.
\end{proof}

\subsection{Large submeasures}\label{subsec:regular egorov}

The goal of this subsection is to prove the following proposition, required to deduce Theorem \ref{thm:main_gibbs} from Theorem \ref{thm:main_cor_dim}.

\begin{prop}\label{prop:stable cor bound}
Let $\{f_j^\lam\}_{j \in \Ak}$ be a parametrized IFS satisfying the smoothness assumptions \ref{as:C2} - \ref{as:hyper}. Let $\left\{ \mu_\lam \right\}_{\lam \in \ov{U}}$ be a family of shift-invariant Gibbs measures on $\Om$ corresponding to a family of continuous potentials $\phi^\lambda\colon\Om\mapsto\R$ satisfying \eqref{eq:unifvar} and \eqref{eq:assong}. Then for every $\lam_0 \in U,\ \eps>0,\ \eps'>0$ and $\theta' \in (0, \theta)$ there exist $\xi>0,\ c>0$, and a set $A \subset \Om$ such that for every $\lam \in B_{\xi}(\lam_0)$ {we have} $\mu_\lam(A) \geq 1 - \eps'$ and the measures $\tilde{\mu}_\lam = \mu_\lam|_A$ satisfy
\begin{equation}\label{eq:dim cor entropy} \dim_{cor}(\tilde{\mu}_\lam, d_\lam) \geq \frac{h_{\mu_{\lam}}}{\chi_{\mu_{\lam}}} - \eps
\end{equation}
and
\begin{equation}\label{eq:tilde mu bounds}	e^{-c|\lambda-\lambda_0|^{\theta'}|\omega|} \tilde{\mu}_{\lambda}([\omega]) \leq\tilde{\mu}_{\lambda_0}([\omega])\leq e^{c|\lambda-\lambda_0|^{\theta'}|\omega|} \tilde{\mu}_{\lambda}([\omega])
\end{equation}
for all $\om \in \Om^*$.
\end{prop}
A standard approach for proving \eqref{eq:dim cor entropy} is applying Egorov's theorem, similarly as in the proof of Proposition \ref{prop:hdim_lam0_main}. In our case the difficulty is to obtain  \eqref{eq:tilde mu bounds} simultaneously. This requires a more quantitative approach in constructing ``Egorov-like'' set. For this purpose we need certain large deviations results, uniform with respect to the parameter, which we state in a slightly more general setting.

We assume now that $\left\{ \mu_\lam \right\}_{\lam \in \ov{U}}$ is a family of measures satisfying assumptions of Proposition \ref{prop:stable cor bound} and $\{ g_\ell^\lambda\colon\Om\mapsto\R\}_{\lam \in \ov{U}},\ \ell = 1, \ldots, p$, is a finite collection of families of potentials, each of them satisfying properties \eqref{eq:unifvar} and \eqref{eq:assong}. 

\begin{prop}\label{prop:largedev}
	Let $\lambda_0\in U$ be arbitrary but fixed. Then for every $\varepsilon>0$ there exists $\xi_D>0$, $C_D>0$ and $s>0$ such that for every $\lambda\in B_{\xi_D}(\lambda_0)$ and every $n\geq1,\ \ell = 1, \ldots, p$,
	$$
	\mu_\lambda\left(\left\{\omega\in\Om:\left|\frac{1}{n}S_ng_\ell^\lambda(\omega)-\int g_\ell^\lambda d\mu_\lambda\right|>\varepsilon\right\}\right)\leq C_{D} e^{-sn}.
	$$
\end{prop}

The proof is based on two lemmas.

\begin{lem}\label{lem:contint}
	For every $\theta' \in (0, \theta)$ and $\lam_0 \in \ov{U}$ there exist $\xi_{21}>0$ and $C_g = C_g(g_1, \ldots, g_l, \theta')>0$ such that
	$$
	\left|\int g_\ell^\lambda d\mu_\lambda-\int g_\ell^\tau d\mu_\tau\right|\leq C_g|\lambda-\tau|^{\theta'}.
	$$
	holds for every $\ell = 1, \ldots, p$ and $\lam \in B_{\xi_{21}}(\lam_0)$.
\end{lem}

\begin{proof}
	Fix $\theta'' \in (\theta', \theta)$ and let $c$ be the constant from Proposition \ref{eq:regeq} corresponding to $\theta''$. Let $\lambda\in U$ be arbitrary, and let $\tau\in B_{\xi_{21}}(\lambda)$ where $\xi_{21}$ is chosen such that $\alpha e^{c\xi_{21}^{\theta''}}<1$. Choose $n\geq1$ {minimal} such that $(\alpha e^{c|\lambda-\tau|^{\theta''}})^n\leq |\lambda-\tau|$. Then
	\[
	\begin{split}
	\int g_\ell^\lambda d\mu_\lambda&\leq b\alpha^n+\sum_{|\ii|=n}g_\ell^\lambda(\ii\omega)\mu_\lambda([\ii])\\
	&\leq b\alpha^n+c_0|\lambda-\tau|^\theta+\sum_{|\ii|=n}g_\ell^\tau(\ii\omega)\mu_\lambda([\ii])\\
	&\leq b\alpha^n+c_0|\lambda-\tau|^\theta+e^{cn|\lambda-\tau|^{\theta''}}\sum_{|\ii|=n}g_\ell^\tau(\ii\omega)\mu_\tau([\ii])\\
	&\leq b\alpha^n+c_0|\lambda-\tau|^\theta+e^{cn|\lambda-\tau|^{\theta''}}\left(\int g_\ell^\tau d\mu_\tau+b\alpha^n\right).
	\end{split}\]
	Thus,
	$$
	\left|\int g_\ell^\lambda d\mu_\lambda-\int g_\ell^\tau d\mu_\tau\right|\leq\left(e^{cn|\lambda-\tau|^{\theta''}}-1\right)M+b\alpha^n\left(e^{cn|\lambda-\tau|^{\theta''}}+1\right)+c_0|\lambda-\tau|^\theta.
	$$
	where $M=\max_{\lambda\in U,\omega\in\Om}|g_\ell^{\lambda}(\omega)|$. Hence, by the choice of $n$,
	$$
	\left|\int g_\ell^\lambda d\mu_\lambda-\int g_\ell^\tau d\mu_\tau\right|\leq\left(\exp \left(c|\lambda-\tau|^{\theta''}\log|\lambda-\tau|/(\log\alpha+c|\lambda-\tau|^{\theta''})\right)-1\right)M+(c_0+2)|\lambda-\tau|^\theta.
	$$
	The map $x\mapsto \frac{x^{\theta''-\theta'}\log x}{\log\alpha+cx^{\theta''}}$ is continuous, hence bounded, on $[0,\xi_{21}]$, say, by $B$. Further, there exists a constant $\widetilde{C}_1>0$ such that $|e^x-1|\leq \widetilde{C}_1|x|$ for every $|x|\leq B\xi_{21}^{\theta'}$. Hence, 
	$$
	\left|\int g_\ell^\lambda d\mu_\lambda-\int g_\ell^\tau d\mu_\tau\right|\leq (\widetilde{C}_1M+c_0+2)|\lambda-\tau|^{\theta'},
	$$
	as desired.
\end{proof}

\begin{lem}\label{lem:largedev_ineq}
	Fix $\lambda_0\in U$ and $\theta' \in (0, \theta)$. For every $\varepsilon>0$ there exist $\xi_{22}>0$ and $C_{22}>0$ such that for every $\lambda\in B_{\xi_{22}}(\lambda_0)$ and every $n\geq1,\ \ell = 1, \ldots, p$,
	\[
	\begin{split} \mu_\lambda & \left( \left\{\omega\in\Om:\left|\frac{1}{n}S_ng_\ell^\lambda(\omega)-\int g_\ell^\lambda d\mu_\lambda\right|>\varepsilon\right\}\right) \\
	& \leq C_{22} e^{cn|\lambda - \lambda_0|^{\theta'}}\mu_{\lambda_0}\left(\left\{\omega\in\Om:\left|\frac{1}{n}S_ng_\ell^{\lambda_0}(\omega)-\int g_\ell^{\lambda_0} d\mu_{\lambda_0}\right|>\frac{\varepsilon}{5}\right\}\right),
	\end{split} \]
	with $c = c(\theta')$ as in Lemma \ref{lem:contint}.
\end{lem}

\begin{proof}
	Fix $\lambda_0 \in U$ and $\eps>0$. Fix $k \in \N$ large enough to have $b\alpha^k \leq \frac{\eps}{5}$. For a given $n \in \N$, let $\vphi_\ell^\lam(\omega) = g_\ell^{\lam}(\omega|_{n+k}1^\infty)$. Note that by \eqref{eq:unifvar} for $g_\ell^\lam$ we have $\left\| \frac{1}{n} S_ng_\ell^\lam - \frac{1}{n}S_n\vphi_\ell^\lam \right\|_{\infty} \leq \frac{\eps}{5}$, whereas \eqref{eq:assong} for $g_\ell^\lam$ yields $\left\| \frac{1}{n} S_n\vphi_\ell^\lam - \frac{1}{n}S_n\vphi_\ell^{\lam_0} \right\|_{\infty} \leq \frac{\eps}{5}$ if $\lam$ is close enough to $\lam_0$. Moreover, functions $\om \mapsto S_n \vphi_\ell^\lam(\om)$ are constant on cylinders of length $n+k$. Therefore, applying Proposition \ref{eq:regeq} and Lemma \ref{lem:contint} gives for $\lam \in B(\lam_0, \xi_{22})$ with $\xi_{22}$ small enough:
	\[
	\begin{split}
	\mu_\lambda & \left( \left\{\omega\in\Om:\left|\frac{1}{n}S_ng_\ell^\lambda(\omega)-\int g_\ell^\lambda d\mu_\lambda\right|>\varepsilon\right\}\right) \\
	& \leq \mu_\lambda \left( \left\{\omega\in\Om:\left|\frac{1}{n}S_n\vphi_\ell^\lam(\omega)-\int g_\ell^\lambda d\mu_\lambda\right|>\frac{4\eps}{5}\right\}\right) \\
	& = \sum \limits_{|i| = n+k} \mu_\lam\left( [i] \right) \mathds{1}_{\left\{ \left|\frac{1}{n}S_n\vphi_\ell^\lam(i 1^\infty)-\int g_\ell^\lambda d\mu_\lambda\right|>\frac{4\eps}{5} \right\}}(i) \\
	& \leq e^{c(n+k)|\lam - \lam_0|^{\theta'}}\sum \limits_{|i| = n+k} \mu_{\lam_0}\left( [i] \right) \mathds{1}_{\left\{ \left|\frac{1}{n}S_n\vphi_\ell^\lam(i 1^\infty)-\int g_\ell^\lambda d\mu_\lambda\right|>\frac{4\eps}{5} \right\}}(i) \\
	& \leq C_{22}e^{cn|\lam - \lam_0|^{\theta'}}\sum \limits_{|i| = n+k} \mu_{\lam_0}\left( [i] \right) \mathds{1}_{\left\{ \left|\frac{1}{n}S_n\vphi_\ell^{\lam_0}(i 1^\infty)-\int g_\ell^{\lambda_0} d\mu_{\lambda_0}\right|>\frac{2\eps}{5} \right\}}(i) \\
	& \leq C_{22}e^{cn|\lam - \lam_0|^{\theta'}} \mu_{\lambda_0}\left(\left\{\omega\in\Om:\left|\frac{1}{n}S_ng_\ell^{\lambda_0}(\omega)-\int g_\ell^{\lambda_0} d\mu_{\lambda_0}\right|>\frac{\varepsilon}{5}\right\}\right),
	\end{split}
	\]
	where $C_{22} = \exp(ck\xi_{22}^{\theta'})$.
\end{proof}

\begin{proof}[Proof of Proposition~\ref{prop:largedev}]
	Fix $\lam_0 \in U$ and $\eps>0$. By \cite[Theorem 6]{Young90}, there exist $C_D > 0$ and $s>0$ such that
	\[ \mu_{\lambda_0}\left(\left\{\omega\in\Om:\left|\frac{1}{n}S_ng_\ell^{\lambda_0}(\omega)-\int g_\ell^{\lambda_0} d\mu_{\lambda_0}\right|>\frac{\varepsilon}{5}\right\}\right) \leq C_D e^{-2sn}\]
	for every $n \in \N$. Combining this with Lemma \ref{lem:largedev_ineq} finishes the proof.
\end{proof}

Fix $\theta' \in (0,\theta),\ \lam_0 \in \ov{U}$ and $\eps>0$. For every $n\geq \log(B_0)/\varepsilon$ let
$$
\Omega_n^c:=\left\{\ii\in\Ak^n:\text{ there exist $\omega\in[\ii]$ and $\ell \in [1,p]$ such that }\left|\frac{1}{n}S_ng_\ell^{\lambda_0}(\omega)-\int g_\ell^{\lambda_0}d\mu_{\lambda_0}\right|>4\varepsilon\right\}.
$$
We define $\Omega_n:=\Ak^n\setminus\Omega_n^c$. Choose  $$\xi \leq \min \{\xi_D, \xi_{21}\}$$ such that $c_0|\lambda-\lambda_0|^{\theta} < \eps$ and $C_g|\lam - \lam_0|^{\theta'}<\varepsilon$ for $\lambda\in B_{\xi_{21}}(\lambda_0)$. Then, for such $\lam$, Lemma \ref{lem:contint} gives that for every $\ii\in\Omega_n^c,\ \omega\in[\ii],\ \ell = 1, \ldots, p$
\begin{equation}\label{eq:defomega}
\left|\frac{1}{n}S_ng_\ell^{\lambda}(\omega)-\int g_\ell^{\lambda}d\mu_{\lambda}\right|>\varepsilon.
\end{equation}
Let us define two sequences $n_k=\lfloor(1+\varepsilon)^k\rfloor$ and $m_k=\lfloor 1+(1+\varepsilon)+\cdots+(1+\varepsilon)^k\rfloor$. For every $K\geq1$ with $m_K\geq\log(B_0)/\varepsilon$ we let
\begin{equation}\label{eq:Xi def}
\Xi_K:=\Omega_{m_K}\times\Omega_{n_{K+1}}\times\Omega_{n_{K+2}}\times\cdots\subset\Om.
\end{equation}
For $k\geq K$, denote  $\Gamma_{m_k}:=\Omega_{m_K}\times\Omega_{n_{K+1}}\times\cdots\times\Omega_{n_k}$. By Proposition \ref{prop:largedev},
\begin{equation}\label{eq:Xi measure}
\begin{split}
\mu_\lambda(\Xi_K^c)&=\sum_{\jj\in\Omega_{m_{K}}^c}\mu_{\lambda}([\jj])+\sum_{k=1}^\infty\sum_{\ii_0\in\Omega_{m_{K}}}\sum_{\ii_1\in\Omega_{n_{K+1}}}\cdots\sum_{\ii_{k-1}\in\Omega_{n_{K+k-1}}}\sum_{\jj\in\Omega_{n_{K+k}}^c}\mu_{\lambda}([\ii_0\ii_1\ldots\ii_{k-1}\jj])\\
&\leq C_D pe^{-sm_K}+\sum_{k=1}^\infty C_D pe^{-sn_{K+k}}\to0\text{ as }K\to\infty.
\end{split}
\end{equation}

\begin{prop}\label{prop:unifegorov}
For every $K$ with $n_K\geq\log(B_0)/\varepsilon$ there exists $c' = c'(\eps, K)>0$ such that the inequality 
	\be\label{eq:regeq_egorov}
	\mu_\lambda([\ii]\cap\Xi_K)\leq e^{c'|\lambda-\tau|^{\theta'}|\ii|} \mu_\tau([\ii]\cap\Xi_K)
	\ee
	holds for every $i \in \Om^*$ and every $\lambda,\tau\in B_{\xi}(\lambda_0)$ (with $\xi$ defined above).
\end{prop}

\begin{proof} First, we shall prove \eqref{eq:regeq_egorov} for $\ii \in \Om^*$ with $|\ii| = m_L$ for $L\geq K$. Note that if  $\ii\notin\Gamma_{m_L}$, then $[\ii] \cap \Xi_K = \emptyset$, hence it suffices to prove the inequality for $\ii\in\Gamma_{m_L}$. By definition,
	$$
	\mu_\lambda([\ii]\cap\Xi_K)=\mu_{\lambda}([\ii])-\sum_{\jj\in\Omega_{n_{L+1}}^c}\mu_{\lambda}([\ii\jj])
	-\sum_{k=1}^\infty\sum_{\ii_1\in\Omega_{n_{L+1}}}\cdots\sum_{\ii_k\in\Omega_{n_{L+k}}}\sum_{\jj\in\Omega_{n_{L+k+1}}^c}\mu_{\lambda}([\ii\ii_1\ldots\ii_k\jj]).
	$$
	For short, denote 
	\begin{eqnarray*}
		b_{L+1}(\lambda)&:=&\frac{1}{\mu_\lambda([\ii])}\sum_{\jj\in\Omega_{n_{L+1}}^c}\mu_{\lambda}([\ii\jj]);\\
		b_{L+k+1}(\lambda)&:=&\frac{1}{\mu_\lambda([\ii])}\sum_{\ii_1\in\Omega_{n_{L+1}}}\cdots\sum_{\ii_k\in\Omega_{n_{L+k}}}\sum_{\jj\in\Omega_{n_{L+k+1}}^c}\mu_{\lambda}([\ii\ii_1\ldots\ii_k\jj]),\ \ k\ge 1.
	\end{eqnarray*}
	Hence, by Proposition \ref{eq:regeq},
	\begin{equation}\label{eq:rat1}
	\frac{\mu_\lambda([\ii]\cap\Xi_K)}{\mu_\tau([\ii]\cap\Xi_K)}\leq\frac{\mu_\lambda([\ii])}{\mu_{\tau}([\ii])}\cdot\frac{1-\sum_{k=1}^\infty e^{-c|\lambda-\tau|^{\theta'} (m_{L+k}+|\ii|)}b_{L+k}(\tau)}{1-\sum_{k=1}^\infty b_{L+k}(\tau)}.
	\end{equation}
	By  the Mean Value Theorem, there exists $\rho\in(e^{-c|\lambda-\tau|^{\theta'}},1)$ such that
	\begin{eqnarray*}
 & &	\log\left(1-\sum_{k=1}^\infty e^{-c|\lambda-\tau|^\theta (m_{L+k}+|\ii|)}b_{L+k}(\tau)\right)-\log\left(1-\sum_{k=1}^\infty b_{L+k}(\tau)\right)\\
	&=& \frac{\sum_{k=1}^\infty (m_{L+k}+|\ii|)\rho^{m_{L+k}+|\ii|-1}b_{L+k}(\tau)}{1-\sum_{k=1}^\infty \rho^{m_{L+k}+|\ii|}b_{L+k}(\tau)}\left(1-e^{c|\lambda-\tau|^{\theta'}}\right)\\
	&\leq &\frac{\sum_{k=1}^\infty (m_{L+k}+|\ii|)b_{L+k}(\tau)}{1-\sum_{k=1}^\infty b_{L+k}(\tau)}c|\lambda-\tau|^{\theta'}.
	\end{eqnarray*}
	By the Gibbs property of $\mu_\tau$ we have
	\[
	\begin{split}
	b_{L+k}(\tau)\leq C_G\mu_{\tau}\left(\bigcup_{\ii\in\Omega_{n_{L+k}}^c}[\jj]\right) & \leq C_G\mu_\tau\left(\left\{ \underset{1 \leq \ell \leq p}{\exists}\ \left|\frac{1}{n_{L+k}}S_{n_{L+k}}g_\ell^{\tau}-\int g_\ell^\tau d\mu_\tau\right|>\varepsilon\right\}\right) \\ 
	&\leq p C_G C_De^{-sn_{L+k}},
	\end{split}
	\]
	where in the last two inequalities we used \eqref{eq:defomega} and Proposition~\ref{prop:largedev}. Hence,
	\begin{eqnarray*}
	\frac{\sum_{k=1}^\infty (m_{L+k}+|\ii|)b_{L+k}(\tau)}{1-\sum_{k=1}^\infty b_{L+k}(\tau)} & \leq& \frac{p C_G C_D\sum_{k=1}^\infty (m_{L+k}+|\ii|)e^{-sn_{L+k}}}{1-p C_G C_D\sum_{k=1}^\infty e^{-sn_{L+k}}}\\
	& \leq & \frac{2p C_G C_D\sum_{k=1}^\infty m_{L+k}e^{-sn_{L+k}}}{1-p C_G C_D\sum_{k=1}^\infty e^{-sn_{L+k}}},
	\end{eqnarray*}
	which is a uniform constant. Combining this with \eqref{eq:rat1} and Proposition~\ref{eq:regeq}, we get
	$$
	\frac{\mu_\lambda([\ii]\cap\Xi_K)}{\mu_\tau([\ii]\cap\Xi_K)}\leq e^{c|\lambda-\tau|^{\theta'}(|\ii|+1)}.
	$$
	Now let us extend \eqref{eq:regeq_egorov} to all $\ii \in \Om^*$ with $|\ii| \geq m_K$. Let $m_L \leq |\ii| < m_{L+1}$ for $L \geq K$. Then
	\[
	\begin{split}
	\mu_\lambda([\ii]\cap\Xi_K) & = \sum \limits_{\jj \in \Ak^{m_{L+1} - |\ii|}} \mu_{\lam}([\ii \jj] \cap \Xi_k) \leq \sum \limits_{\jj \in \Ak^{m_{L+1} - |\ii|}} e^{c|\lam - \tau|^{\theta'}m_{L+1}} \mu_{\tau}([\ii \jj] \cap \Xi_k) \\
	& \leq e^{c|\lam - \tau|^{\theta'}m_{L+1}} \mu_{\tau}([\ii]\cap\Xi_K) \leq e^{c|\lam - \tau|^{\theta'}|\ii| \frac{m_{L+1}}{m_L}} \mu_{\tau}([\ii]\cap\Xi_K)\\
	& \leq e^{(3+\eps)c|\lam - \tau|^{\theta'}|\ii|} \mu_{\tau}([\ii]\cap\Xi_K).
	\end{split}
	\]
	Finally, for $\ii \in \Om^*$ with $|\ii| < m_K$, the same calculation as above shows
	\[ \mu_\lambda([\ii]\cap\Xi_K) \leq e^{c m_K |\lam - \tau|^{\theta'}|\ii|}\mu_\tau([\ii]\cap\Xi_K). \]
\end{proof}

\begin{lem}\label{lem:unifegorov deviation}
	For every $\ell = 1, \ldots p,\ K\geq\log(B_0)/\varepsilon,\ n\geq m_K$, and every $\ii\in\Om_*$ with $|\ii|=n$ and $[\ii]\cap\Xi_K \neq \emptyset$, every $\omega\in[\ii]$ and every $\lambda\in B_\xi(\lambda_0)$, the following holds:
	$$
	\left|\frac{1}{n}S_ng_\ell^\lambda(\omega)-\int g_\ell^\lambda d\mu_\lambda\right|<(6+4M)\varepsilon,
	$$
	where $M=\max_{\lambda\in \ov{U},\omega\in\Om, 1 \leq \ell \leq p}|g_\ell^\lambda(\omega)|$.
\end{lem}

\begin{proof}
	Let $L\geq K$ be such that $m_L\leq n<m_{L+1}$. Then
	\[
	\begin{split}
	&\left|\frac{1}{n}S_ng_\ell^\lambda(\omega)-\int g_\ell^\lambda d\mu_\lambda\right|\\&\leq\frac{m_L}{n}\left|\frac{1}{m_L}S_{m_L}g_\ell^\lambda(\omega)-\int g_\ell^\lambda d\mu_\lambda\right|+\left|\frac{1}{n}(S_ng_\ell^\lambda(\omega)-S_{m_L}g_\ell^\lambda(\omega))-\frac{n-m_L}{n}\int g_\ell^\lambda d\mu_\lambda\right|\\
	&\leq \frac{m_L}{n}6\varepsilon+\frac{n-m_L}{n}2M\leq 6\varepsilon+ \frac{n_{L+1}}{m_L}\cdot 2M\leq 6\varepsilon+\frac{(1+\varepsilon)^{L+1}}{(1+\varepsilon)^{L}-1}\varepsilon\cdot 2M.
	\end{split}
	\]
	Since $K$ is large, the claim follows.
\end{proof}
Now we are ready to prove Proposition \ref{prop:stable cor bound}.

\begin{proof}[Proof of Proposition \ref{prop:stable cor bound}]
Let $g_1^\lam(\om) = P_\lam - \phi^\lam(\om)$ and $g_2^\lam (\om) = -\log \left| \left(f^\lam_{\om_1}\right)' (\Pi^{\lam}(\sigma \om))\right|$. Then $h_{\mu_{\lam}} = \int g_1^\lam d\mu_{\lam}$ and $\chi_{\mu_\lam} = \int g_2^\lam d\mu_{\lam}$.
Fix $\eps>0,\ \eps'>0$, and $\theta' \in [0,\theta)$. Let $\xi>0$ be small enough, so that Proposition \ref{prop:unifegorov} and Lemma \ref{lem:unifegorov deviation} hold. Let $A = \Xi_K$ be defined as in \eqref{eq:Xi def} for fixed $K \geq \log(B_0)/\eps$, large enough to have $\mu_\lam(A) \geq 1 - \eps'$ for $\lam \in B_{\xi}(\lam_0)$ by \eqref{eq:Xi measure}. Then $\tilde{\mu}_\lam = \mu_\lam |_A$ satisfies \eqref{eq:tilde mu bounds} by Proposition \ref{prop:unifegorov}. By the Gibbs property and Lemma \ref{lem:unifegorov deviation}, for $u \in \Om^*$ satisfying $[u] \cap A \neq \emptyset$ with $|u| = n \geq m_K$ and any $\om \in [u]$, we have
\[ \tilde{\mu}_{\lam}([u]) \leq \mu_\lam([u]) \leq C_G \exp(-P_\lam n + S_n\phi^\lam(\om)) = C_G \exp(-S_n g_1^\lam(\om)) \leq C_G e^{-n(h_{\mu_\lam} - (6+4M)\eps) }\]
and
\[ \left| \left(f^\lam_{u}\right)' (\Pi^{\lam}(\sigma^n \om))\right| \geq e^{-n(\chi_{\mu_\lam} + (6 + 4M)\eps)}. \]
Therefore, setting $A_n = \{ u \in \Ak^n : [u] \cap A \neq \emptyset \}$ and applying Lemma \ref{lem:pbdp}, we obtain for $\alpha > 0$,
\[
\begin{split} \Ek_{\alpha}(\tilde{\mu}_{\lam} , d_{\lam}) & = \sum \limits_{n=0}^{\infty} \sum \limits_{u \in A_n} \sum_{\substack{i,j\in \Ak\\ i \neq j}} \left|f_{u}^{\lam} (X)\right|^{-\alpha} \tilde{\mu}_\lam([ui])\tilde{\mu}_\lam([uj]) \\ & \leq C_{61}^{\alpha} C_G \sum \limits_{n=0}^{\infty} \sum \limits_{u \in A_n} \sum_{\substack{i,j\in \Ak\\ i \neq j}} e^{-n(h_{\mu_\lam} - (6+4M)\eps - \alpha(\chi_{\mu_\lam} + (6 + 4M)\eps))} \tilde{\mu}_{\lam}([uj]) \\
& \leq C_{61}^{\alpha} C_G \# \Ak \sum \limits_{n=0}^{\infty} e^{-n(h_{\mu_\lam} - (6+4M)\eps - \alpha(\chi_{\mu_\lam} + (6 + 4M)\eps))} < \infty,
\end{split}
\]
provided $\alpha < \frac{h_{\mu_\lam} - (6+4M)\eps}{\chi_{\mu_\lam} + (6 + 4M)\eps}$. This shows $\dim_{cor}(\tilde{\mu}_\lam, d_\lam) \geq \frac{h_{\mu_\lam} - (6+4M)\eps}{\chi_{\mu_\lam} + (6 + 4M)\eps}$.
\end{proof}


\section{Proofs of Theorems \ref{thm:main_cor_dim} and \ref{thm:main_gibbs}}

\begin{lem}\label{lem:dim_cor_continuity}
	Let $\{f_j^\lam\}_{j \in \Ak}$ be a parametrized IFS satisfying smoothness assumptions \ref{as:C2} - $\ref{as:hyper}$. Let $\left\{ \mu_\lam \right\}_{\lam \in \ov{U}}$ be a collection of finite Borel measures on $\Om$ satisfying \ref{as:measure}. Then the map
	\[ \ov{U} \ni \lam \mapsto \dim_{cor}(\mu_\lam, d_\lam) \]
	is continuous.
\end{lem}

\begin{proof}
	Fix arbitrary $\alpha>0,\ \eps > 0$. It is enough to prove that there exists a constant $\widehat{C}>0$ such that inequality
	\[ \Ek_{\alpha}(\mu_{\lam} , d_{\lam}) \leq \widehat{C} \Ek_{\alpha + \eps}(\mu_{\lam'} , d_{\lam'}) \]
	holds provided $\lam$ and $\lam'$ are close enough. By \ref{as:measure} and the parametric bounded distortion property (Lemma \ref{lem:pbdp}), 
	\[
	\begin{split} \Ek_{\alpha}(\mu_{\lam} , d_{\lam}) & = \sum \limits_{n=0}^{\infty} \sum \limits_{u \in \Ak^n} \sum_{\substack{i,j\in \Ak\\ i \neq j}} \left|f_{u}^{\lam} (X)\right|^{-\alpha} \mu_\lam([ui])\mu_\lam([uj]) \\ 
	& \leq C_{62} \sum \limits_{n=0}^{\infty} \sum \limits_{u \in \Ak^n} \sum_{\substack{i,j\in \Ak\\ i \neq j}} e^{(c+c_{62}) n|\lam - \lam'|^{\theta}} \left|f_{u}^{\lam'} (X)\right|^{-\alpha} \mu_{\lam'}([ui])\mu_{\lam'}([uj]) \\
	& \leq C_{62} \sum \limits_{n=0}^{\infty} \sum \limits_{u \in \Ak^n} \sum_{\substack{i,j\in \Ak\\ i \neq j}} \left|f_{u}^{\lam'} (X)\right|^{-(\alpha+\eps)} \mu_{\lam'}([ui])\mu_{\lam'}([uj]) \\
	& = C_{62} \Ek_{\alpha + \eps}(\mu_{\lam'} , d_{\lam'}),
	\end{split}
	\]
	where the last inequality holds provided $|\lam - \lam'|$ is small enough, as $\left|f_{u}^{\lam'} (X)\right|^{-\eps} \geq \gamma_2^{-\eps n}$ by \ref{as:hyper}.
\end{proof}

\subsection{Proof of Theorem \ref{thm:main_cor_dim}}
Fix $\lam_0 \in U$ with $\dim_{cor}(\mu_{\lam_0}, d_{\lam_0}) > 1$. Let $\eps>0$ be small enough to have 
\[\gamma := \frac{\min \left\{\dim_{cor}(\mu_{\lam_0}, d_{\lam_0}), 1 + \min\{\delta, \theta \} \right\} - 4\eps - 1}{2} > 0.\]
Let $q = 1 + 2\gamma + 2\eps$. Then
\[ 1 + 2\gamma + \eps < q \leq \min \left\{\dim_{cor}(\mu_{\lam_0}, d_{\lam_0}), 1 + \min\{\delta, \theta \} \right\} - 2\eps.  \]
Let $\beta>0$ be small enough to have
\[ q(1 + a_0\beta) \leq \min \left\{\dim_{cor}(\mu_{\lam_0}, d_{\lam_0}), 1 + \min\{\delta, \theta \} \right\} - \eps, \]
where $a_0$ is as in Theorem \ref{thm:sobolev integral bound}. By Theorem \ref{thm:sobolev integral bound}, there exists an neighbourhood $J$ of $\lambda_0$ in $U$, interval $I$ containing $\lam_0$ and compactly supported in $J$ and smooth function $\rho$ with $0 \leq \rho \leq 1,\ \supp(\rho) \subset J$ and $\rho \equiv 1$ on $I$, such that
\[ \int_{I} \|\nu_\lam\|^2_{2,\gamma}d\lam \leq \int_{J}\|\nu_\lambda\|_{2,\gamma}^2\rho(\lambda)d\lambda\le \widetilde{C}_1\Ek_{q(1+a_0\beta)}(\mu_{\lambda_0}, d_{\lam_0}) + \tilde{C}_2 < \infty \]
as $q(1 + a_0\beta) \leq \dim_{cor}(\mu_{\lam_0}, d_{\lam_0}) - \eps$. Therefore, $\|\nu_\lam\|^2_{2,\gamma} < \infty$ for Lebesgue almost every $\lam \in {I}$, hence
\[\dim_S((\Pi^\lam)_* \mu_\lam) \geq 1 + 2\gamma \geq \min \left\{\dim_{cor}(\mu_{\lam_0}, d_{\lam_0}), 1 + \min\{\delta, \theta \} \right\} - 4\eps\]
holds almost surely on ${I}$. As $\eps$ can be taken arbitrary small and the function $\lam \mapsto \dim_{cor}(\mu_\lam, d_\lam)$ is continuous by Lemma \ref{lem:dim_cor_continuity}, we can conclude the result in the same way as in the proof of Theorem \ref{thm:main_hausdorff} (see the last paragraph of Section \ref{sec:proof of main_hausdorff}).

\subsection{Proof of Theorem \ref{thm:main_gibbs}}
As Proposition \ref{eq:regeq} implies that measures $\{ \mu_\lam\}_{\lam \in \ov{U}}$ satisfy \ref{as:measure} with $\theta'$ arbitrarily close to $\theta$, the first assertion of Theorem \ref{thm:main_gibbs} follows from Theorem \ref{thm:main_cor_dim}. For the absolute continuity part, fix $\eps>0$ and $\eps'>0$ and let $\tilde{\mu}_\lam$ be as in Proposition \ref{prop:stable cor bound}. By Theorem \ref{thm:main_cor_dim} we have $\dim_S((\Pi^\lam)_*\tilde{\mu}_\lam) > 1$ for Lebesgue almost every $\lam$ with $\frac{h_{\mu_{\lam}}}{\chi_{\mu_{\lam}}} > 1 + \eps$. As any measure on $\R$ with Sobolev dimension greater than $1$ is absolutely continuous (with $L^2$ density), passing with $\eps'$ and $\eps$ to zero finishes the proof.

\section{Applications}

\subsection{Place-dependent Bernoulli convolutions}\label{sec:Bern}

Our first application is the place-dependent Bernoulli convolution studied in \cite{BB15}. Let $0<\rho<\frac{1}{2}$ and $0.5<\lambda<1$ and let us consider the following dynamical system $f:[-1,1]\times[0,1]\mapsto[-1,1]\times[0,1]$, where
\[
f(x,y)=\left\{\begin{array}{cc}
\left(\lambda x-(1-\lambda),\frac{2y}{1+2\rho x}\right) & \text{if }0\leq y<\frac{1}{2}+\rho x \\
\left(\lambda x+(1-\lambda),\frac{2y-2\rho x-1}{1-2\rho x}\right) & \text{if } \frac{1}{2}+\rho x\leq y\leq1.
\end{array}\right.
\]
For the action of $f$ on the rectangle $[-1,1]\times[0,1]$ see Figure~\ref{ffunction}.
\begin{figure}
	\includegraphics[width=170mm]{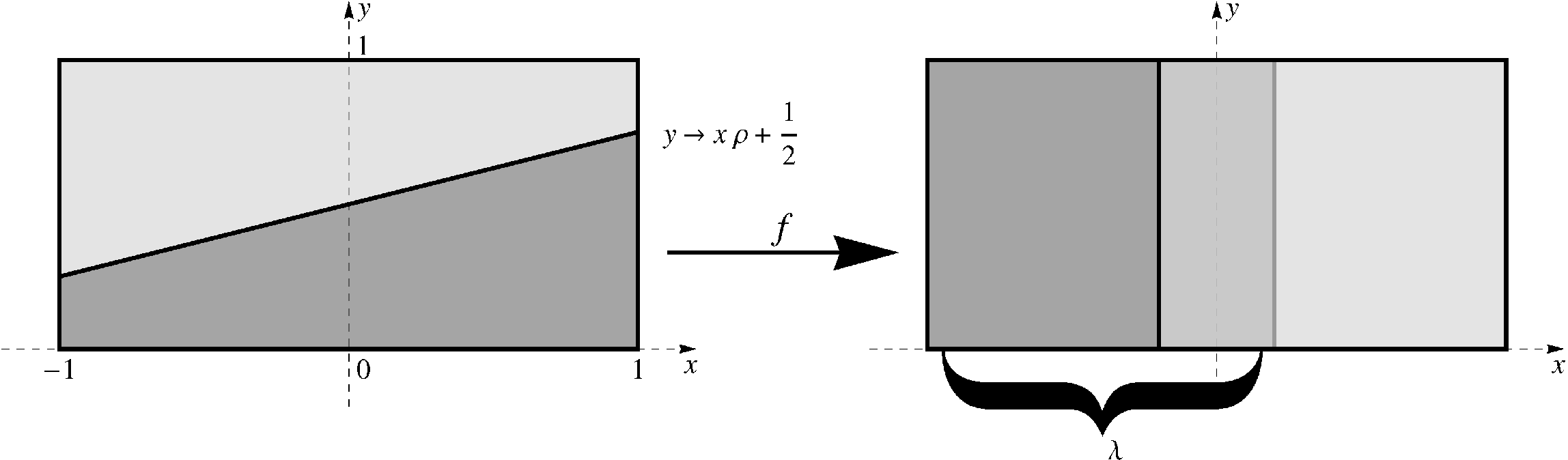}\\
	\caption{The map $f$ acting on the rectangle $[-1,1]\times[0,1]$.}\label{ffunction}
\end{figure}

Let $\nu_{\lambda,\rho}$ be the place-dependent invariant measure of the IFS on $[-1,1]$ $$\Psi_{\lambda}=\left\{\psi_0^{\lambda}(x)=\lambda x-(1-\lambda),\psi_1^{\lambda}(x)=\lambda x+(1-\lambda)\right\}$$ with probabilities $\left\{p_0(x)=\frac{1}{2}+\rho x,p_1(x)=\frac{1}{2}-\rho x\right\}$. That is, $\nu_{\lambda,\rho}$ is the unique probability measure of the dual operator $L^*$, where
$$
Lg(x)=\left(\frac{1}{2}+\rho x\right)g(\lambda x-(1-\lambda))+\left(\frac{1}{2}-\rho x\right)g(\lambda x+(1-\lambda)),
$$
for any continuous test function $g\colon[0,1]\mapsto\R$. In fact, by \cite[Theorem~1.1]{FL}, 
\begin{equation}\label{eq:convplacedep}
\lim_{n\to\infty}L^ng(x)=\int gd\nu_{\lambda,\rho}\text{ uniformly on $[0,1]$.}
\end{equation}
Applying \eqref{eq:convplacedep} and the bounded convergence theorem, simple calculations show that
\[
\frac{1}{n}\sum_{k=0}^{n-1}\overline{\mathcal{L}}_2\circ f^{-k}\rightarrow \nu_{\lambda,\rho}\times\mathcal{L}_1\text{ weakly},
\]
where $\overline{\mathcal{L}}_2$ is the normalized Lebesgue measure on the rectangle. Hence, by the results of Schmeling and Troubetzkoy~\cite[Section~2, 3]{SchTro}, the measure $\nu_{\lambda,\rho}\times\mathcal{L}_1$ is the unique SBR-measure of the map $f$. Therefore, the property $\nu_{\mathrm{SBR}}\ll\mathcal{L}_2$ is equivalent to $\nu_{\lambda,\rho}\ll\mathcal{L}_1$ and moreover $\dim_{\rm H}\nu_{SBR}=1+\dim_{\rm H}\nu_{\lambda,\rho}$. 

Clearly, the IFS $\Psi_\lambda$ satisfies the conditions \ref{as:C2}-\ref{as:hyper} for $\lambda$ in an arbitrary compact subinterval of $(0,1)$. Moreover, it is easy to see that $\nu_{\lambda,\rho}$ is a push-forward measure of a parameter-dependent Gibbs measure $\mu_{\lambda,\rho}$. More precisely, let $\Omega=\{-1,1\}^\N$ and
$$
\Pi^\lambda(\omega)=\sum_{k=1}^\infty\omega_k\lambda^{k-1},
$$
and let $\phi^\lambda(\omega)=\log\left(p_{\omega_1}(\Pi^\lambda(\sigma\omega))\right)$. It is easy to see that $\phi^\lambda$ satisfies \eqref{eq:unifvar} and \eqref{eq:assong} for every fixed $\rho\in[0,1/2)$. Moreover,
\begin{eqnarray*}
	&\chi_{\mu_{\lambda,\rho}}=-\log\lambda;\\
	&h_{\mu_{\lambda,\rho}}=-\int_{\R}\left(\frac{1}{2}+\rho x\right)\log\left(\frac{1}{2}+\rho x\right)+\left(\frac{1}{2}-\rho x\right)\log\left(\frac{1}{2}-\rho x\right)d\nu_{\lambda,\rho}(x).
\end{eqnarray*}
Shmerkin and Solomyak~\cite[Theorem~2.6]{SS2} showed that $\Psi_\lambda$ satisfies the transversality condition \ref{as:trans} on the interval $\lambda\in(0,0.6684755)$. Hence we can apply Theorem~\ref{thm:main_gibbs} and verify the claim \cite[Theorem~4.1]{BB15}.

\begin{thm}\label{th:app1}
	For every $0\leq\rho<0.5$ and Lebesgue almost every $\lambda\in(0.5,0.6684755)$,
	\[\dim_{\rm H}\nu_{\lambda,\rho}=\min\Bigl\{1,\frac{h_{\mu_{\lambda,\rho}}}{-\log\lambda}\Bigr\}
	\]
	Moreover, $\nu_{\lambda,\rho}$ is absolutely continuous for Lebesgue almost every
	\[
	\lambda\in\left\{\lambda\in(0.5,0.6684755):h_{\mu_{\lambda,\rho}}>-\log\lambda\right\}.
	\]
	In particular, the region contains the quadrilateral formed by the points $(0,0.5)$, $(0.45,0.55)$, $(0.45,0.668)$, $(0,0.668)$.
\end{thm}
\begin{figure}
	\includegraphics[width=100mm]{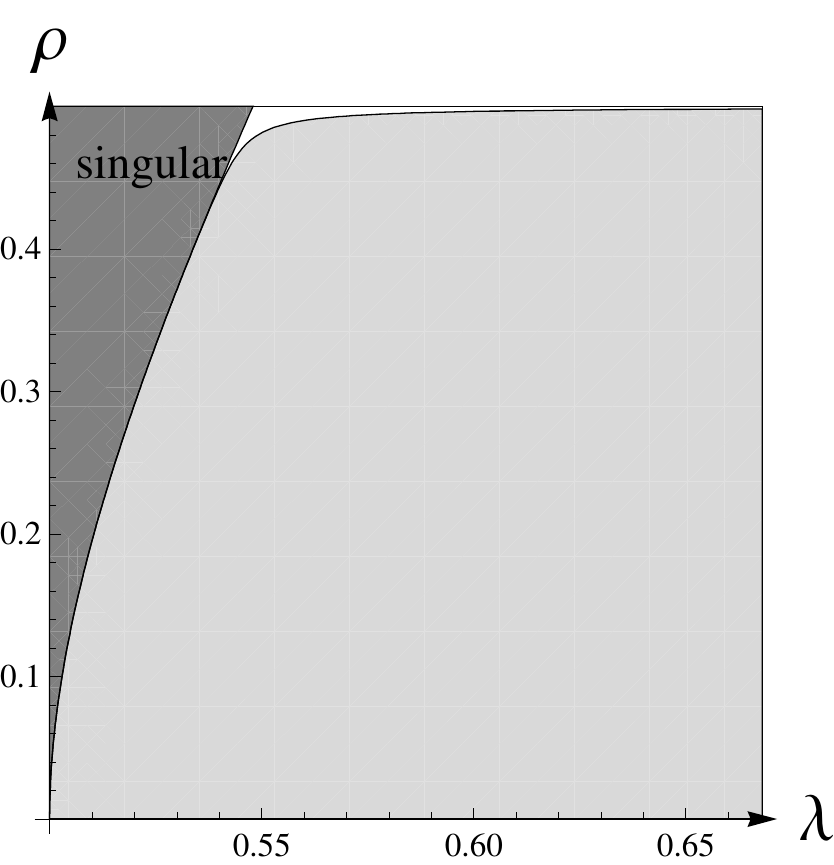}\\
	\caption{The singularity and absolute continuity region of the measure $\mu_{\lambda,\rho}$.}\label{fregion}
\end{figure}

It follows from the calculations in \cite{BB15}, that for every $N\geq1$,
\begin{multline}\label{eentbernbound}
\log2-\sum_{n=1}^{N}\frac{(2\rho)^{2n}}{2n(2n-1)}F_n-\frac{(2\rho)^{N+1}}{(2N+2)(2N+1)(1-(2\rho)^2)}\leq h_{\mu_{\lambda,\rho}}\leq\log2-\sum_{n=1}^{N}\frac{(2\rho)^{2n}}{2n(2n-1)}F_n.
\end{multline}
where $F_n=\int x^{2n}d\mu_{\lambda,\rho}(x)$. The quantities $F_n$ can be expressed inductively by 
\begin{multline*}
F_n=\frac{(1-\lambda)^{2n}}{1+\lambda^{2n-1}(4n\rho(1-\lambda)-\lambda)}+\\
\sum_{m=1}^{n-1}\frac{2m(1-\lambda)^{2n-2m}\lambda^{2m-1}}{1+\lambda^{2n-1}(4n\rho(1-\lambda)-\lambda)}\binom{2n}{2m}\left(\frac{\lambda}{2m}-\frac{2\rho(1-\lambda)}{2n-2m+1}\right)F_{m}.
\end{multline*}
Using the estimates \eqref{eentbernbound}, we can approximate the region in Theorem \ref{th:app1}, see Figure~\ref{fregion}.

\subsection{Blackwell measure for binary channel} Our second application is the absolute continuity of the Blackwell measure for a binary symmetric channel with a noise. Let us first introduce the basic notations, following B\'ar\'any, Pollicott and Simon \cite{BPS} and B\'ar\'any and Kolossv\'ary \cite{BK}. Let $X:=\left\{X_i\right\}_{i=-\infty }^{\infty }$ be a binary, symmetric, stationary, ergodic Markov chain  source $X_i\in \left\{0,1\right\}$, with a probability transition matrix
\begin{equation*}\label{111}
\Pi :=\left[\begin{array}{cc}
p  & 1-p \\
1-p &  p  \\
\end{array}
\right].
\end{equation*}
By adding to $X$ a binary independent and identically distributed (i.i.d.) noise $E=\left\{E_i\right\}_{i=-\infty }^{\infty }$ independent of $X$ with
$$
\mathbb{P}(E_i=0)=1-\varepsilon ,\qquad \mathbb{P}(E_i=1)=\varepsilon,
$$
we get a Markov chain $Y:=\left\{Y_i\right\}_{i=-\infty }^{\infty }$, $Y_i=(X_i,E_i)$ with states $\{(0,0), (0,1), (1,0), \\(1,1) \}$ and transition probabilities:
\begin{equation*}\label{112}
M:=\left[\begin{array}{cccc}
p(1-\varepsilon )        & p\varepsilon  &  (1-p)(1-\varepsilon )& (1-p)\varepsilon  \\
p(1-\varepsilon )        &  p\varepsilon &  (1-p)(1-\varepsilon )&  (1-p)\varepsilon \\
(1-p)(1-\varepsilon )   & (1-p)\varepsilon  & p(1-\varepsilon ) & p\varepsilon \\
(1-p)(1-\varepsilon )   & (1-p)\varepsilon  & p(1-\varepsilon ) &  p\varepsilon\\
\end{array}
\right].
\end{equation*}
Let $\Psi:\left\{(0,0),(0,1),(1,0),(1,1)\right\}\mapsto \{0,1\}$ be a surjective map such that
\[
\Psi(0,0)=\Psi(1,1)=0\text{ and }\Psi(0,1)=\Psi(1,0)=1.
\]
We consider the ergodic stationary process $Z=\left\{Z_i=\Psi(Y_i)\right\}_{i=-\infty}^{\infty}$, which is the corrupted output of the channel. Equivalently, $Z$ is the stationary stochastic process
$Z_i=X_i\bigoplus E_i$, where $\bigoplus$ denotes the binary addition.

According to \cite[Example~4.1]{HManal} and \cite[Example~1]{BPS}, the entropy of $Z$ can be expressed as follows. Consider the 3-dimensional simplex $$W:=\Bigl\{\underline{w}\in\R^4:w_i\geq0, \ \sum_{1\le i \le 4} w_i=1\Bigr\}$$ and define $W_0,W_1\subset W$ by
$$W_0:=\left\{\underline{w}\in W:w_2=w_3=0\right\},\ W_1:=\left\{\underline{w}\in W:w_1=w_4=0\right\}.$$
Consider two matrices
\begin{equation*}
M_0:=\left[
\begin{array}{cccc}
p(1-\varepsilon) & 0 & 0 & (1-p)\varepsilon \\
p(1-\varepsilon) & 0 & 0 & (1-p)\varepsilon \\
(1-p)(1-\varepsilon) & 0 & 0 & p\varepsilon \\
(1-p)(1-\varepsilon) & 0 & 0 & p\varepsilon \\
\end{array}\ \ 
\right]\text{ and }\ \  
M_1:=\left[
\begin{array}{cccc}
0 & p\varepsilon & (1-p)(1-\varepsilon) & 0 \\
0 & p\varepsilon & (1-p)(1-\varepsilon) & 0 \\
0 & (1-p)\varepsilon & p(1-\varepsilon) & 0 \\
0 & (1-p)\varepsilon & p(1-\varepsilon) & 0 \\
\end{array}
\right],
\end{equation*}
and let $(r_0(\underline{w}),r_1(\underline{w}))$ be the {\em place-dependent} probability vector of the form
\begin{equation*}\label{eprobsimp}
r_i(\underline{w})=\|\underline{w}^TM_i\|_1,
\end{equation*}
where $\|.\|_1$ denotes the $l_1$ norm and $\underline{w}\in W$. Introduce two functions $f_0:W\mapsto W_0$ and $f_1:W\mapsto W_1$ such that
\begin{equation*}\label{efuncsimp}
f_i(\underline{w})=\frac{\underline{w}^TM_i}{\|\underline{w}^TM_i\|_1}.
\end{equation*}
Then the entropy of $Z$ can be expressed as follows:
\begin{equation*}\label{eent}
H(Z)=-\int_{W_0\cup W_1} \bigl[r_0(\underline{w})\log r_0(\underline{w})+r_1(\underline{w})\log r_1(\underline{w})\bigr]dQ(\underline{w}),
\end{equation*}
where the Blackwell measure $Q$ is the unique measure with $\mathrm{supp}(Q)\subseteq W_0\cup W_1$, such  that for every continuous function $h\colon W_0\cup W_1\mapsto\R$,
\begin{equation*}\label{ebalckwellsimpl}
\int h(\underline{w})dQ(\underline{w})=\int r_0(\underline{w})h(f_0(\underline{w}))+r_1(\underline{w})h(f_1(\underline{w}))dQ(\underline{w}).
\end{equation*}
It was shown in \cite[Section~3.1, 3.2]{BPS} that for the binary symmetric channel, the measure $Q$ on $W_0\cup W_1$ is conjugated to the place-dependent invariant probability measure $\nu_{\varepsilon,p}$ on $[0,1]$ for the IFS $\Psi_{\varepsilon,p}=\left\{S_0^{\varepsilon,p},S_1^{\varepsilon,p}\right\}$:
\begin{align*}
S_0^{\varepsilon,p}(x) &:=\dfrac{x\cdot p\cdot (1-\varepsilon )+(1-x)\cdot (1-p)\cdot(1-\varepsilon)}{x\cdot
	\left[p(1-\varepsilon )+(1-p)\cdot \varepsilon \right]+(1-x)\cdot\left[(1-p)(1-\varepsilon )+p\cdot \varepsilon \right]},\\ 
S_1^{\varepsilon,p}(x) &:=\dfrac{x\cdot p\cdot \varepsilon +(1-x)\cdot (1-p)\cdot\varepsilon}{x\cdot\left[p\varepsilon +(1-p)\cdot(1- \varepsilon) \right]  +(1-x)\cdot  \left[(1-p)\varepsilon  +p\cdot (1-\varepsilon) \right]}.
\end{align*}
and the place-dependent probability vector $(p_0^{\varepsilon,p}(x),p_1^{\varepsilon,p}(x))$:
\begin{align*}
p_0^{\varepsilon,p}(x) &:=  x\cdot  \left[p(1-\varepsilon )+  (1-p)\cdot \varepsilon \right]  +(1-x)\cdot  \left[(1-p)(1-\varepsilon )  +p\cdot \varepsilon \right],\\ 
p_1^{\varepsilon,p}(x) &:=  x\cdot  \left[p\varepsilon +  (1-p)\cdot(1- \varepsilon) \right]  +(1-x)\cdot  \left[(1-p)\varepsilon  +p\cdot (1-\varepsilon) \right]. 
\end{align*}
In particular, $Q\ll\mathcal{L}_1|_{W_0\cup W_1}$ if and only if $\nu_{\varepsilon,p}\ll\mathcal{L}_1$.

Observe that for $\varepsilon=1/2$, $S_0^{\varepsilon,p}(x)=S_1^{\varepsilon,p}(x)=(2p-1)x+1-p$ and so $\nu_{\varepsilon,p}$ is the Dirac mass on the point $1/2$. Hence, we may assume that $\varepsilon\neq1/2$.

For every fixed $\varepsilon\in(0,1)\setminus\{1/2\}$, the IFS $\Psi_{\varepsilon,p}$ satisfies the conditions \ref{as:C2}-\ref{as:hyper} for $p$ in an arbitrary compact subinterval of $(0,1)$; and $\nu_{\varepsilon,p}$ is a push-forward measure of the Gibbs measure $\mu_{\varepsilon,p}$ with respect to the potential $\phi^{\varepsilon,p}(\omega)=\log\left(p_{\omega_1}^{\varepsilon,p}(\Pi^{\varepsilon,p}(\sigma\omega))\right)$ satisfying \eqref{eq:unifvar} and \eqref{eq:assong}, where $\Pi^{\varepsilon,p}$ is the natural projection of the IFS $\Psi_{\varepsilon,p}$.

B\'ar\'any and Kolossv\'ary \cite{BK} showed that for every fixed $\varepsilon\neq1/2$ the IFS $\Psi_{\varepsilon,p}$ satisfies the transversality condition \ref{as:trans} with respect to the parameter $p$ and has $\frac{h_{\mu_{\varepsilon,p}}}{\chi_{\mu_{\varepsilon,p}}}>1$ on every interval $I$ for which $\{\varepsilon\}\times I$ is contained in the red region in Figure~\ref{blackabsreg}. Thus, the main theorem of the present paper applies and \cite[Theorem~1.1]{BK} remains correct:

\begin{thm}
	For every fixed $\varepsilon\in(0,1)\setminus\{1/2\}$ and for Lebesgue-almost every $p$ such that $(\varepsilon,p)\in R$ is in the red region of Figure~\ref{blackabsreg}, the measure $\nu^{\varepsilon,p}$ is absolutely continuous. For instance, the red region contains two quadrilaterals formed by $(0.5,0.75)$, $(0.37,0.775)$, $(0.5,0.795)$, $(0.63,0.775)$ and $(0.5,0.25), (0.37,0.225),(0.5,0.205),(0.63,0.225)$. 
\end{thm}

It was shown by B\'ar\'any, Pollicott and Simon \cite{BPS} that $\mu_{\varepsilon,p}$ is singular in the blue region of Figure~\ref{blackabsreg}.

\begin{figure}
	\includegraphics[width=120mm]{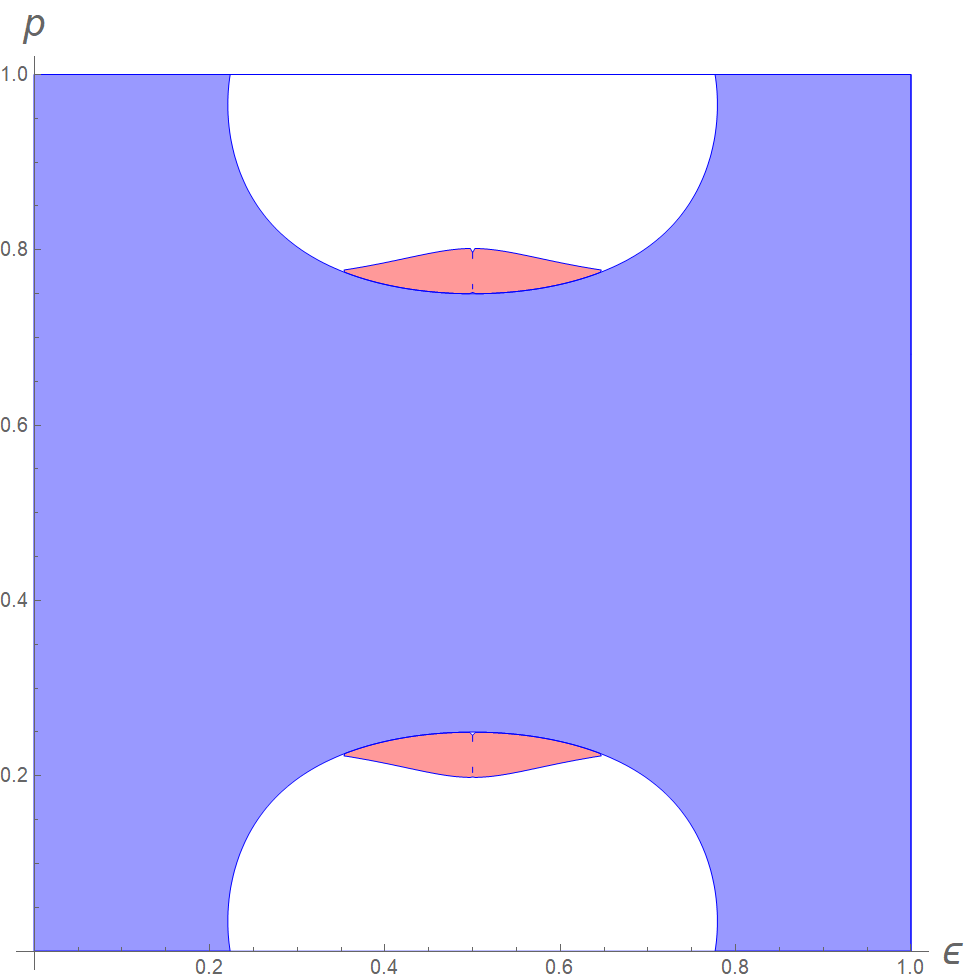}\\
	\caption{The singularity (blue) and transversality region with $\frac{h_{\mu_{\varepsilon,p}}}{\chi_{\mu_{\varepsilon,p}}}>1$ (red) of the measure $\nu_{\varepsilon,p}$, \cite[Figure~1]{BK}.}\label{blackabsreg}
\end{figure}

\subsection{Absolute continuity of equilibrium measures for hyperbolic IFS with overlaps}\label{sec:ac equilib}

\thispagestyle{empty}

First we recall briefly the notion of equilibrium measure in the setting of IFS.
Let $\Ak = \{1,\ldots,m\}$ and suppose we have an IFS $\Psi=\{f_j\}_{j \in \Ak}$ of the class $C^{1+\theta}$ on a compact interval $X\subset \R$.
We assume that
that the system $\{f_j\}_{j \in \Ak}$ is uniformly hyperbolic and contractive:
\be \label{hyper}
0 < \gam_1 \le |f_j'(x)| \le \gam_2 < 1\ \ \mbox{for all}\ j \in \Ak,\ x\in X.
\ee
As before, $\Om = \Ak^\N$ and $\sigma$ denotes the left shift on $\Om$. We write $\Pi:\Om\to \R$ for the natural projection map associated with the IFS.
Consider the pressure function, defined by
\be \label{eq-pressure}
P_\Ak(t) = P_\Psi(t)=\lim_{n\to \infty} n^{-1} \log \sum_{u\in \Ak^n} \|f'_u\|^t.
\ee
It is well-known that this limit exists, $t\mapsto P_\Ak(t)$ is continuous and strictly decreasing. According to the general theory of thermodynamical formalism (see e.g., \cite{PU}),
$$
P_\Psi(t) = P(\sig,t \phi),
$$
where  $\phi(\om) = \log|f_{\om_1}'(\Pi(\sig \om))|$ is the potential associated with the IFS {and $P(\sigma, \cdot)$ is the topological pressure}. 
The {\em equilibrium state} for the potential $t\phi$ is a Borel probability measure $\mu$ on $\Om$ satisfying
$$
P_\Psi(t) = h_\mu + t\int \phi\,d\mu,
$$
where $h_\mu = h_\mu(\sig)$, see \cite[3.5]{PU}. Observe that $\int \phi\,d\mu=-\chi_\mu$ by the definition of the Lyapunov exponent.
Denote by $s=s(\Psi)$ the solution of the Bowen's equation:
\be \label{Bowen}
s=s(\Psi):\ \ P_\Psi(s) = 0.
\ee
It is well-known  that $s(\Psi)$ is the upper bound for the Hausdorff dimension of the attractor.
We say that $\mu$ is an {\em equilibrium measure} for the IFS $\Psi$ if it is the equilibrium state for the potential $s(\Psi)\cdot \phi$. Thus, by definition,
$$
\mu\ \mbox{is an equilibrium measure}\ \implies\ s(\Psi) = \frac{h_\mu}{\chi_\mu}\,.
$$ 
The equilibrium measure is the {\em dimension-maximizing measure} for the IFS in the symbolic space. Under our assumptions, the equilibrium measure $\mu$ is the unique Gibbs measure for the
potential $s\phi = s(\Psi)\cdot \phi$, which implies that 
$$
\mu([u]) \asymp \diam([u])^s,
$$
for any cylinder set $[u]$ in $\Om$. Here $\diam([u])$ is the diameter in the metric associated with the IFS: $d(\om,\tau) = |X_{\om\wedge\tau}|$. It follows that $\mu$ has local dimension $s$ at 
{\em every point} in $\Om$; in particular, the correlation dimension $\dim_{\rm cor}(\mu) = s$.

Given a family of hyperbolic IFS $\Psi^\lam$ (with overlaps) depending on a parameter $\lam\in \ov U$, with equilibrium measure $\mu_\lam$, 
we expect that {\em typically}, in the sense of almost every parameter, the projection of the equilibrium measure $(\Pi^\lam)_* \mu_\lam$ has Hausdorff dimension
$\min\{1, s(\Psi^\lam)\}$, and is absolutely continuous when $s(\Psi^\lam)>1$. This is what we prove under the assumptions of regularity and transversality. It is a simple consequence of 
Theorem~\ref{thm:main_gibbs}, but we state it as a theorem because of its importance.

\begin{thm} \label{th-equilib}
	Let $\Psi^\lam=\{f_j^\lam\}_{j \in \Ak}$ be a parametrized IFS satisfying smoothness assumptions \ref{as:C2} - \ref{as:hyper} and the transversality condition \ref{as:trans} on $U$.
	Let $\mu_\lam$ be the equilibrium measure for $\Psi^\lam$ and $s(\Psi^\lam)$ the solution of the Bowen's equation \eqref{Bowen}. Then $\dim_H((\Pi^\lam)_* \mu_\lam) = \min\{1, s(\Psi^\lam)\}$
	for a.e.\ $\lam\in U$ and $(\Pi^\lam)_* \mu_\lam$ is absolutely continuous with a density in $L^2$ for Lebesgue almost every $\lam$ in the set 
	$\{ \lam \in U : s(\Psi^\lam) > 1 \}$.
\end{thm}

\begin{proof} As noted above, the equilibrium measure $\mu_\lam$ satisfies $\dim_{cor}(\mu_\lam) = s(\Psi^\lam)$. By Theorem~\ref{thm:main_hausdorff} and Theorem~\ref{thm:main_cor_dim}, it is enough to show that the equillibrium measure $\mu_\lam$ satisfies \ref{as:measure}. By Proposition~\ref{eq:regeq}, it is enough to show that potential $\phi^\lambda(\omega)=s(\Psi^\lambda)\log\left|(f_{\omega_1}^\lambda)'(\Pi^\lambda(\sigma\omega))\right|$ satisfies \eqref{eq:unifvar} and \eqref{eq:assong}. 
	
	The condition \eqref{eq:unifvar} is straightforward, since by assumption $\gamma_1<\left|(f_{\omega_1}^\lambda)'(\Pi^\lambda(\sigma\omega))\right|<\gamma_2$ on $U$ and trivially $s(\Psi^\lambda)\leq\frac{\log m}{-\log\gamma_2}.$ On the other hand,
	\[
	\begin{split}
	|\phi^{\lambda}(\omega)-\phi^\tau(\omega)|&=\Big|s(\Psi^\lambda)\log\left|(f_{\omega_1}^\lambda)'(\Pi^\lambda(\sigma\omega))\right|\log\left|(f_{\omega_1}^\lambda)'(\Pi^\lambda(\sigma\omega))\right|-s(\Psi^\tau)\log\left|(f_{\omega_1}^\tau)'(\Pi^\tau(\sigma\omega))\right|\Big|\\
	&\leq -\log\gamma_1|s(\Psi^\lambda)-s(\Psi^\tau)|+\frac{\log m}{-\log\gamma_2}\Big|\log|(f_{\omega_1}^\lambda)'(\Pi^\lambda(\sigma\omega))|-\log\left|(f_{\omega_1}^\tau)'(\Pi^\tau(\sigma\omega))\right|\Big|\\
	&\leq -\log\gamma_1|s(\Psi^\lambda)-s(\Psi^\tau)|+\frac{\log m}{-\gamma_1\log\gamma_2}\Big|(f_{\omega_1}^\lambda)'(\Pi^\lambda(\sigma\omega))-(f_{\omega_1}^\tau)'(\Pi^\tau(\sigma\omega))\Big|.
	\end{split}
	\]
	By the assumptions \ref{as:C2} - \ref{as:hyper}, simple manipulation shows that $\lambda\mapsto(f_{\omega_1}^\lambda)'(\Pi^\lambda(\sigma\omega))$ is a Lipschitz map with Lipschitz constant independent of $\omega$. Hence, it is enough to show that $\lambda\mapsto s(\Psi^\lambda)$ is Lipschitz. But clearly,
	$$
	-\log\gamma_2|s-t|\leq|P_{\Phi^\lambda}(t)-P_{\Phi^\lambda}(s)|\leq-\log\gamma_1|s-t|,
	$$
	and so 
	\[
	\begin{split}
	|s(\Psi^\lambda)-s(\Psi^\tau)|&\leq(-\log\gamma_2)^{-1}|P_{\Phi^\lambda}(s(\Psi^\lambda))-P_{\Phi^\lambda}(s(\Psi^\tau))|\\
	&=(-\log\gamma_2)^{-1}|P_{\Phi^\lambda}(s(\Psi^\tau))-P_{\Phi^\tau}(s(\Psi^\tau))|\\
	&\leq (-\log\gamma_2)^{-1}c|\lambda-\tau|,
	\end{split}
	\]
	where the last inequality follows by Lemma~\ref{lem:eigen} since $\lambda\mapsto s(\Psi^\tau)\log\left|(f_{\omega_1}^\lambda)'(\Pi^\lambda(\sigma\omega))\right|$ satisfies \eqref{eq:assong}.
\end{proof}

\subsection{Natural measures for non-homogeneous self-similar IFS} 
Consider a self-similar IFS on the line $\Fk=\{f_j(x) = r_j x + a_j\}_{j\in \Ak}$, where $r_j \in (0,1)$ and $a_j\in \R$. Recall that the
similarity dimension is the number $s=s(\Fk)$, such that $\sum_{j\in \Ak} r_j^s = 1$. Assume that the IFS is non-degenerate, in the sense that the fixed points of $f_j$ are all distinct. In this case the equilibrium measure is the Bernoulli product measure $(\pb^\N)$ on $\Om$, where $\pb = (r_1^s,\ldots,r_m^s)$ is the vector of probability weights associated with the similarity dimension. We focus on the question of absolute continuity for the {\em natural} self-similar measure $\nu_\Fk = \Pi_*(\pb^\N)$. (For the Hausdorff dimension $\dim_H(\nu_\Fk)$ Hochman \cite{Hoch} obtained results that are much sharper than what we get with our method, so we don't discuss the latter.) For non-homogeneous self-similar measures results on absolute continuity for a typical parameter in a ``transversality region'' were obtained by Neunh\"auserer \cite{Neun} and Ngai and Wang \cite{NW} independently. However,
in their results the probabilities in the definition of self-similar measure are fixed, and so nothing can be claimed for the natural measure for a.e.\ parameter. More recently, Saglietti, Shmerkin, and Solomyak \cite{SSS} proved absolute continuity for a.e.\ parameter in the entire ``super-critical region'' (i.e., where $h_\mu/\chi_\mu>1$), however, there also, probabilities are fixed, and an application of Fubini's Theorem doesn't yield anything for the natural measure. The following is an immediate consequence of
Theorem~\ref{th-equilib}.

\begin{cor} \label{cor-nonhom}
	Let $\Fk_\lam = \{r_j(\lam) x + a_j(\lam)\}_{j\in \Ak}$ be a family of non-degenerate self-similar IFS satisfying smoothness assumptions \ref{as:C2} - \ref{as:hyper} and the transversality condition \ref{as:trans} on $U$. Then the natural self-similar measure $\nu_\lam$ is absolutely continuous with a density in $L^2$ for a.e.\ $\lam\in U$ such that the similarity dimension is strictly greater than 1.
\end{cor}

Specific regions where the transversality condition holds were found in \cite{Neun,NW}. In particular, we have the following for the family of the IFS $\{\lam_1 x , \lam_2 + x\}$, where the
1-parameter family is obtained by assuming $\lam = \lam_1,\ \lam_2 = c\lam$ for a fixed $c>0$.

\begin{cor} \label{cor-nonhom2}
	Let $\nu_{\lam_1,\lam_2}$ be the natural self-similar measure for the IFS $\{\lam_1 x , \lam_2x+1\}$. Then $\nu_{\lam_1,\lam_2}$ is 
	absolutely continuous with a density in $L^2$ for a.e.\ $(\lam_1,\lam_2)$ such that $\lam_1 + \lam_2 >1$ and $\max\{\lam_1,\lam_2\} \le 0.668$.
\end{cor}

\subsection{Some random continued fractions}\label{sec:rcf}

Consider the IFS $\Fk_{\alpha,\beta}=\{f_1,f_2\} =: \{\frac{x+\alpha}{x+\alpha+1},\frac{x+\beta}{x+\beta+1}\}$ on the real line, for $0 \le \alpha < \beta$. Applying the maps randomly (not necessarily independently), we obtain a random continued fraction $[1,Y_1, 1, Y_2, 1, Y_3,\ldots]$ where $Y_i \in \{\alpha,\beta\}$ and we are using the notation
$$
[a_1,a_2,a_3,\ldots] :=
\frac{1}{a_1+ \displaystyle{\frac{1}{a_2 +
			\displaystyle{\frac{1}{a_3+ \ldots}}}}}
$$
In the case $\alpha=0$ the IFS is {\em parabolic}; it was first studied by Lyons \cite{Lyons}, motivated by a problem from the theory of Galton-Watson trees. In \cite{SSU Parabolic} it was shown that the invariant measure
for the IFS corresponding to $Y_i$ applied i.i.d., with probabilities $(\half,\half)$ is absolutely continuous for a.e.\ $\beta\in (0.215, \beta_c)$, where $\beta_c\in (0.2688, 0.2689)$ is the ``critical value'',
such that
$$
\frac{\log 2}{\chi_{\beta_c}}=1,
$$
where $\chi_{\beta_c}$ is the Lyapunov exponent of the measure $(\half,\half)^\N$. 
Note that the IFS $\Fk_{0,\beta}$ is overlapping, i.e., its two cylinder intervals have
non-trivial intersection, for $\beta\in (0,0.5)$.

\medskip

In this paper we restrict ourselves to smooth {\em hyperbolic} IFS, so we need to take $\alpha>0$. However, we can take a very small positive $\alpha$ and expect somewhat similar behavior.
The convex hull of the attractor for $\Fk_{\alpha,\beta}$ is the closed interval having the attracting fixed points of $f_1, f_2$ as its endpoints; it is
$
X_{\alpha,\beta} = \Bigl[\frac{\sqrt{\alpha^2 + 4\alpha} - \alpha}{2}, \frac{\sqrt{\beta^2 + 4\beta} - \beta}{2}\Bigr].
$
It is easy to check that the condition for the IFS to be overlapping, i.e., $\Lk^1\bigl(f_1(X_{\alpha,\beta}) \cap f_2(X_{\alpha,\beta}) \bigr)>0$ is
$$
\beta + \alpha + 4 > 3\bigl(\sqrt{\beta^2+4\beta} + \sqrt{\alpha^2 + 4\alpha}\bigr).
$$
It is satisfied, e.g., when $\alpha \in (0, 10^{-4}]$ and $\beta \in (\alpha,0.485)$. 

\begin{example} \label{ex-cf}
	Denote by $\Pi^{\alpha,\beta}$ the natural projection from $\Om=\{1,2\}^\N$ to the attractor and consider the equilibrium Gibbs measure $\mu_{\alpha,\beta}$ for the IFS.
	Fix $\alpha \in (0, 10^{-4}]$ and $\beta = \sqrt{2}-1 = 0.41421\ldots$ Denote 
	$\eta_{\alpha,\beta}:=\Pi^{\alpha,\beta}_* \mu_{\alpha,\beta}$. Then $\eta_{\alpha, \beta+\lam}$ is absolutely continuous with a density in $L^2$ for 
	a.e.\ $\lam\in U = (0,0.485-\beta)\approx 
	(0,0.077)$.
\end{example}

\begin{sloppypar}
In order to derive this claim from Theorem~\ref{th-equilib} we need to check transversality and that $h_{\mu_{\alpha,\beta}}/\chi_{\mu_{\alpha,\beta}}>1$ holds. (The regularity assumptions are obviously satisfied.) It is well-known  
that as soon as there is an overlap, the condition $s(\Psi_{\alpha,\beta})=h_{\mu_{\alpha,\beta}}/\chi_{\mu_{\alpha,\beta}}>1$ is satisfied, but for the reader's convenience we provide a short proof in Appendix \ref{app:pressure drop}, see Corollary \ref{cor-drop2}. Checking transversality is non-trivial; we indicate it in the next subsection. (In fact, we could get a larger interval of transversality 
$(\approx 0.215, 0.485)$ for $\alpha \in (0, 10^{-4}]$ with the method of 
\cite[Section 6]{SSU Parabolic}, which is more delicate.) 
\end{sloppypar}

\subsection{Checking transversality}\label{sec:tran}
Sometimes slightly different forms of the transversality conditions are used. Here they are:

\begin{align}
\begin{split}\label{tran1}
\exists&\,\eta>0:\ \forall\, u,v\in \Om,\ \ u_1 \ne v_1, \  \lam\in \ov U\\
&\qquad\left|\Pi^\lam(u) - \Pi^\lam(v)\right| \le  \eta \implies \left|\textstyle{\frac{d}{d\lam}}(\Pi^\lam(u) - \Pi^\lam(v))\right| \ge \eta;
\end{split}
\\[8pt]
\begin{split}\label{tran2}
\exists&\,\eta>0:\ \forall\, u,v\in \Om,\ \ u_1 \ne v_1, \ \lam\in \ov U\\
&\qquad \Pi^\lam(u) = \Pi^\lam(v) \implies \left|\textstyle{\frac{d}{d\lam}}(\Pi^\lam(u) - \Pi^\lam(v))\right| \ge \eta;
\end{split}
\\[8pt]
\begin{split}\label{tran3}
\exists&\, C_T>0:\  \forall\, u,v\in \Om,\ \ u_1 \ne v_1,\ r>0 \\
&\qquad \Lk^1\left\{\lam\in \ov U:\ |\Pi^\lam(u) - \Pi^\lam(v)| \le r \right\} \le C_T \cdot r.
\end{split}
\end{align}

\begin{lem} \label{lem-tran}
	Under regularity assumptions \ref{as:C2} - \ref{as:hyper}, all three conditions \eqref{tran1} - \eqref{tran3} are equivalent.
\end{lem}

\begin{proof}
	The implication \eqref{tran1}$\implies$\eqref{tran2} is trivial.
	
	\begin{sloppypar}
		The implication \eqref{tran1} $\implies$\eqref{tran3} is the usual transversality argument, see \cite[Lemma~7.3]{SSU1}.
	\end{sloppypar}
	
	\smallskip
	
	Let us prove \eqref{tran3}$\implies$\eqref{tran2}. We argue by contradiction. If \eqref{tran2} does not hold, we can use compactness of $\Om$ and $\ov U$ and find $u,v\in \Om$ with
	$u_1 \ne v_1$, and $\lam_0 \in \ov U$ such that $F(\lam) = \Pi^\lam(u) - \Pi^\lam(v)$ satisfies
	$$
	F(\lam_0) = \frac{d}{d\lam} F(\lam_0) = 0.
	$$
	Using that $\Pi^\lam\in C^{1,\delta}$ (Proposition \ref{prop:pi hoelder}), we can write
	\begin{eqnarray*}
		|F(\lam_0 + t)| & = & |F(\lam_0 + t) - F(\lam_0) - F'(\lam_0) t| \\
		& = & | F'(\lam_0 + \tau)t  - F'(\lam_0) t| \ \ \ \ \mbox{for some $\tau\in (0,t)$ by the Lagrange Theorem}\\
		& = & |t|\cdot |F'(\lam_0 + \tau) - F'(\lam_0)| \le |t| \cdot C_\delta|\tau|^{\delta} < C_\delta|t|^{1+\delta},
	\end{eqnarray*}
	which clearly contradicts \eqref{tran3} for $r$ sufficiently small.
	
	\smallskip
	
	It remains to show \eqref{tran2}$\implies$\eqref{tran1}, but this again follows by compactness of $\Om$ and $\ov U$ and continuity of  $\lam\mapsto \Pi^\lam$ and
	$\lam\mapsto \frac{d}{d\lam}\Pi^\lam$.
\end{proof}

Next we consider two 1-parameter families of IFS for which it is possible to verify the transversality condition, under appropriate assumptions. They are variants and modifications of the
parametrized families of IFS from \cite{SSU1,SSU Parabolic}.

\begin{proof}[Proof of transversality in Example~\ref{ex-cf}]
	Let $f(x) = \frac{x}{x+1}$, so that $\Fk^\lam = \{f(x+\alpha), f(x + \beta+\lam)\}$, and let $\Pi^\lam$ be the corresponding natural projection map. We can consider this IFS on 
	$X = [0,0.5]$ for all these parameters.
	Here it is more convenient to verify the transversality condition in the form \eqref{tran2}. Let $u,v\in \Om$ with $u_1\ne v_1$. 
	Without loss of generality we can assume that $u_1=2$ and $v_1=1$. Then we have by the Lagrange Theorem,
	\begin{eqnarray*}
		\Pi^\lam(u) - \Pi^\lam(v) & = & f\bigl(\beta + \lam + \Pi^\lam(\sig u)\bigr) - f\bigl(\alpha + \Pi^\lam(\sig v)\bigr) \\
		& = & f'(c)\cdot \left[\beta-\alpha+\lam + \Pi^\lam(\sig u) - \Pi^\lam(\sig v)\right] \\
		& =: & f'(c)\cdot \Psi^\lam(u,v).
	\end{eqnarray*}
	Since $f'(c)\ge \gam_1>0$, we obtain that
	$$
	\left\{\lam\in \ov U:\ |\Pi^\lam(u) - \Pi^\lam(v)| \le r \right\} \subset \left\{\lam\in \ov U:\ |\Psi^\lam(u,v)| \le r/\gam_1 \right\}.
	$$
	In order to verify \eqref{tran3}, it suffices to show that $\frac{d}{d\lam} \Psi^\lam(u,v)\ge\delta>0$.
	We have
	\be \label{ineqa}
	\textstyle{\frac{d}{d\lam}} \Psi^\lam(u,v) = 1 + \frac{d}{d\lam} \Pi^\lam(\sig u) - \frac{d}{d\lam} \Pi^\lam(\sig v) \ge 1 - \frac{d}{d\lam} \Pi^\lam(\sig v),
	\ee
	using monotonicity. 
	We can write
	$$
	\Pi^\lam(\sig v) = f_1^{i_0} f_2^{\lam} f_1^{i_1} f_2^\lam f_1^{i_2} f_2^\lam\ldots
	$$
	for some $i_n\ge 0$, where we write $f_1 \equiv f_1^\lam = f(x+\alpha)$ and $f_2^\lam = f(x + \beta + \lam)$, so that
	$$
	\Pi^\lam(\sig v)  = f_1^{i_0}f\bigl(\beta + \lam + f_1^{i_1} f(\beta + \lam + f_1^{i_2}\ldots) \bigr).
	$$
	Then simply using that $\|f_1'\|_\infty < 1$ and the maximum of the derivative is attained at the left endpoint by concavity, yields
	$$
	{\textstyle{\frac{d}{d\lam}}} \Pi^\lam(\sig v) < f'(\beta + \lam)\Bigl(1 + f'(\beta + \lam)\bigl(1 + f'(\beta+\lam)(1 + \cdots)\bigr)\Bigr) = \frac{f'(\beta+\lam)}{1 - f'(\beta+\lam)}.
	$$
	It remains to note that $f'(\beta+\lam) < f'(\beta) = 1/2$, hence ${\textstyle{\frac{d}{d\lam}}} \Pi^\lam(\sig v) < 1$, which implies the desired claim, in view of \eqref{ineqa}.
\end{proof}

\subsection{``Vertical'' translation family}
Next we consider a class of 1-parameter families of IFS for which it is possible to verify the transversality condition, under appropriate assumptions. This is also a modification of the
parametrized families of IFS from \cite{SSU1,SSU Parabolic}.

Let $\{f_j\}_{j\in \Ak}$ be a $C^{1+\delta}$ IFS on $X$ and consider a ``translation perturbation'' $\{f_j^\lam\}_{j \in \Ak}$, satisfying \ref{as:hyper},
of the following form:
assume that
$$
\{f^\lam_j(x) = f_j(x) + a_j(\lam)\}_{j\in \Ak},
$$
and assume that it is well-defined on $X$ for $\lam\in \ov U$. We call it ``vertical'' because the graphs are translated vertically.
Sometimes it is useful to consider IFS consisting of ``horizontal'' shifts of the same function, that is, IFS of the form $\{f(x+c_j)\}_{j=1}^m$, like Example~\ref{ex-cf}.
Such families may be treated in a way similar to the ``vertical''
translation families with a few modifications, see \cite[Section 7]{SSU1} and \cite[Section 6]{SSU Parabolic}. Instead of treating this case in full generality, we focused on a specific example of random continued fractions above.

Denote for $i\ne j$ in $\Ak$:
\be \label{def-I}
X_{ij}:= \left\{x\in X:\ \exists\,\lam\in \ov U,\ \exists\,y\in X\ \mbox{such that} \ f^\lam_i(x) = f^\lam_j(y)\right\}.
\ee
Note that $X_{ij}$ is empty if the corresponding 1st order cylinders never overlap.
We further define, for $i\ne j$ in $\Ak$ such that $X_{ij}\ne \es$:
\be \label{def-eta}
{\|f'_i\|}_{X_{ij}}:= {\|f'_i|_{X_{ij}}\|}_\infty,\ \ \ \ \eta_{ij} := \min\Bigl|\textstyle{\frac{d}{d\lam}}\bigl[a_i(\lam) - a_j(\lam)\bigr]\Bigr|.
\ee
Let
\be \label{dmax}
D_{\max}:= \max_i\Bigl(\frac{{\|\frac{d}{d\lam}a_i\|}_\infty}{1 - {\|f_i'\|}_\infty}\Bigr)\,.
\ee

\begin{prop} \label{prop-trans-ex}
	{\bf (i)} If
	\be \label{cond1}
	\eta_{ij} - \bigl(\|f'_i\|_{X_{ij}} +\|f'_j\|_{X_{ji}} \bigr)\cdot D_{\max}>0\ \ \ \mbox{for all $i\ne j$ such that $X_{ij}\ne \es$},
	\ee
	then the transversality condition holds on $U$.
	
	{\bf (ii)} Assume, in addition, that $f_j'(x) > 0$ and $\frac{d}{d\lam}a_j\ge 0$ for all $j\in \Ak$. If
	\be \label{cond2}
	\eta_{ij} - \|f'_j\|_{X_{ji}} \cdot D_{\max}>0\ \ \ \mbox{for all $i\ne j$ such that $X_{ij}\ne \es$},
	\ee
	then the transversality condition holds on $U$.
\end{prop}

Before the proof we present a more familiar special case.
Let $\{f_j^\lam\}_{j \in \Ak}$ be a $C^{1+\delta}$ IFS on $X$, satisfying \ref{as:hyper}.
Consider the translation family
$$
\{f^\lam_1(x) = f_1(x) + \lam,\ f^\lam_j(x)= f_j(x),\ j> 1\},
$$
and assume that it is well-defined on $X$ for $\lam\in \ov U$.
Note that only $f_1^\lam$ changes with $\lam$.
Moreover, we assume that only the cylinder $f^\lam_1(X)$ can intersect other 1-st order cylinders, that is
$$
i\ne j,\ f_i(X) \cap f_j(X) \ne \es \implies 1 \in \{i,j\}.
$$

\begin{cor}
	{\bf (i)} If
	$$
	2{\|f'_1\|}_\infty +  {\|f'_j\|}_{X_{j1}} < 1\  \ \mbox{for all}\ \ 1<j \le m,
	$$
	then the transversality condition  holds on $U$.
	
	{\bf (ii)} Assume, in addition, that $f_j'(x) > 0$ for all $j\in \Ak$. If
	$$
	{\|f'_1\|}_\infty +  {\|f'_j\|}_{X_{j1}} < 1\ \ \mbox{for all}\  \ 1<j \le m,
	$$
	then the transversality condition holds on $U$.
\end{cor}

The derivation of the corollary from the proposition is immediate, since in this case we have $\eta_{1j} = 1$ for $j>1$ and $D_{\max} = (1 - {\|f'_1\|}_\infty)^{-1}$.

\begin{proof}[Proof of Proposition~\ref{prop-trans-ex}]
	Consider the symbolic cylinder sets $[i]\subset \Om$ and let
	$$
	M_\infty:= \max_{u\in \Om} \Bigl\|\textstyle{\frac{d}{d\lam}}\Pi^\lam(u)\Bigr\|_\infty,\ \ \ M_i: = \max_{u\in [i]} \Bigl\|\frac{d}{d\lam} \Pi^\lam(u)\Bigr\|_\infty,\ \ i\in \Ak.
	$$
	We have
	$$
	u\in [i] \implies \Pi^\lam(u) = a_i(\lam) + f_i(\Pi^\lam(\sig u)),
	$$
	hence
	\be \label{deriva}
	\frac{d}{d\lam} \Pi^\lam(u) = \frac{d}{d\lam} a_i(\lam) + f_i'(\Pi^\lam(\sig u))\cdot \frac{d}{d\lam} \Pi^\lam(\sig u)\ \ \ \mbox{for}\ \ u\in [i].
	\ee
	It follows that
	$$
	M_i \le \Bigl\|\textstyle{\frac{d}{d\lam}} a_i(\lam) \Bigr\|_\infty + {\|f'_i\|}_\infty \cdot M_\infty,
	$$
	and since $M_\infty = \max_i M_i$, we obtain from \eqref{dmax} that
	\be \label{ineq1}
	M_\infty \le D_{\max}.
	\ee
	
	Now we verify the transversality condition in the form \eqref{tran2}. If $\Pi^\lam(u) = \Pi^\lam(v)$ and $u_1\ne v_1$, then $u\in [i]$ and $v\in [j]$ for some $i\ne j$ such that $X_{ij}\ne \es$.
	Without loss of generality we can assume that $\frac{d}{d\lam}\bigl[a_i(\lam) - a_j(\lam)\bigr]>0$ in the definition of $\eta_{ij}$, otherwise, exchange $i$ and $j$.
	Then \eqref{deriva} yields
	\be \label{ura1}
	\frac{d}{d\lam} \left (\Pi^\lam(u) - \Pi^\lam(v)\right)  = \frac{d}{d\lam}\bigl[a_i(\lam) - a_j(\lam)\bigr]
	+ f_i'(\Pi^\lam(\sig u))\cdot \frac{d}{d\lam} \Pi^\lam(\sig u) - f_j'(\Pi^\lam(\sig v))\cdot \frac{d}{d\lam} \Pi^\lam(\sig v).
	\ee
	Note that
	$$
	\Pi^\lam(u) = f_i(\Pi^\lam(\sig u)) = f_j(\Pi^\lam(\sig v)) = \Pi^\lam(v),
	$$
	hence $\Pi^\lam(\sig u)\in X_{ij}$ and  $\Pi^\lam(\sig v)\in X_{ji}$.
	Therefore, \eqref{ura1} yields
	$$
	\left|\frac{d}{d\lam} \left (\Pi^\lam(u) - \Pi^\lam(v)\right)\right| \ge \eta_{ij}  - \bigl(\|f'_i\|_{X_{ij}} +\|f'_j\|_{X_{ji}} \bigr)\cdot D_{\max}>0,
	$$
	assuming  \eqref{cond1}.
	This proves part (i) of the proposition.
	
	In order to verify part (ii), note that if all $f_j$ and $\lam\mapsto a_j(\lam)$ are monotone increasing, we also get that $\frac{d}{d\lam} \Pi^\lam(u)\ge 0$ for all $u\in \Om$, hence \eqref{ura1} implies
	$$
	\left|\frac{d}{d\lam} \left (\Pi^\lam(u) - \Pi^\lam(v)\right)\right| \ge \eta_{ij}  - \|f'_j\|_{X_{ji}} \cdot D_{\max}>0,
	$$
	which is bounded away from zero under the assumption \eqref{cond1}. This concludes the proof of \eqref{tran2}
\end{proof}


\begin{example}\label{z94}
	Let $\Psi:= \left\{ f_i \right\}_{i=1}^{m }$ be a $C^{1+\delta}$ IFS on $X$. We assume that there exists a partition
	$\mathcal{A}=\mathcal{I}_{-1}\cup\mathcal{I}_1$ such that
	for very $i,j\in\mathcal{I}_k$, we have
	\begin{equation}
	\label{z98}
	f _{i } (X)\cap f _{j}(X)=\emptyset,
	\quad
	i\ne j,\ i,j\in\mathcal{I}_k, k=-1,1
	\end{equation}
	Recall the definition of $\gamma _2$ from
	\ref{as:hyper}. Besides \eqref{z98}, our second assumption is as follows:
	\begin{equation}\label{z96}
	\gamma  _2<\frac{1}{2}.
	\end{equation}
	
	We define $\kappa(i)=k$ if $i\in \mathcal{I}_k$, $k=-1,1$.
	Then we introduce the family $\Psi^\lambda =\left\{ f _{ i}^{ \lambda } \right\}_{i=1}^{m }$ with a parameter interval $\lambda \in U$, where
	\begin{equation}
	\label{z97}
	f _{ i}^{ \lambda }(x):=f_i(x)+\kappa(i)\cdot \lambda.
	\end{equation}
	Together with \eqref{z96}, this yields
	\begin{equation}
	\label{z70}
	\left|
	\frac{d}{d\lambda }(a_i(\lambda )-a_j(\lambda ))
	\right|\equiv
	\left\{
	\begin{array}{ll}
	2
	,&
	\hbox{if $\kappa (i)\ne\kappa (j)$;}
	\\
	0   
	,&
	\hbox{if $\kappa (i)=\kappa (j)$.}
	\end{array}
	\right.
	\text{ and }
	D_{\max}\leq\frac{1}{1-\gamma _2}<2.
	\end{equation}
	The parameter interval $U$ is an open interval centered at $0$, and $U$ is so small that 
	\begin{equation}
	\label{z75}
	f _{ i}^{ \lambda }(X)\subset \text{int}(X),\text{ and }
	f^\lambda _{i } (X)\cap f^\lambda  _{j}(X)=\emptyset,
	\quad
	i\ne j,\ i,j\in\mathcal{I}_k,\  k=-1,1,\ \lambda \in \overline{U }.
	\end{equation}
	The (first level) cylinder intervals are $X _{i}^{ \lambda }:=f _{ i}^{ \lambda }(X)$, $i\in\mathcal{A}$ and $\lambda \in U$.
	Observe that
	\begin{equation}
	\label{z71}
	X_{ij}\ne \emptyset \Longleftrightarrow
	\exists \lambda \in \overline{U },\ X_{i}^{\lambda }\cap X_{j}^{\lambda }\ne \emptyset.
	\end{equation}
	Using this and  \eqref{z70} we obtain 
	\begin{equation}
	\label{z95}
	X_{ij}\ne \emptyset
	\Longrightarrow
	\text{ either }
	\left(  
	i\in \mathcal{I}_{-1}\  \& \ j\in \mathcal{I}_{1}\right)
	\text{ or }
	\left(  
	j\in \mathcal{I}_{-1}\  \&\  i\in \mathcal{I}_{1}\right)
	\Longrightarrow
	\eta _{ij}=2.
	\end{equation}
	Putting together this and the second part of \eqref{z70} we obtain that 
	\eqref{cond1} holds and consequently
	the transversality condition holds on $U$.
\end{example}

\bigskip

\begin{rem}\label{z91}
	The partition $\mathcal{A}=\mathcal{I}_{-1}\cup \mathcal{I}_1$
	satisfying \eqref{z98} exists, for example, if every point in $X$ is covered by at most two level-1 cylinder intervals. That is
	\begin{equation}
	\label{z93}
	\sum_{i=1}^{m}\ind_{f_i(X)}\leq 2.
	\end{equation}
	In fact,
	let $[a_j,b_j]:=X_j:=f_j(X)$. Without loss of generality, we may assume that the cylinder intervals $X_j$  are ordered in such a way 
	that the left endpoints are in increasing order. If two level-1 cylinder intervals share the same left endpoint, that is,  $a_j=a_{j+1}$, then  we set $|X_j|\geq|X_{j+1}|$.
	Define $\mathcal{I}_1$ inductively, as follows:
	$1\in\mathcal{I}_1$. 
	If the set $\mathcal{I}_1$ already contains $1=n_1<n_2<\cdots <n_\ell$, then we let
	$n_{\ell +1}:=\min \left\{ j\in\mathcal{A}:\, b_{\ell  }<a_{j } \right\}$, if such  $a_j$ exists; otherwise, we stop and
	set
	$\mathcal{I}_{-1}:=\mathcal{A}\setminus \mathcal{I}_1$.
	It is easy to see that \eqref{z98} holds.
\end{rem}

\begin{rem}\label{z69}
	If we consider an IFS like in Example \ref{z94} but  allow that every point is covered by at most $2\ell +1$ cylinder intervals for
	$\ell \geq 1$ and  assume that  $\gamma _2<\frac{1}{2\ell +1}$, then we get that the transversality condition holds in the same way. Namely, we can partition $\mathcal{A}$ into $2\ell +1$  families $\mathcal{I}_{-\ell} ,\dots  \mathcal{I}_\ell $
	in such a way that
	there are no intersections between distinct cylinder intervals  from the same family.
	For all functions corresponding to the family $\mathcal{I}_k$
	the translation is defined to be $k\cdot \lambda $.    
	Then the minimal value of $\eta _{ij}$ is equal to 1 and $D_{\max}\leq\frac{\ell}{1-\gamma _2}$. This implies that \eqref{cond1} holds if $\gamma _2<\frac{1}{2\ell+1}$.
\end{rem}

\begin{definition}\label{z82}
	We say that $\mathfrak{A}$ is a transversality-typical property of sufficiently smooth IFSs if the following holds: Whenever $\left\{ \Psi^\lambda  \right\}_{\lambda \in U}$ is a
	one-parameter family of sufficiently smooth IFSs for which the transversality condition holds then for $\mathcal{L}_1$ almost all $\lambda \in U$
	the IFS $\Psi^\lambda$ has property $\mathfrak{A}$.     
\end{definition}

We use the notation of Example \ref{z94}. In particular, we are
given a compact interval $X\subset \mathbb{R}$ and a $C^{1+\delta}$ IFS
$\left\{ f_i \right\}_{i=1}^{m }$ on $X$ such that 
\begin{equation}
\label{z79}
X_i:=f_i(X)\subset \mathrm{int}(X) \text{ for all } i\in\mathcal{A}.
\end{equation}
Below we consider a translation perturbation family of $\Psi$. That is, 
\begin{equation}
\label{z78}
\Psi^{\mathbf{t}}:=\left\{ f _{i}^{\mathbf{t}} \right\}_{i=1}^{m}, \quad
f _{i}^{\mathbf{t}}(x):=f_i(x)+t_i,\quad 
\mathbf{t}\in B(0,\delta _0),
\end{equation}
where $\delta _0>0$ is so small that \eqref{z79} 
holds if we replace  $f_i$ with $f_i^{\mathbf{t}}$ and $X_i$ with 
$X _{i}^{\mathbf{t} }:=f_i^{\mathbf{t}}(X)$ for all $i\in\mathcal{A}$.

\begin{claim}\label{z90}
	Assume that
	\begin{enumerate}
		[{\bf (a)}]
		\item  all points of $X$ are covered by at most two 
		of the cylinder intervals $X_k$ and 
		\item $\gamma _2<1/2$.       
	\end{enumerate}

	Let $\mathfrak{A}$ be a transversality-typical property. 
	Then there exists $0<\delta _*\leq\delta _0$
	such that for $\mathcal{L}^m$-a.e. $\mathbf{t}\in B(0,\delta _*)$,
	the  translated IFS  $\left\{\Psi^{\mathbf{t}}\right\}_{i=1}^{m}$
	(defined in \eqref{z78})
	has property $\mathfrak{A}$.   
\end{claim}

\begin{sloppypar}
\begin{proof}
	Using Remark \ref{z91}, we can find a partition $\mathcal{A}=\mathcal{I}_{-1}\cup \mathcal{I}_1$ such that $f_i(X)\cap f_{j}(X)=\emptyset $ for
	distinct $i,j\in\mathcal{I}_k$, $k=-1,1$. 
	Let $\delta _1>0$ be so small that
	$0<4\delta_1<\delta _0 $ and
	\begin{equation}
	\label{z74}
	X_i\cap X_j=\emptyset \Longrightarrow
	X^{t}_i\cap X^{t}_j=\emptyset\quad
	\text{ for all  } \mathbf{t}\in B(0,4\delta _1)
	\end{equation}
	Hence
	\begin{equation}
	\label{z77}
	X^{\mathbf{}}_i\cap X^{\mathbf{t}}_j
	=\emptyset,
	\quad
	i\ne j,\ i,j\in\mathcal{I}_k,\  k=-1,1,\ \mathbf{t}\in B(0,4\delta _1).
	\end{equation}
	Let $U:=(-\frac{1}{\sqrt{m}}\delta _1,\frac{1}{\sqrt{m}}\delta _1)$ and for a $\lambda \in U$ we define
	$\widetilde{\mathbf{a}}(\lambda ):=(\kappa (1)\lambda ,\dots  ,\kappa(m)\lambda ) $, where we recall that $\kappa (i)=k$ if $i\in\mathcal{I}_k$.
	Finally, for a 
	$\mathbf{t}\in B(0,\delta _1)$  let 
	$$\mathbf{a}_{\mathbf{t}}(\lambda ):=
	\mathbf{t}+\widetilde{\mathbf{a}}(\lambda ).
	$$
	Then $\|\mathbf{a}_{\mathbf{t}}(\lambda )\|<2\delta _1,$ 
	$\mathbf{t}\in B(0,\delta_1 ),\ \lambda \in U$. Hence 
	\begin{equation}
	\label{z73}
	X_i^{\mathbf{a}_{\mathbf{t}}(\lambda )}\subset X\quad
	\text{ and }\quad
	X_i^{\mathbf{a}_{\mathbf{t}}(\lambda )}\cap
	X_j^{\mathbf{a}_{\mathbf{t}}(\lambda )}=\emptyset,\ i\ne j,\ i,j\in\mathcal{I}_k,\  k=-1,1,\ \lambda \in \overline{U }.
	\end{equation}
	Example \ref{z94} shows that
	\begin{equation}
	\label{z72}
	\text{ the transversality condition holds for the family }
	\left\{ \Psi^{\mathbf{a}_{\mathbf{t}}(\lambda )} \right\}_{\lambda \in U}
	\quad \text{ for all }
	\mathbf{t}\in B(0,\delta_1 ).
	\end{equation}
	Let 
	$$
	H:=\left\{ \pmb{\tau}\in B\left(0,\frac{\delta _1}{2\sqrt{m}}  \right)
	:
	\Psi^{\pmb{\tau}} \text{ does not have property } \mathfrak{A}
	\right\}.
	$$
	We need to prove that $\mathcal{L}^m(H)=0$. To get a contradiction assume that 
	$\mathcal{L}^m(H)>0$. Then $H$ has a Lebesque density point 
	$\widehat{\pmb{\tau}}\in B(0,\frac{\delta _1}{2\sqrt{m}})$. Let $V$ be the intersection of  
	$ B\left(0,\frac{\delta _1}{2\sqrt{m}}  \right)$ with the $(m-1)$-dimensional hyperplane which goes through the origin and is orthogonal to 
	the vector $(\kappa (1) ,\dots  ,\kappa(m) )$. Then by the Fubini theorem
	there exists a point $\mathbf{t}\in V$ such that 
	$\mathcal{L}^1\left\{ \lambda \in U:
	\mathbf{a}_{\mathbf{t}}(\lambda )\in H
	\right\}>0$. But this contradicts \eqref{z72} and the fact that $\mathfrak{A}$ is a transversality-typical property.
\end{proof}
\end{sloppypar}

\section{Open questions and further directions}\label{sec:open}

As Theorem \ref{thm:main_cor_dim} guarantees more refined properties of $(\Pi^\lam)_*\mu_\lam$ than mere absolute continuity, it is natural to ask whether a weaker condition than \ref{as:measure} is sufficient for an almost sure absolute continuity in the supercritical region $\left\{\lam : \frac{h_{\mu_\lam}}{\chi_{\mu_\lam}}>1\right\}$. In particular, is \ref{as:measure_cont} sufficient? In our case, condition \ref{as:measure} is needed to guarantee regularity of the error term $e_j(\omega_1,\omega_2,\lambda)$ from \eqref{eq:ej def}, allowing us to follow the approach of Peres and Schlag \cite{PS00}.

Another natural direction of further research is to generalise the main result for multivariable parameters. Peres and Schlag in \cite[Section~7]{PS00} were handling this case for fixed (parameter independent) measures. In the case of parameter-dependent measures with one-dimensional family of parameters, we were using in the proof of Proposition~7.2 the Property \ref{as:measure} of the family of measures to provide proper estimates of the energy.  The main issue in the case of multiparameter-dependent measures comes from the behaviour of the error term $e_j(\omega_1,\omega_2,\lambda)$. Namely, is it possible to follow \cite[Lemma~7.10]{PS00} and use the Property \ref{as:measure} to deduce similar estimates for the energy or higher regularity assumptions shall be made for the measures?

An application of the multiparameter case would be the natural equilibrium measure for self-conformal systems with translation parameters. Furthermore, one could study the absolute continuity of the Furstenberg measure induced by the K\"aenm\"aki measure {(that is, the natural equilibrium measure for self-affine IFS, see \cite{Ka}). For self-affine systems whose linear parts are strictly positive matrices the K\"aenm\"aki measure is a Gibbs measure which smoothly depends on the matrix elements,} see B\'ar\'any and Rams \cite{BR18} and Jurga and Morris \cite{JM}. The absolute continuity and the dimension of the Furstenberg measure  induced by the K\"aenm\"aki measure plays a central role in the calculation of the dimension of the K\"aenm\"aki measure, see \cite{BR18}. 

Another possible direction of further research is to study the absolute continuity of the SBR-measures of parametrized dynamical systems. 
Persson \cite{P} considered a class of piecewise affine hyperbolic
maps on a set $K \subset\R^2$, with one contracting and one expanding direction, which contains the class of the Belykh maps, as well as the fat baker's transformations. The Belykh map, first introduced by Belykh \cite{Bel} and later considered by Schmeling and Troubetzkoy \cite{SchTro} for a wider range of parameters, {which contains the fat baker's transformations as a special case}.

For a parametrized family of Belykh maps, to prove the absolute continuity of an SBR-measure, one needs to show that the family of conditional measures over the stable foliation are absolutely continuous almost surely. Unlike the system defined in Subsection~\ref{sec:Bern}, the SBR-measure does not have a product structure, so the conditional measures of the stable directions depend not only on the parameters but also on the foliation itself. Persson \cite{P} studied such systems, however, according to a personal communication \cite{Ppc}, the proof contains a crucial error, similar to B\'ar\'any \cite{BB15}.

{Extending our main results to the case of parabolic (and possibly infinite) iterated functions systems (as in \cite{SSU1, SSU Parabolic, MU Parabolic}) is yet another possible research direction. It seems well motivated in the context of continued fractions expansion and would allow extending the results of Section \ref{sec:rcf} to their natural generality.}

\appendix

\section{Proof of Lemma \ref{lem:d2 bounds}}\label{app:proof of d2_bounds}
For $u = (u_1, \ldots, u_n) \in \Om^*$ we have
\be\label{eq:dfdx} \frac{d}{dx} f^\lam_{u}(x) = \prod \limits_{k=1}^n \left( \frac{d}{dx}f^\lam_{u_k} \right) (f^\lam_{\sigma^k u} x),
\ee
hence
\be\label{eq:dfdx sum}
\frac{d^2}{dx^2} f^\lam_{u}(x) = \left( \frac{d}{dx} f^\lam_{u}(x) \right) \sum_{k=1}^n \frac{  \left(\frac{d^2}{dx^2} f^\lam_{u_k}\right) (f^\lam_{\sigma^k u}(x))\cdot\frac{d}{dx}f^\lam_{\sigma^k u}(x)  }{  \left(\frac{d}{dx}f^\lam_{u_k}\right) (f^\lam_{\sigma^k u}(x))  }.
\ee
Applying  \ref{as:C2} and  \ref{as:hyper} we obtain
\be\label{eq:dx2 dx ratio} \left|\frac{\frac{d^2}{dx^2} f^\lam_{u}(x)}{\frac{d}{dx} f^\lam_{u}(x)} \right| \leq \frac{M_1}{\gamma_1} \sum \limits_{k=1}^n \left|\frac{d}{dx}f^\lam_{\sigma^k u}(x) \right| \leq \frac{M_1}{\gamma_1} \sum \limits_{k=1}^n \gamma_2^{n-k} \leq \frac{M_1}{\gamma_1 (1 - \gamma_2)}.
\ee
This proves \eqref{eq:d2dx bound}. For the proof of \eqref{eq:dxdlam bound}, note first that differentiating \eqref{eq:dfdx} with respect to $\lambda$ gives
\[ \frac{d^2}{d\lam dx} f^\lam_{u}(x) = \left( \frac{d}{dx} f^\lam_{u}(x) \right) \sum_{k=1}^n \frac{  \frac{d}{d\lam}\left( \left(\frac{d}{dx}f^\lam_{u_k}\right) (f^\lam_{\sigma^k u}(x)) \right)  }{  \left(\frac{d}{dx}f^\lam_{u_k}\right) (f^\lam_{\sigma^k u}(x))  }. \]
Applying  \ref{as:hyper} as before we get
\be\label{dxdlam sum}
\left| \frac{\frac{d^2}{d\lam dx} f^\lam_{u}(x)}{\frac{d}{dx} f^\lam_{u}(x)} \right| \leq \frac{1}{\gamma_1} \sum \limits_{k=1}^n \left| \frac{d}{d\lam}\left( \left(\frac{d}{dx}f^\lam_{u_k}\right) (f^\lam_{\sigma^k u}(x)) \right) \right|.
\ee
By  \ref{as:C2} and  \ref{as:dxdlam} we have
\begin{eqnarray}
\left| \frac{d}{d\lam}\left( \left(\frac{d}{dx}f^\lam_{u_k}\right) (f^\lam_{\sigma^k u}(x)) \right) \right| & \leq & \left| \frac{d^2}{d\lam dx} f^\lam_{u_k}(f^\lam_{\sigma^k u}(x)) \right| + \left| \left( \frac{d^2}{dx^2} f^\lam_{u_k} \right) (f^\lam_{\sigma^k u} (x)) \right| \cdot \left| \left( \frac{d}{d\lam} f^\lam_{\sigma^k u} \right)(x) \right| \nonumber \\
& \leq & M_2 + M_1 |h_k(\lam)|, \label{eq:dlam bound1} 
\end{eqnarray}
where $h_k(\lam) = \frac{d}{d\lam} f^\lam_{\sigma^k u} (x)$. By  \ref{as:lam hoelder} we have $L = \sup \limits_{j \in \Ak} \sup \limits_{\lam \in U} \left\| \frac{d}{d\lam} f^\lam_j \right\|_{\infty} < \infty$. Moreover, by  \ref{as:hyper}, we have for $1 \leq k \leq n-1$
\begin{eqnarray} |h_k(\lam)| & = & \left| \frac{d}{d\lam} \left( f^\lam_{u_{k+1}} \left( f^\lam_{\sigma^{k+1} u} (x) \right) \right) \right| \nonumber \\
&= & \left| \left( \frac{d}{d\lam} f^\lam_{u_{k+1}} \right) \left( f^\lam_{\sigma^{k+1} u} (x) \right) + \left( \frac{d}{dx} f^\lam_{u_{k+1}} \right) \left( f^\lam_{\sigma^{k+1} u} (x) \right) \cdot \left(\frac{d}{d\lam}  f^\lam_{\sigma^{k+1} u} (x) \right) \right| \nonumber \\
& \leq & L + \gamma_2 |h_{k+1}(\lam)|, \label{eq:h_k step}
\end{eqnarray}
with $|h_{n}(\lam)| = \left| \frac{d}{d\lam} \mathrm{id}(x) \right|  = 0$. Therefore, iterating \eqref{eq:h_k step} yields
\be \label{eq:h_k bound}
|h_k(\lambda)| \leq L \sum \limits_{j=0}^{n-1-k} \gamma_2^j \leq \frac{L}{1 - \gamma_2}.
\ee
Combining \eqref{dxdlam sum}, \eqref{eq:dlam bound1}, \eqref {eq:h_k step} and \eqref{eq:h_k bound} gives
\[ \left| \frac{\frac{d^2}{d\lam dx} f^\lam_{u}(x)}{\frac{d}{dx} f^\lam_{u}(x)} \right| \leq \frac{(M_2 + \frac{M_1 L}{1-\gamma_2})n}{\gamma_1}. \]
This concludes the proof of Lemma \ref{lem:d2 bounds}.
\section{Some more regularity lemmas}

\begin{lem}\label{lem:dlam hoelder}
	There exists a constant $C_{71} > 0$ such that
	\[ \left|\frac{d}{d\lam}f_{u}^{\lam_1}(x) - \frac{d}{d\lam}f_{u}^{\lam_2}(x)\right| \leq C_{71} | \lam_1 - \lam_2|^\delta \]
	holds for all $\lambda_1, \lambda_2 \in U,\ x \in X,\ u \in \Om^* $.
\end{lem}

\begin{proof}
	We will prove the claim inductively with respect to $n = |u|$. More precisely, let us assume that
	\be \label{eq:hoelder induction}
	\left|\frac{d}{d\lam} f^{\lam_1}_{u}(x) - \frac{d}{d\lam} f^{\lam_2}_{u}(x)\right| \leq C_{72}\sum \limits_{k=0}^{n-1}k\gamma_2^k |\lam_1 - \lam_2|^{\delta}
	\ee
	holds for all $u \in \Ak^n,\ \lam_1,\lam_2 \in U$ and $x \in X$ with some large enough constant $C_{72}$ (its value will be specified later). We shall prove that \eqref{eq:hoelder induction} holds also for $n+1$. Fix $u = (u_1, \ldots, u_{n+1}) \in \Ak^{n+1}$ and let $v = (u_1, \ldots, u_n)$. We have
	\begin{eqnarray}
	\left| \frac{d}{d\lam} f^{\lam_1}_{u}(x) - \frac{d}{d\lam} f^{\lam_2}_{u}(x) \right| & \leq &  \left| \left( \frac{d}{d\lam} f^{\lam_1}_{v} \right) \left( f^{\lam_1}_{u_{n+1}}(x) \right) - \left( \frac{d}{d\lam} f^{\lam_2}_{v} \right) \left( f^{\lam_2}_{u_{n+1}}(x) \right) \right| + \nonumber \\
	& &\bigg| \left( \left( \frac{d}{dx} f^{\lam_1}_{v} \right) \left( f^{\lam_1}_{u_{n+1}}(x) \right) \right) \left( \frac{d}{d\lam} f^{\lam_1}_{u_{n+1}}(x) \right) - \nonumber \\
	& & \left( \left( \frac{d}{dx} f^{\lam_2}_{v} \right) \left( f^{\lam_2}_{u_{n+1}}(x) \right) \right) \left( \frac{d}{d\lam} f^{\lam_2}_{u_{n+1}}(x) \right)\bigg| \nonumber \\
	& =: & A_1 + A_2. \nonumber
	\end{eqnarray}
	Let $L = \sup \limits_{j \in \Ak}\ \sup \limits_{\lam \in U} \| \frac{d}{d\lam}f^\lam_{j}\|_{\infty}$. Assumption  \ref{as:lam hoelder} implies that $L$ is finite. By \eqref{eq:hoelder induction},  \ref{as:lam hoelder}, \ref{as:dxdlam},  \ref{as:hyper} and \eqref{eq:dxdlam bound} we obtain
	\begin{eqnarray}
	A_1 & \leq & \left| \left( \frac{d}{d\lam} f^{\lam_1}_{v} \right) \left( f^{\lam_1}_{u_{n+1}}(x) \right) - \left( \frac{d}{d\lam} f^{\lam_2}_{v} \right) \left( f^{\lam_1}_{u_{n+1}}(x) \right) \right| + \nonumber \\
	& & \left| \left( \frac{d}{d\lam} f^{\lam_2}_{v} \right) \left( f^{\lam_1}_{u_{n+1}}(x) \right) - \left( \frac{d}{d\lam} f^{\lam_2}_{v} \right) \left( f^{\lam_2}_{u_{n+1}}(x) \right) \right| \nonumber \\
	& \leq & C_{72} \sum \limits_{k=0}^{n-1}k \gamma_2^k|\lam_1 - \lam_2|^{\delta} + \left\|\frac{d^2}{dxd\lam} f^{\lam_2}_{v}\right\|_{\infty}|f^{\lam_1}_{u_{n+1}}(x) - f^{\lam_2}_{u_{n+1}}(x)| \nonumber \\
	& \leq &  C_{72} \sum \limits_{k=0}^{n-1}k \gamma_2^k|\lam_1 - \lam_2|^{\delta} + LC_{52}n\left\|\frac{d}{dx} f^{\lam_2}_{v}\right\|_{\infty}|\lam_1 - \lam_2| \nonumber \\
	& \leq & C_{72} \sum \limits_{k=0}^{n-1}k \gamma_2^k|\lam_1 - \lam_2|^{\delta} + LC_{52}n\gamma_2^n|\lam_1 - \lam_2|. \label{eq:fu hoelder A1}
	\end{eqnarray}
	Therefore, application of  \ref{as:lam hoelder} and  \ref{as:hyper} gives
	\begin{eqnarray}
	A_2 & \leq & \left| \left( \frac{d}{dx} f^{\lam_2}_{v} \right) \left( f^{\lam_2}_{u_{n+1}}(x) \right) \right| \cdot \left| \frac{d}{d\lam} f^{\lam_1}_{u_{n+1}}(x) - \frac{d}{d\lam} f^{\lam_2}_{u_{n+1}}(x)  \right| \nonumber \nonumber + \\
	& & \left| \frac{d}{d\lam} f^{\lam_1}_{u_{n+1}}(x) \right| \cdot \left|  \left( \frac{d}{dx} f^{\lam_1}_{v} \right) \left( f^{\lam_1}_{u_{n+1}}(x) \right)  - \left( \frac{d}{dx} f^{\lam_2}_{v} \right) \left( f^{\lam_2}_{u_{n+1}}(x) \right)  \right|  \nonumber \\
	& \leq & \gamma_2^nC_3|\lam_1 - \lam_2|^\delta + L \left|  \left( \frac{d}{dx} f^{\lam_1}_{v} \right) \left( f^{\lam_1}_{u_{n+1}}(x) \right)  - \left( \frac{d}{dx} f^{\lam_2}_{v} \right) \left( f^{\lam_2}_{u_{n+1}}(x) \right)  \right| \nonumber \\
	& =: & \gamma_2^nC_3|\lam_1 - \lam_2|^\delta + L A_3  \label{eq:fu hoelder A2}
	\end{eqnarray}
	Furthermore, by  Lemma \ref{lem:d2 bounds}, \ref{as:lam hoelder} and \ref{as:hyper}
	\begin{eqnarray}
	A_3 & \leq & \left|  \left( \frac{d}{dx} f^{\lam_1}_{v} \right) \left( f^{\lam_1}_{u_{n+1}}(x) \right)  - \left( \frac{d}{dx} f^{\lam_1}_{v} \right) \left( f^{\lam_2}_{u_{n+1}}(x) \right)  \right| + \nonumber\\
	& & \left|  \left( \frac{d}{dx} f^{\lam_1}_{v} \right) \left( f^{\lam_2}_{u_{n+1}}(x) \right)  - \left( \frac{d}{dx} f^{\lam_2}_{v} \right) \left( f^{\lam_2}_{u_{n+1}}(x) \right)  \right|\nonumber \\
	& \leq & \left\| \frac{d^2}{dx^2} f^{\lam_1}_{v} \right\|_{\infty} |f^{\lam_1}_{u_{n+1}}(x) - f^{\lam_2}_{u_{n+1}}(x)| + \sup \limits_{\lam \in U}\left\|\left( \frac{d^2}{d\lam dx} f^{\lam}_{v} \right) \left( f^{\lam_2}_{u_{n+1}}(x) \right) \right\|_{\infty} |\lam_1 - \lam_2| \nonumber \\
	& \leq & C_{51} \left\| \frac{d}{dx} f^{\lam_1}_{v} \right\|_{\infty} L|\lam_1 - \lam_2| + C_{52}n \sup \limits_{\lam \in U} \left\| \frac{d}{dx} f^{\lam}_{v} \right\|_{\infty} |\lam_1 - \lam_2|\nonumber \\
	& \leq & \left( LC_{51} + C_{52}n\right)\gamma_2^n |\lam_1 - \lam_2|. \nonumber
	\end{eqnarray}
	Combining the above inequality with \eqref{eq:fu hoelder A1} and \eqref{eq:fu hoelder A2} yields
	\[ \left| \frac{d}{d\lam} f^{\lam_1}_{u}(x) - \frac{d}{d\lam} f^{\lam_1}_{u}(x) \right| \leq C_{72}\sum \limits_{k=0}^{n}k\gamma_2^k |\lam_1 - \lam_2|^{\delta},  \]
	provided $C_{72}$ is large enough. As \eqref{eq:hoelder induction} holds for $n=1$ by  \ref{as:lam hoelder}, this concludes the inductive proof of \eqref{eq:hoelder induction} for $n \geq 1$. As $\sum \limits_{k=0}^{\infty}k\gamma_2^k < \infty$, the proof of the lemma is completed.
\end{proof}

\begin{lem}
	There exist constants $C_{75} > 0, C_{76}>0$ such that
	\be\label{eq:dx2 hoelder}
	\left|\frac{d^2}{dx^2}f_{u}^{\lam_1}(x) - \frac{d^2}{dx^2}f_{u}^{\lam_2}(x)\right| \leq C_{75} |u| | \lam_1 - \lam_2|^\delta \sup \limits_{\lam \in [\lam_1, \lam_2]} \left\| \frac{d}{dx} f^\lam_{u} \right\|_{\infty}
	\ee
	and
	\be\label{eq:dxdlam hoelder}
	\left|\frac{d^2}{d\lam dx}f_{u}^{\lam_1}(x) - \frac{d^2}{d\lam dx}f_{u}^{\lam_2}(x)\right| \leq C_{76} |u|^2 | \lam_1 - \lam_2|^\delta \sup \limits_{\lam \in [\lam_1, \lam_2]} \left\| \frac{d}{dx} f^\lam_{u} \right\|_{\infty}
	\ee
	hold for all $\lambda_1, \lambda_2 \in U,\ x \in X,\ u \in \Om^* $.
\end{lem}

\begin{proof}
	We shall prove \eqref{eq:dx2 hoelder}. The proof of \eqref{eq:dxdlam hoelder} is similar and we omit it. Let $n = |u|$. By \eqref{eq:dfdx sum} we have
	\begin{eqnarray} 
	\bigg| \frac{d^2}{dx^2}f_{u}^{\lam_1}(x) & - & \frac{d^2}{dx^2}f_{u}^{\lam_2}(x)\bigg|  \leq  \left| \frac{d}{dx} f^{\lam_1}_{u}(x) - \frac{d}{dx} f^{\lam_2}_{u}(x)  \right| \cdot \sum_{k=1}^n \left|  \frac{  \left(\frac{d^2}{dx^2} f^{\lam_1}_{u_k}\right) (f^{\lam_1}_{\sigma^k u}(x))\cdot\frac{d}{dx}f^{\lam_1}_{\sigma^k u}(x)  }{  \left(\frac{d}{dx}f^{\lam_1}_{u_k}\right) (f^{\lam_1}_{\sigma^k u}(x)) }  \right| + \nonumber \\
	& & \left| \frac{d}{dx}f_{u}^{\lam_2}(x)\right| \cdot \sum_{k=1}^n \left|  \frac{  \left(\frac{d^2}{dx^2} f^{\lam_1}_{u_k}\right) (f^{\lam_1}_{\sigma^k u}(x))\cdot\frac{d}{dx}f^{\lam_1}_{\sigma^k u}(x)  }{  \left(\frac{d}{dx}f^{\lam_1}_{u_k}\right) (f^{\lam_1}_{\sigma^k u}(x)) } -  \frac{  \left(\frac{d^2}{dx^2} f^{\lam_2}_{u_k}\right) (f^{\lam_2}_{\sigma^k u}(x))\cdot\frac{d}{dx}f^{\lam_2}_{\sigma^k u}(x)  }{  \left(\frac{d}{dx}f^{\lam_2}_{u_k}\right) (f^{\lam_2}_{\sigma^k u}(x)) }   \right|  \nonumber \\
	& =: & A_1 \cdot A_2 + \left| \frac{d}{dx}f_{u}^{\lam_2}(x)\right| \cdot \sum \limits_{k=1}^n h_k(x). \label{eq:d2dx A1-3}
	\end{eqnarray}
	We will bound now the above terms. First, by \eqref{eq:dxdlam bound} and the mean value theorem,  we have
	\[ A_1 \leq \left| \frac{d^2}{d\lam dx} f^{\xi}_{u}(x)  \right|\left| \lam_1 - \lam_2 \right| \leq C_{52} |u|  \left| \lam_1 - \lam_2 \right| \sup_{\lam \in [\lam_1, \lam_2]} \left|\frac{d}{dx} f^\lam_{u} (x) \right|,   \]
	where $\xi \in U$ is a point lying between $\lam_1$ and $\lam_2$. By \eqref{eq:dx2 dx ratio} (recall \eqref{eq:dfdx sum})
	\[ A_2 \leq \frac{M_1}{\gamma_1 (1 - \gamma_2)}. \]
	Reducing the expression defining $h_k(x)$ to a common denominator and applying  \ref{as:hyper} gives
	
	\begin{eqnarray}
	h_k(x) & \leq & \frac{1}{\gamma_1^2}  \bigg|  \left(\frac{d}{dx}f^{\lam_2}_{u_k}\right) (f^{\lam_2}_{\sigma^k u}(x)) \cdot\left(\frac{d^2}{dx^2} f^{\lam_1}_{u_k}\right) (f^{\lam_1}_{\sigma^k u}(x))\cdot\frac{d}{dx}f^{\lam_1}_{\sigma^k u}(x) \nonumber \\
	&  - & \left(\frac{d}{dx}f^{\lam_1}_{u_k}\right) (f^{\lam_1}_{\sigma^k u}(x)) \cdot \left(\frac{d^2}{dx^2} f^{\lam_2}_{u_k}\right) (f^{\lam_2}_{\sigma^k u}(x))\cdot\frac{d}{dx}f^{\lam_2}_{\sigma^k u}(x)  \bigg| \nonumber \\
	& \leq & \frac{1}{\gamma_1^2} \Bigg( \bigg| \left(\frac{d}{dx}f^{\lam_2}_{u_k}\right) (f^{\lam_2}_{\sigma^k u}(x)) - \left(\frac{d}{dx}f^{\lam_1}_{u_k}\right) (f^{\lam_1}_{\sigma^k u}(x)) \bigg| \cdot\bigg| \left(\frac{d^2}{dx^2} f^{\lam_1}_{u_k}\right) (f^{\lam_1}_{\sigma^k u}(x))\cdot\frac{d}{dx}f^{\lam_1}_{\sigma^k u}(x)\bigg|\ +  \nonumber\\
	& \bigg| & \left(\frac{d}{dx}f^{\lam_1}_{u_k}\right) (f^{\lam_1}_{\sigma^k u}(x)) \bigg| \cdot \bigg| \left(\frac{d^2}{dx^2} f^{\lam_1}_{u_k}\right) (f^{\lam_1}_{\sigma^k u}(x))\cdot\frac{d}{dx}f^{\lam_1}_{\sigma^k u}(x) - \left(\frac{d^2}{dx^2} f^{\lam_2}_{u_k}\right) (f^{\lam_2}_{\sigma^k u}(x))\cdot\frac{d}{dx}f^{\lam_2}_{\sigma^k u}(x) \bigg| \Bigg) \nonumber\\
	& =: & \frac{1}{\gamma^2} \left( A_3 \cdot A_4 + A_5 \cdot A_6 \right). \nonumber
	\end{eqnarray}
	By  \ref{as:C2},  \ref{as:dxdlam}, \eqref{eq:d2dx bound} we have
	\begin{eqnarray} A_3 & \leq & \bigg| \left(\frac{d}{dx}f^{\lam_2}_{u_k}\right) (f^{\lam_2}_{\sigma^k u}(x)) - \left(\frac{d}{dx}f^{\lam_1}_{u_k}\right) (f^{\lam_2}_{\sigma^k u}(x))\bigg| + \bigg| \left(\frac{d}{dx}f^{\lam_1}_{u_k}\right) (f^{\lam_2}_{\sigma^k u}(x)) - \left(\frac{d}{dx}f^{\lam_1}_{u_k}\right) (f^{\lam_1}_{\sigma^k u}(x)) \bigg| \nonumber \\
	& \leq &  \left| \lam_1 - \lam_2 \right| \sup_{\lam \in [\lam_1, \lam_2]} \left| \left( \frac{d^2}{d\lam dx} f^\lam_{u_k} \right) (f^{\lam_2}_{\sigma^k u}(x)) \right| + \left\| \frac{d^2}{dx^2} f^{\lam_1}_{u_k} \right\|_{\infty} | f^{\lam_2}_{\sigma^k u}(x) - f^{\lam_1}_{\sigma^k u}(x)|  \nonumber \\
	&  \leq & M_2|\lambda_1 - \lambda_2| + M_1 |\lam_1 - \lam_2| \sup \limits_{\lam \in [\lam_1, \lam_2]} \left| \frac{d}{d\lam} f^{\lam}_{\sigma^k u}(x) \right| \leq M_{11}|\lambda_1 - \lambda_2|,  \nonumber
	\end{eqnarray}
	for some constant $M_{11} > 0$, as $\sup \limits_{\lam \in U} \left| \frac{d}{d\lam} f^{\lam}_{\sigma^k u}(x) \right|$ is bounded uniformly in $u \in \Om^*, 1 \leq k \leq n$ and $x \in X$ by Lemma \ref{lem:dlam hoelder}. Assumptions  \ref{as:C2} and  \ref{as:hyper} imply
	\[ A_4 \leq M_1 \gamma_2^{n-k} \text{ and } A_5 \leq \gamma_2.\]
	Applying  \ref{as:C2},  \ref{as:hyper}, \eqref{eq:dxdlam bound}, Lemma \ref{lem:dlam hoelder} gives
	\begin{eqnarray}
	A_6 & \leq & \bigg| \left(\frac{d^2}{dx^2} f^{\lam_1}_{u_k}\right) (f^{\lam_1}_{\sigma^k u}(x))\bigg| \cdot \bigg| \frac{d}{dx}f^{\lam_1}_{\sigma^k u}(x) - \frac{d}{dx}f^{\lam_2}_{\sigma^k u}(x) \bigg| + \nonumber \\
	& & \bigg| \left(\frac{d^2}{dx^2} f^{\lam_1}_{u_k}\right) (f^{\lam_1}_{\sigma^k u}(x)) - \left(\frac{d^2}{dx^2} f^{\lam_2}_{u_k}\right) (f^{\lam_2}_{\sigma^k u}(x)) \bigg| \cdot \bigg|\frac{d}{dx}f^{\lam_2}_{\sigma^k u}(x) \bigg| \nonumber \\
	& \leq & M_1  |\lam_1 - \lam_2|  \sup \limits_{\lam \in [\lam_1, \lam_2]} \left| \frac{d^2}{dx d\lam} f^{\lam}_{\sigma^k u}(x) \right| + \nonumber\\
	& & \gamma_2^{n-k} \bigg| \left(\frac{d^2}{dx^2} f^{\lam_1}_{u_k}\right) (f^{\lam_1}_{\sigma^k u}(x)) - \left(\frac{d^2}{dx^2} f^{\lam_2}_{u_k}\right) (f^{\lam_2}_{\sigma^k u}(x)) \bigg| \nonumber \\
	& \leq & M_1 |n-k|\gamma_2^{n-k} |\lam_1 - \lam_2|  + \gamma_2^{n-k}A_7  \nonumber
	\end{eqnarray}
	and again by  \ref{as:C2} and Lemma \ref{lem:dlam hoelder}
	\begin{eqnarray}
	A_7 & \leq & \bigg| \left(\frac{d^2}{dx^2} f^{\lam_1}_{u_k}\right) (f^{\lam_1}_{\sigma^k u}(x)) - \left(\frac{d^2}{dx^2} f^{\lam_1}_{u_k}\right) (f^{\lam_2}_{\sigma^k u}(x)) \bigg| + \nonumber \\
	& & \bigg| \left(\frac{d^2}{dx^2} f^{\lam_1}_{u_k}\right) (f^{\lam_2}_{\sigma^k u}(x)) - \left(\frac{d^2}{dx^2} f^{\lam_2}_{u_k}\right) (f^{\lam_2}_{\sigma^k u}(x)) \bigg| \nonumber \\
	& \leq & C_1 |f^{\lam_1}_{\sigma^k u}(x) - f^{\lam_2}_{\sigma^k u}(x)|^\delta + C_2|\lambda_1 - \lambda_2|^\delta \nonumber \\
	& \leq & C_1 |\lam_1 - \lam_2|^{\delta} \sup \limits_{\lam \in [\lam_1, \lam_2]} \left| \frac{d}{d\lam} f^\lam_{\sigma^k u}(x)\right|^{\delta} + C_2|\lambda_1 - \lambda_2|^\delta \leq M_{12} |\lam_1 - \lam_2|^{\delta}. \nonumber
	\end{eqnarray}
	Combining the above with \eqref{eq:d2dx A1-3}, bound on $h_k$ and estimates on $A_1, \ldots, A_7$ and recalling that $\sum \limits_{k=1}^{n}|n-k|\gamma_2^{n-k} \leq \sum \limits_{k=0}^{\infty} k\gamma_2^k < \infty$ finishes the proof of \eqref{eq:dx2 hoelder}.
\end{proof}

\section{Proof of Proposition \ref{prop:C_1_delta_pi}}\label{app:proof of C_1_delta_pi}

We will write $d(u,v)$ for $d_{\lam_0}(u,v)$. Let $n = |u\wedge v|$, so  that $u\wedge v = u_1\ldots u_n$. Let us begin by proving \eqref{eq:4.21.1}. We have 
\begin{eqnarray}
\textstyle{\frac{d}{d\lam}} (\Pi^\lam(u) - \Pi^\lam(v))   & = &  \textstyle{\frac{d}{d\lam}} \left[f_{u \wedge v}^\lam (\Pi^\lam(\sig^n u)) - f_{u \wedge v}^\lam (\Pi^\lam(\sig^n v)) \right]   \nonumber \\[1.2ex]
& = &   \left(\textstyle{\frac{d}{d\lam}} f_{u \wedge v}^\lam\right)(\Pi^\lam (\sig^n u)) - \left(\textstyle{\frac{d}{d\lam}} f_{u \wedge v}^\lam\right)(\Pi^\lam (\sig^n v))   + \nonumber \\[1.2ex]
& &   \left(\textstyle{\frac{d}{dx}} f_{u \wedge v}^\lam\right)(\Pi^\lam (\sig^n u))\cdot \textstyle{\frac{d}{d\lam}} \Pi^\lam(\sig^n u) - \left(\textstyle{\frac{d}{dx}} f_{u \wedge v}^\lam\right)(\Pi^\lam (\sig^n v)) \cdot \textstyle{\frac{d}{d\lam}} \Pi^\lam(\sig^n v)   \nonumber \\[1.2ex]
& = &   \left(\textstyle{\frac{d}{d\lam}} f_{u \wedge v}^\lam\right)(\Pi^\lam (\sig^n u)) - \left(\textstyle{\frac{d}{d\lam}} f_{u \wedge v}^\lam\right)(\Pi^\lam (\sig^n v))   + \nonumber \\[1.2ex]
& &   \left(\textstyle{\frac{d}{dx}} f_{u \wedge v}^\lam\right)(\Pi^\lam (\sig^n u))\cdot \left[ \textstyle{\frac{d}{d\lam}} \left(\Pi^\lam(\sig^n u) - \Pi^\lam(\sig^n v)\right)\right]   + \nonumber \\[1.2ex]
& &   \left[\left(\textstyle{\frac{d}{dx}} f_{u \wedge v}^\lam\right)(\Pi^\lam (\sig^n u)) - \left(\textstyle{\frac{d}{dx}}f_{u \wedge v}^\lam\right)(\Pi^\lam (\sig^n v))\right]   \cdot   \textstyle{\frac{d}{d\lam}} \Pi^\lam(\sig^n v)   \nonumber \\[1.2ex]
& =: & A_1 + A_2 + A_3. \label{est10}
\end{eqnarray}

Application of \eqref{eq:dxdlam bound}, Lemma \ref{lem:dx dist comparision} and  \ref{as:hyper}  yields
\begin{eqnarray} |A_1| & \leq & \left\| \frac{d^2}{dxd\lam}f^{\lam}_{u \wedge v}\right\|_{\infty} |\Pi^\lam (\sig^n u) - \Pi^\lam (\sig^n v)| \leq C_{52} n \left\| \frac{d}{dx}f^{\lam}_{u \wedge v}\right\|_{\infty} \leq \frac{C_{52}}{c_1} n d(u,v)^{1 - \beta/4} \nonumber \\
& \leq & \frac{C_{52}}{c_1} n \gamma_2^{3n\beta/4} d(u,v)^{1 - \beta} \leq \frac{C_{\beta,1}}{3}d(u,v)^{1-\beta} \nonumber,
\end{eqnarray}
provided $C_{\beta,1}$ is chosen large enough. Using Lemma \ref{lem:dx dist comparision} together with the fact that $\frac{d}{d\lam} \Pi^{\lam}$ is bounded on $U \times \Om$ (following from Proposition \ref{prop:pi hoelder}), one obtains
\[ |A_2| \leq \frac{C_{\beta,1}}{3} d(u,v)^{1 - \beta/4} \leq \frac{C_{\beta,1}}{3} d(u,v)^{1 - \beta}, \]
if $C_{\beta,1}$ is large enough. Boundedness of $\frac{d}{d\lam} \Pi^{\lam}$, \eqref{eq:d2dx bound} and Lemma \ref{lem:dx dist comparision} imply

\begin{eqnarray} |A_3| & \leq & \left\| \frac{d^2}{dx^2}f^{\lam}_{u \wedge v}\right\|_{\infty} |\Pi^\lam (\sig^n u) - \Pi^\lam (\sig^n v)| \left| \textstyle{\frac{d}{d\lam}} \Pi^\lam(\sig^n v) \right| \leq C_{51} \left\| \frac{d}{dx}f^{\lam}_{u \wedge v}\right\|_{\infty} \left| \textstyle{\frac{d}{d\lam}} \Pi^\lam(\sig^n v) \right| \nonumber \\
& \leq & \frac{C_{\beta,1}}{3}d(u,v)^{1-\beta} \nonumber,
\end{eqnarray}
once again for $C_{\beta,1}$ large enough. This finishes the proof of \eqref{eq:4.21.1}. For the proof of \eqref{eq:4.21.2}, let us write a decomposition analogous to \eqref{est10}:
\[ \frac{d}{d\lam} \left( \Pi^{\lam_1}(u) - \Pi^{\lam_1}(v) \right) - \frac{d}{d\lam} \left( \Pi^{\lam_2}(u) - \Pi^{\lam_2}(v) \right) = \left( A_1^{\lam_1} - A_1^{\lam_2} \right) + \left( A_2^{\lam_1} - A_2^{\lam_2} \right) + \left( A_3^{\lam_1} - A_3^{\lam_2} \right).\]
We have
\begin{eqnarray} |A_1^{\lam_1} - A_1^{\lam_2}| & = & \left| \int \limits_{\Pi^{\lam_1}(\sig^n v)}^{\Pi^{\lam_1}(\sig^n u)} \frac{d^2}{dx d\lam} f^{\lam_1}_{u \wedge v}(y)dy - \int \limits_{\Pi^{\lam_2}(\sig^n v)}^{\Pi^{\lam_2}(\sig^n u)} \frac{d^2}{dx d\lam} f^{\lam_2}_{u \wedge v}(y)dy \right| \nonumber \\
& \leq & \int \limits_{S} \left| \frac{d^2}{dx d\lam} f^{\lam_1}_{u \wedge v}(y) -  \frac{d^2}{dx d\lam} f^{\lam_2}_{u \wedge v}(y) \right|dy + \int \limits_{S_1} \left| \frac{d^2}{dx d\lam} f^{\lam_1}_{u \wedge v}(y) \right|dy + \int \limits_{S_2} \left| \frac{d^2}{dx d\lam} f^{\lam_2}_{u \wedge v}(y) \right|dy, \label{eq:A1 lam}
\end{eqnarray}
where
\begin{eqnarray}S &=& [\Pi^{\lam_1}(\sig^n u), \Pi^{\lam_1}(\sig^n v)] \cap [\Pi^{\lam_2}(\sig^n u), \Pi^{\lam_2}(\sig^n v)],\nonumber \\
S_1 & = & [\Pi^{\lam_1}(\sig^n u), \Pi^{\lam_1}(\sig^n v)] \setminus [\Pi^{\lam_2}(\sig^n u), \Pi^{\lam_2}(\sig^n v)], \nonumber \\
S_2 & = & [\Pi^{\lam_2}(\sig^n u), \Pi^{\lam_2}(\sig^n v)] \setminus [\Pi^{\lam_2}(\sig^n u), \Pi^{\lam_2}(\sig^n v)]. \nonumber
\end{eqnarray}
Set $L = \sup \limits_{\lam \in U} \sup \limits_{u \in \Om} \left| \frac{d}{d\lam} \Pi^\lam(u) \right|$. We have then $|\Pi^{\lam_1}(\sig^n u) - \Pi^{\lam_2}(\sig^n u)| \leq L|\lam_1 - \lam_2|$ and $|\Pi^{\lam_1}(\sig^n v) - \Pi^{\lam_2}(\sig^n v)| \leq L|\lam_1 - \lam_2|$, hence
\be\label{eq:S12 length}
|S_1|,|S_2| \leq 2L|\lam_1 - \lam_2|.
\ee Applying this together with \eqref{eq:dxdlam hoelder} and \eqref{eq:dxdlam bound} to \eqref{eq:A1 lam}, followed by Lemma \ref{lem:dx dist comparision} and  \ref{as:hyper} as before,  yields
\begin{eqnarray}
|A_1^{\lam_1} - A_1^{\lam_2}| & \leq & \left( C_{76} n^2 | \lam_1 - \lam_2|^\delta  + 4LC_{52} n |\lam_1 - \lam_2| \right) \sup \limits_{\lam \in [\lam_1, \lam_2]} \left\| \frac{d}{dx} f^\lam_{u \wedge v} \right\|_{\infty} \nonumber \\
&\leq  &\frac{C_{\beta, 1, \delta}}{3}|\lambda_1 - \lambda_2|^\delta d(u,v)^{1 - \beta} \nonumber
\end{eqnarray}
if $C_{\beta, 1, \delta}$ is large enough. Furthermore, applying Proposition \ref{prop:pi hoelder}, \eqref{eq:d2dx bound}, \eqref{eq:dxdlam bound}, Lemma \ref{lem:dx dist comparision} and  \ref{as:hyper}, we obtain
\begin{eqnarray}
|A_2^{\lam_1} - A_2^{\lam_2}| & \leq & \left| \left(\textstyle{\frac{d}{dx}} f_{u \wedge v}^{\lam_1}\right)(\Pi^{\lam_1} (\sig^n u)) - \left(\textstyle{\frac{d}{dx}} f_{u \wedge v}^{\lam_2}\right)(\Pi^{\lam_2} (\sig^n u)) \right|\cdot\left| \textstyle{\frac{d}{d\lam}} \left(\Pi^{\lam_1}(\sig^n u) - \Pi^{\lam_1}(\sig^n v)\right) \right| +\nonumber \\
& & \left| \left(\textstyle{\frac{d}{dx}} f_{u \wedge v}^{\lam_2}\right)(\Pi^{\lam_2} (\sig^n u)) \right|\cdot\left| \textstyle{\frac{d}{d\lam}} \left(\Pi^{\lam_1}(\sig^n u) - \Pi^{\lam_1}(\sig^n v)\right) - \textstyle{\frac{d}{d\lam}} \left(\Pi^{\lam_2}(\sig^n u) - \Pi^{\lam_2}(\sig^n v)\right) \right| \nonumber \\
& \leq & 2L \left| \left(\textstyle{\frac{d}{dx}} f_{u \wedge v}^{\lam_1}\right)(\Pi^{\lam_1} (\sig^n u)) - \left(\textstyle{\frac{d}{dx}} f_{u \wedge v}^{\lam_2}\right)(\Pi^{\lam_1} (\sig^n u)) \right| + \nonumber \\
& & 2L \left| \left(\textstyle{\frac{d}{dx}} f_{u \wedge v}^{\lam_2}\right)(\Pi^{\lam_1} (\sig^n u)) - \left(\textstyle{\frac{d}{dx}} f_{u \wedge v}^{\lam_2}\right)(\Pi^{\lam_2} (\sig^n u)) \right| + \nonumber \\
& & \sup \limits_{\lam \in [\lam_1, \lam_2]} \left\| \frac{d}{dx} f^\lam_{u \wedge v} \right\|_{\infty} \left( \left| \textstyle{\frac{d}{d\lam}} \left(\Pi^{\lam_1}(\sig^n u) - \Pi^{\lam_2}(\sig^n u)\right)\right| + \left| \textstyle{\frac{d}{d\lam}} \left(\Pi^{\lam_1}(\sig^n v) - \Pi^{\lam_2}(\sig^n v)\right) \right| \right) \nonumber \\
& \leq & 2L \left( \sup \limits_{\lam \in [\lam_1, \lam_2]}  \left\|\frac{d^2}{d\lambda dx} f^{\lambda}_{u \wedge v}\right\|_{\infty}|\lam_1 - \lam_2| + \sup \limits_{\lam \in [\lam_1, \lam_2]}  \left\|\frac{d^2}{dx^2} f^{\lambda}_{u \wedge v}\right\|_{\infty}|\Pi^{\lam_1}(\sig^n u) - \Pi^{\lam_2}(\sig^n u)| \right) + \nonumber \\
& & 2\sup \limits_{\lam \in [\lam_1, \lam_2]} \left\| \frac{d}{dx} f^\lam_{u \wedge v} \right\|_{\infty}C_{\delta}|\lam_1 - \lam_2|^\delta  \nonumber \\
& \leq & 2L \sup \limits_{\lam \in [\lam_1, \lam_2]} \left\| \frac{d}{dx} f^\lam_{u \wedge v} \right\|_{\infty} \left( C_{52}n |\lam_1 - \lam_2| + C_{51}L |\lam_1 - \lam_2| \right) \nonumber + \\
& & 2\sup \limits_{\lam \in [\lam_1, \lam_2]} \left\| \frac{d}{dx} f^\lam_{u \wedge v} \right\|_{\infty}C_{\delta}|\lam_1 - \lam_2|^\delta \leq \frac{C_{\beta, 1, \delta}}{3}|\lambda_1 - \lambda_2|^\delta d(u,v)^{1 - \beta} \nonumber
\end{eqnarray}
for $C_{\beta, 1, \delta}$ large enough. By \eqref{eq:dxdlam bound} and Proposition \ref{prop:pi hoelder}, we have
\begin{eqnarray}
|A_3^{\lam_1} - A_3^{\lam_2}| & \leq &  \left| \int \limits_{\Pi^{\lam_1}(\sig^n v)}^{\Pi^{\lam_1}(\sig^n u)} \frac{d^2}{dx^2} f^{\lam_1}_{u \wedge v}(y)dy - \int \limits_{\Pi^{\lam_2}(\sig^n v)}^{\Pi^{\lam_2}(\sig^n u)} \frac{d^2}{dx^2} f^{\lam_2}_{u \wedge v}(y)dy \right| \cdot \left| \textstyle{\frac{d}{d\lam}} \Pi^{\lam_1}(\sig^n v) \right|+  \nonumber\\
& & \left(\int \limits_{\Pi^{\lam_2}(\sig^n v)}^{\Pi^{\lam_2}(\sig^n u)} \left| \frac{d^2}{dx^2} f^{\lam_2}_{u \wedge v}(y) \right| dy\right) \cdot \left|\textstyle{\frac{d}{d\lam}} \Pi^{\lam_1}(\sig^n v) - \textstyle{\frac{d}{d\lam}} \Pi^{\lam_2}(\sig^n v)\right| \nonumber \\
& \leq & L \left| \int \limits_{\Pi^{\lam_1}(\sig^n v)}^{\Pi^{\lam_1}(\sig^n u)} \frac{d^2}{dx^2} f^{\lam_1}_{u \wedge v}(y)dy - \int \limits_{\Pi^{\lam_2}(\sig^n v)}^{\Pi^{\lam_2}(\sig^n u)} \frac{d^2}{dx^2} f^{\lam_2}_{u \wedge v}(y)dy \right| + \label{eq:A3 lam} \\
& & C_{52} C_{\delta} n |\lam_1 - \lam_2|^\delta \sup \limits_{\lam \in [\lam_1, \lam_2]} \left\| \frac{d}{dx} f^\lam_{u \wedge v} \right\|_{\infty} . \nonumber
\end{eqnarray}
Let intervals $S, S_1, S_2$ be defined as before. Then by \eqref{eq:dx2 hoelder}, \eqref{eq:d2dx bound} and \eqref{eq:S12 length}
\begin{eqnarray}  & & \left| \int \limits_{\Pi^{\lam_1}(\sig^n v)}^{\Pi^{\lam_1}(\sig^n u)} \frac{d^2}{ dx^2} f^{\lam_1}_{u \wedge v}(y)dy - \int \limits_{\Pi^{\lam_2}(\sig^n v)}^{\Pi^{\lam_2}(\sig^n u)} \frac{d^2}{ dx^2} f^{\lam_2}_{u \wedge v}(y)dy \right| \nonumber \\
& \leq & \int \limits_{S} \left| \frac{d^2}{ dx^2} f^{\lam_1}_{u \wedge v}(y) -  \frac{d^2}{ dx^2} f^{\lam_2}_{u \wedge v}(y) \right|dy + \int \limits_{S_1} \left| \frac{d^2}{ dx^2} f^{\lam_1}_{u \wedge v}(y) \right|dy + \int \limits_{S_2} \left| \frac{d^2}{ dx^2} f^{\lam_2}_{u \wedge v}(y) \right|dy \nonumber \\
& \leq & C_{75} n | \lam_1 - \lam_2|^\delta \sup \limits_{\lam \in [\lam_1, \lam_2]} \left\| \frac{d}{dx} f^\lam_{u \wedge v} \right\|_{\infty} + 4LC_{51}|\lam_1 - \lam_2|\sup \limits_{\lam \in [\lam_1, \lam_2]} \left\| \frac{d}{dx} f^\lam_{u \wedge v} \right\|_{\infty} \nonumber \\
& \leq & C_{86} n | \lam_1 - \lam_2|^\delta \sup \limits_{\lam \in [\lam_1, \lam_2]} \left\| \frac{d}{dx} f^\lam_{u \wedge v} \right\|_{\infty} \nonumber
\end{eqnarray}
for some constant $C_{86} > 0$. Combining this with \eqref{eq:A3 lam} and applying Lemma \ref{lem:dx dist comparision} and  \ref{as:hyper} gives
\[ |A_3^{\lam_1} - A_3^{\lam_2}| \leq (C_{52} C_{\delta} + C_{86})n |\lam_1 - \lam_2|^\delta \sup \limits_{\lam \in [\lam_1, \lam_2]} \left\| \frac{d}{dx} f^\lam_{u \wedge v} \right\|_{\infty} \leq \frac{C_{\beta, 1, \delta}}{3}|\lambda_1 - \lambda_2|^\delta d(u,v)^{1 - \beta} \]
if $C_{\beta, 1, \delta}$ is large enough. Finally, putting together bounds on $|A_i^{\lam_1} - A_i^{\lam_2}|$ finishes the proof of \eqref{eq:4.21.2}.

\section{Drop of the pressure}\label{app:pressure drop}

Let $\Ak = \{1,\ldots,m\}$ and suppose we have an IFS $\Psi=\{f_j\}_{j \in \Ak}$ of the class $C^{1+\delta}$ on a compact interval $X\subset \R$.
We assume that
that the system $\{f_j\}_{j \in \Ak}$ is uniformly hyperbolic and contractive:
\be \label{hyper}
0 < \gam_1 \le |f_j'(x)| \le \gam_2 < 1\ \ \mbox{for all}\ j \in \Ak,\ x\in X.
\ee

Let $\Om = \Ak^\N$ and let $\sigma$ denote the left shift on $\Om$. Let $\Ak^* = \bigcup \limits_{n \geq 0} \Ak^n$ and let $|u| = n$ for $u \in \Ak^n$. For $u = (u_1, \ldots u_n) \in \Ak^*$ denote
\[ f_{u} = f_{u_1 \ldots u_n} := f_{u_1} \circ \ldots \circ f_{u_n} \]
(with $f_{u} = \mathrm{id}$ if $u$ is an empty word).

\medskip

Consider the pressure function, defined by
\be \label{eq-pressure}
P_\Ak(t) = P_\Psi(t)=\lim_{n\to \infty} n^{-1} \log \sum_{u\in \Ak^n} \|f'_u\|^t.
\ee
It is well-known that this limit exists, $t\mapsto P_\Ak(t)$ is continuous and strictly decreasing (it is also convex, but we will not need this).

\begin{lem} \label{lem-drop1}
	Suppose that $\Bk = \Ak \setminus\{m\}$. Then $P_\Bk(t) < P_\Ak(t)$ for all $t\ge 0$. (The functions of the IFS are assumed to be the same. The claim can be expressed in words by saying that if we drop one of the functions of the IFS, then the pressure drops strictly.)
\end{lem}

\thispagestyle{empty}

\begin{proof}
	For $t=0$ the claim is trivial, so let us fix $t>0$. Observe that the pressure can be expressed in the following alternative way:
	\be \label{alter}
	P_\Ak(t) = \lim_{n\to \infty} n^{-1} \log \sum_{u\in \Ak^n} \inf_{x\in X} |f'_u(x)|^t.
	\ee
	Indeed, by the Bounded Distortion Property, there exists $K>1$ such that $|f_u'(x)| \le K |f_u'(y)|$ for all $u\in \Ak^*$ and $x,y\in X$, and \eqref{alter} follows. Denote
	$$
	Z_n(\Ak,t) = \sum_{u\in \Ak^n} \inf_{x\in X} |f'_u(x)|^t.
	$$
	We claim that
	\be \label{claim2}
	Z_n(\Ak,t) \ge Z_n(\Bk,t)\cdot (1+\delta_t)^n,\ \ \mbox{where}\ \ \delta_t = \frac{\gam_1^t}{(m-1)\gam_2^t}\,.
	\ee
	This will immediately imply that $P_\Bk(t) < P_\Ak(t)$, as desired.
	We have 
	$$
	Z_1(\Ak,t) = Z_1(\Bk,t) + \inf_{x\in X} |f'_m(x)|^t \ge Z_1(\Bk,t)\cdot (1 + \delta_t),
	$$
	by \ref{as:hyper}. 
	Since $\inf_{x\in X} |f'_{ju}(x)|^t \ge \inf_{x\in X} |f'_j(x)|^t\cdot \inf_{x\in X} |f'_u(x)|^t$, we have
	$$
	Z_{n+1}(\Ak,t) \ge Z_1(\Ak,t)\cdot Z_n(\Ak,t),
	$$
	and \eqref{claim2} follows by induction.
\end{proof}
\textbf{Consequences.} Under the assumptions and notation of Section \ref{sec:ac equilib}, let $s(\Psi)$ be the unique zero of the pressure function $P_\Psi(t)$:
$$
P_\Psi(s(\Psi))=0.
$$

\begin{cor}\label{cor-drop1}
	Suppose that $\Phi$ is a proper subset of $\Psi$. Then $s(\Psi) > s(\Phi)$.
\end{cor}

This is immediate from Lemma~\ref{lem-drop1}.

\begin{cor}\label{cor-drop2}
	Suppose that the attractor of $\Psi$ is the entire interval $X$ and the IFS is overlapping in the sense that
	\be \label{overlap}
	\sum_{j\in \Ak} |X_j| > |X|,\ \ \mbox{where}\ X_j = f_j(X).
	\ee
	Then $s(\Psi)>1$.
\end{cor}

\begin{proof}
	We have $X = \bigcup_{j\in \Ak} X_j$ by assumption. Then \eqref{overlap} implies that there exist $i\ne j$ in $X$ such that $X_i\cap X_j$ is a non-empty interval. We can find $k\in \N$ and $w\in \Ak^k$ such that $X_w\subset X_i\cap X_j$. It follows easily that
	$$
	\bigcup_{u\in \Ak^k\setminus \{w\}} X_u = X.
	$$
	Denote $\Psi^k = \{f_u:\ u\in \Ak^k\}$, the IFS of $k$-th iterates. It follows from the existence of the limit in \eqref{eq-pressure} that $P_{\Psi^k}(t)  = k P_\Psi(t)$, hence
	$s(\Psi^k) = s(\Psi)$. By Corollary~\ref{cor-drop1}, we have $s(\Psi^k \setminus \{f_u\}) < s(\Psi^k)$. It suffices to show that for an IFS $\Phi$ whose attractor is an interval $X$ we have $s(\Phi)\ge 1$. {But this follows from the inequality $1 = \dim_H(\Lam_\Phi) \le s(\Phi)$, where $\Lam_\Phi$ is the attractor of $\Phi$.}
\end{proof}

\end{document}